\documentclass[a4paper]{article}
\usepackage{mathrsfs}
\usepackage{bbm}
\usepackage{cite}
\usepackage{amsfonts}
\usepackage{amsmath,amssymb}
\usepackage{enumitem}
\usepackage{latexsym}
\usepackage{graphicx}
\usepackage{subfigure}
\usepackage{dcolumn}
\usepackage{bm}
\usepackage{fancyhdr}
\usepackage{mathrsfs}
\usepackage{wasysym}
\usepackage{ifpdf}
\usepackage[
  bookmarks=true,
  bookmarksopen=true,
  breaklinks=true,
  colorlinks=true,
  linkcolor=blue,anchorcolor=blue,
  citecolor=blue,filecolor=blue,
  menucolor=blue,
  urlcolor=blue]{hyperref}
\newcommand{\varep}{\varepsilon}
\def\Xint#1{\mathchoice
      {\XXint\displaystyle\textstyle{#1}}%
      {\XXint\textstyle\scriptstyle{#1}}%
      {\XXint\scriptstyle\scriptscriptstyle{#1}}%
      {\XXint\scriptscriptstyle\scriptscriptstyle{#1}}%
                   \!\int}
\def\XXint#1#2#3{{\setbox0=\hbox{$#1{#2#3}{\int}$}
         \vcenter{\hbox{$#2#3$}}\kern-.5\wd0}}

\def\dashint{\Xint-}
\def\R{\mathbb{R}}
\def\e{\varepsilon}

 \numberwithin{equation}{section}
\newtheorem{theorem}{Theorem}[section]
\newtheorem{lemma}[theorem]{Lemma}
\newtheorem{corollary}[theorem]{Corollary}

\newtheorem{remark}[theorem]{Remark}

\newenvironment{proof}[1][Proof]{\noindent\textbf{#1.} }{\hfill $\Box$}

\allowdisplaybreaks
\allowdisplaybreaks \numberwithin{equation}{section}
\makeatletter
\begin{document}

\author{Jun Geng$^1$,
Bojing Shi$^{1}$ \thanks {E-mail addresses: gengjun@lzu.edu.cn (Jun Geng), shibj@lzu.edu.cn (Bojing Shi).}\\
$^{1}$  School of Mathematics and Statistics, Lanzhou University,\\
Lanzhou, Gansu 730000, China}
\date{}
\title{\textbf{Quantitative estimates in almost periodic homogenization of parabolic systems\thanks{The first author was supported in part by the National Natural Science Foundation of China (12371096) and the Fundamental Research Funds for the Central Universities (lzujbky-2024-oy05), and the second author was supported in part by the China Postdoctoral Science Foundation (2022M721438).}}}

\maketitle
\begin{abstract}

We consider a family of second-order parabolic operators $\partial_t+\mathcal{L}_\varepsilon$ in divergence form with rapidly oscillating, time-dependent and almost-periodic coefficients. We establish uniform interior and boundary H\"older and Lipschitz estimates as well as convergence rate. The estimates of fundamental solution and Green's function are also established. In contrast to periodic case, the main difficulty is that the corrector equation
$
(\partial_s+\mathcal{L}_1)(\chi^\beta_{j})=-\mathcal{L}_1(P^\beta_j)
$ in $\mathbb{R}^{d+1}$
may not be solvable in the almost periodic setting for linear functions $P(y)$ and $\partial_t \chi_S$ may not in $B^2(\mathbb{R}^{d+1})$.
 Our results are new even in the case of time-independent coefficients.


\end{abstract}

\section{\bf Introduction}
The primary purpose of this paper is to investigate the convergence rates in $L^2$, boundary H\"older and boundary Lipschitz estimates for a family of parabolic operators $\partial_t+\mathcal{L}_\varepsilon$ with rapidly oscillating and time-dependent almost periodic coefficients subject to Dirichlet boundary data. Specifically,
let $\Omega\subset\mathbb{R}^d$ be a bounded domain and $u_\varepsilon\in L^2(0,T;H^1(\Omega))$ be the weak solution of the Dirichlet problem:
$(\partial_t+\mathcal{L}_\varepsilon)u_\varepsilon=F$ in $\Omega\times(0,T)$ and $u_\varepsilon=g$ on $\partial\Omega\times(0,T)$ as well as $u_\varepsilon=h$ on $\Omega\times\{0\}$,
where
\begin{equation*}
\mathcal{L}_\varepsilon=-\text{div}\left[A\left(\frac{x}{\varepsilon},\frac{t}{\varepsilon^2}\right)\nabla
\right]=-\frac{\partial}{\partial x_i}\left[a_{ij}^{\alpha\beta}\left(\frac{x}{\varepsilon},\frac{t}{\varepsilon^2}\right)\frac{\partial}{\partial
x_j}\right],~~\varepsilon>0.
\end{equation*}
Throughout this paper we will assume that the coefficient matrix
$A(y,s)=\left(a^{\alpha\beta}_{ij}(y,s)\right)$ with $1\leq i,j\leq d, 1\leq \alpha, \beta\leq m$ is real, bounded measurable, and satisfies the ellipticity condition
\begin{equation}\label{ellipticity}
\mu |\xi|^2\le  a^{\alpha\beta}_{ij}(y,s)\xi_i^\alpha
\xi_j^\beta\leqslant \frac{1}{\mu}|\xi|^2
\quad \text{ for any }
\xi=(\xi_i^\alpha)\in \mathbb{R}^{d\times m}, (y,s)\in \mathbb{R}^{d+1},
\end{equation}
where $\mu >0$.
We shall be interested in the quantitative homogenization of second order parabolic systems with bounded measurable coefficients that are almost-periodic in the sense of H. Bohr, which means that $A$ is the uniform limit of a sequence of trigonometric polynomials in $\mathbb{R}^{d+1}$.
 Let ${\rm Trig}(\mathbb{R}^{d+1})$ denote the set of all trigonometric polynomials.
 The closure of the set ${\rm Trig}(\mathbb{R}^{d+1})$ with respect to the $L^\infty$ norm
 is called the Bohr space of almost-periodic functions.

 A useful equivalent description of the almost-periodic functions is given as follows.
 Let $A$ be bounded and continuous in $\mathbb{R}^{d+1}$. Then $A$ is almost periodic in the sense of Bohr if and only if
\begin{equation}\label{AP}
\limsup_{R\to\infty}\sup_{Y\in \mathbb{R}^{d+1}}\inf_{\substack{Z\in \mathbb{R}^{d+1}\\|Z|\leq R}}\|A(\cdot+Y)-A(\cdot+Z)\|_{L^\infty(\mathbb{R}^{d+1})}=0,
\end{equation}
where $Y=(y,s), Z=(z, \tau)\in\mathbb{R}^{d+1}$.

Set
\begin{equation}\label{AP-1}
\rho(R):=\sup_{Y\in \mathbb{R}^{d+1}}\inf_{\substack{Z\in \mathbb{R}^{d+1}\\|Z|\leq R}}\|A(\cdot+Y)-A(\cdot+Z)\|_{L^\infty(\mathbb{R}^{d+1})}.
\end{equation}
Notice that $\rho(R)=0$ for any $R\geq L\sqrt{d+1}$ if $A$ is $L$-periodic. We say that $A$ is uniformly almost periodic if and only if $A$ is bounded, continuous and satisfies (\ref{AP}).

Let $\Omega\subset\mathbb{R}^d$ be a bounded   domain and
$0<T<\infty$.
We are interested in the initial-Dirichlet problem,
\begin{equation}\label{IDP}
\left\{\aligned
(\partial_t +\mathcal{L}_\varepsilon) u_\varepsilon  &=  F &\quad &\text { in } \Omega_T=\Omega \times (0, T),\\
u_\varepsilon & = g & \quad & \text{ on } S_T=\partial\Omega \times (0, T),\\
u_\varepsilon  &=h &\quad &\text{ on } \Omega \times \{ t=0\}.
\endaligned
\right.
\end{equation}

Under suitable conditions on $F$, $g$, $h$ and $\Omega$, it is known that the weak solution $u_\varepsilon$ of
(\ref{IDP}) converges weakly in $L^2(0, T; H^1(\Omega))$ and
strongly in $L^2(\Omega_T)$ to $u_0$.
 Furthermore, the function $u_0$ is the weak solution of
the (homogenized) initial-Dirichlet problem,
\begin{equation}\label{IDP-0}
\left\{\aligned
(\partial_t +\mathcal{L}_0) u_0 &=  F &\quad &\text { in } \Omega_T,\\
u_0 & = g & \quad & \text{ on } S_T,\\
u_0  &=h &\quad &\text{ on } \Omega \times \{ t=0\}.
\endaligned
\right.
\end{equation}
The operator $\mathcal{L}_0$ in (\ref{IDP-0}),
called the homogenized operator,
is a second-order elliptic operator with constant coefficients \cite{ZKO}.

Our first main result concerns the uniform interior H\"older estimates for $\partial_t+\mathcal{L}_\varepsilon$.
\begin{theorem}\label{Holder-1}
   Suppose that $A$ is uniformly almost periodic and satisfies (\ref{ellipticity}). Let $u_\varepsilon\in L^2\big(t_0-4r^2,t_0;H^{1}(B(x_0,2r))\big)$ be a weak solution of
\begin{equation*}
(\partial_t+\mathcal{L}_\varepsilon)u_\varepsilon=f+{\rm div}(g)~~{\rm in}~~2Q=Q_{2r}(x_0,t_0).
\end{equation*}
Then for any $0<\alpha<1$,
\begin{equation*}
\aligned
\|u_\varepsilon\|_{\Lambda^{\alpha,\alpha/2}(Q)}\leq &C r^{-\alpha}\left(\dashint_{2Q} |u_\varepsilon|^2\right)^{1/2}
\\
&+ C\sup_{\substack{ 0<\ell<r\\(y,s)\in Q}}\ell^{2-\alpha}\left(\dashint_{Q_\ell(y,s)} |f|^2\right)^{1/2}+ C\sup_{\substack{ 0<\ell<r\\(y,s)\in Q}}\ell^{1-\alpha}\left(\dashint_{Q_\ell(y,s)}|g|^2\right)^{1/2},
\endaligned
\end{equation*}
where $C$ depends at most on $d,m,\alpha,\mu$ and $A$.
\end{theorem}

The next theorem gives the uniform boundary H\"older estimates for $\partial_t+\mathcal{L}_\varepsilon$.
\begin{theorem}\label{boundary holder}
Let $\Omega$ be a bounded $C^{1,\eta}$ domain in $\mathbb{R}^d$ for some $\eta\in(0,1)$ and $0<\alpha<1$. Suppose that $A$ is uniformly almost periodic and satisfies (\ref{ellipticity}). Let $u_\varepsilon \in
L^2(0,T;H^1(\Omega))$ be a weak solution to
$(\partial_t +\mathcal{L_\varepsilon})u_\varepsilon= F +{\rm div}(f)$ in $\Omega_T$, $u_\varepsilon=g$ on $S_T$ and $u_\varepsilon=0$ on $\Omega\times\{t=0\}$. Then we have
\begin{equation*}
\|u_\varepsilon\|_{\Lambda^{\alpha,\alpha/2}(\Omega_T)}\leq C \|g\|_{X^\alpha(S_T)}+ C\sup_{\substack{(x,t)\in\Omega_T\\0<r<r_0}}r^{1-\alpha}\left(\dashint_{\Omega_{r}(x,t)}|f|^2\right)^{1/2}
+C\sup_{\substack{(x,t)\in\Omega_T\\0<r<r_0}}r^{2-\alpha}\dashint_{\Omega_{r}(x,t)}|F|,
\end{equation*}
where $r_0={\rm diam}(\Omega)$ and $C$ depends only on $d$, $m$, $\alpha$, $\Omega$ and $A$.
\end{theorem}

    We shall also establish the following convergence rates for $u_\varepsilon-u_0$.
\begin{theorem}\label{rate}
  Let $\Omega$ be a bounded $C^{1,1}$ domain in $\mathbb{R}^d$. Suppose that $A$ is uniformly almost periodic and satisfies (\ref{ellipticity}) and $F\in L^2(\Omega_T)$. Let $u_\varepsilon\in L^2(0,T;H^1(\Omega))$ and $u_0\in L^2(0,T;H^2(\Omega))$ be the weak solutions of (\ref{IDP}) and (\ref{IDP-0}), respectively. Then there exists a modulus $\eta: (0,1]\rightarrow [0,\infty)$, depending only on $A$, such that $\eta(t)\to 0$ as $t\to 0$, and there exists $\delta$, depending on $\e, A$, such that $\delta\rightarrow 0$ as $\varepsilon\rightarrow 0$,
\begin{equation*}
\|u_\varepsilon-u_0\|_{L^2(\Omega_T)}\leq C\eta(\varepsilon)\left\{\|u_0\|_{L^2(0,T;H^2(\Omega))}+\|F\|_{L^2(\Omega_T)}
+\sup_{\delta^2<t<T}
\left(\frac{1}{\delta}\int_{t-\delta^2}^t\int_{\Omega}|\nabla u_0|^2\right)^{1/2}\right\},
\end{equation*}
where $C$ depends only on $d$, $m$, $\mu$, $\Omega$ and $T$.
\end{theorem}

\begin{remark}
By the similar argument as in Remark 1.3 and Remark 3.10 in \cite{GS-2017}, we know that if $g=0$, then
\begin{align*}
\int_{t-\delta^2}^t\int_{\Omega}|\nabla u_0|^2 &\leq C\left\{\|\partial_tu_0\|_{L^2(\Omega_T)}+\|F\|_{L^2(\Omega_T)}\right\}\left(\int_{t-\delta^2}^t\int_{\Omega}| u_0|^2\right)^{1/2}
\\&\leq C\delta \left\{\|\partial_tu_0\|_{L^2(\Omega_T)}+\|F\|_{L^2(\Omega_T)}\right\}\sup_{0<t<T}\|u_0(\cdot, t)\|_{L^2(\Omega)}.
\end{align*}
Using energy estimate yields that
\begin{equation*}
\sup_{\delta^2<t<T}
\left(\frac{1}{\delta}\int_{t-\delta^2}^t\int_{\Omega}|\nabla u_0|^2\right)^{1/2}\leq C\left\{\|\partial_t u_0\|_{L^2(\Omega_T)}+\|F\|_{L^2(\Omega_T)}+\|h\|_{L^2(\Omega)}\right\}.
\end{equation*}
It follows that
\begin{equation*}
\|u_\varepsilon-u_0\|_{L^2(\Omega_T)}\leq C\eta(\varepsilon)\left\{\|u_0\|_{L^2(0,T;H^2(\Omega))}+\|F\|_{L^2(\Omega_T)}
+\|h\|_{L^2(\Omega)}\right\}.
\end{equation*}
Moreover, if $g=0$ and $h=0$, then
\begin{equation*}
\|u_\varepsilon-u_0\|_{L^2(\Omega_T)}\leq C\eta(\varepsilon)\|F\|_{L^2(\Omega_T)},
\end{equation*}
where $C$ depends only on $d$, $m$, $\mu$, $\Omega$ and $T$.
\end{remark}

The next theorem gives the large-scale interior Lipschitz estimates for $\partial_t+\mathcal{L}_\varepsilon$.
\begin{theorem}\label{thm-interior-Lip}
Suppose that $A$ satisfies \eqref{ellipticity} and is uniformly almost periodic.  Suppose also that there exist $C_0>0$ and $N\geq\frac{3q-2}{q-2}$ such that
 \begin{equation}\label{decay-cond}
 \rho(R)\leq C_0[\log R]^{-N}\quad {\rm for~~any}~~R\geq 2,
 \end{equation}
where $q>2$. Let $u_\varepsilon$ be a weak solution to $(\partial_t+\mathcal{L}_\varepsilon)u_\varepsilon=F$ in $Q_R=Q_R(x_0,t_0)$, where $R>\varepsilon$ and $F\in L^p(Q_R)$ for some $p>d+2$. Then for any $\varepsilon\leq r<R$,
\begin{equation}\label{est-Lip}
\aligned
\left(\dashint_{Q_r}|\nabla u_\varepsilon|^2\right)^{1/2}\leq C\left\{\left(\dashint_{Q_R}|\nabla u_\varepsilon|^2\right)^{1/2}+R\left(\dashint_{Q_R}|F|^p\right)^{1/p}\right\},
\endaligned
\end{equation}
where $C>0$ depends only on $d$, $\mu$ and $p$.
\end{theorem}

Our last main result concerns the large-scale boundary Lipschitz estimates for $\partial_t+\mathcal{L}_\varepsilon$.
\begin{theorem}\label{thm-bound-Lip}
Suppose that $\Omega$ is a bounded $C^{1,\alpha}$ domain in $\mathbb{R}^d$ for some $\alpha\in (0,1)$ and $A$ satisfies \eqref{ellipticity} and is uniformly almost periodic. Suppose also that $A$ satisfies \eqref{decay-cond}. Let $u_\varepsilon$ be a weak solution to
\begin{equation}\label{equ-loc}
(\partial_t+\mathcal{L}_\varepsilon)u_\varepsilon=F\quad{\rm in}~~\Omega_R=\Omega_R(x_0,t_0)\quad{\rm and}\quad u_\varepsilon=f\quad{\rm on}~~\Delta_R=\Delta_R(x_0,t_0),
\end{equation}
 where $x_0\in\partial\Omega$ and $\varepsilon<R\leq1$, $F\in L^p(\Omega_R)$ for some $p>d+2$ and $f\in C^{1+\alpha}(\Delta_R)$.
 Then for any $\varepsilon\leq r<R$,
\begin{equation*}
\aligned
\left(\dashint_{\Omega_r}|\nabla u_\varepsilon|^2\right)^{1/2}\leq C\left\{\left(\dashint_{\Omega_R}|\nabla u_\varepsilon|^2\right)^{1/2}+R\left(\dashint_{\Omega_R}|F|^p\right)^{1/p}
+R^{-1}\|f\|_{C^{1+\alpha}(\Delta_R)}\right\},
\endaligned
\end{equation*}
where $C>0$ depends only on $d$, $\mu$, $\alpha$, $p$ and $\Omega$.
\end{theorem}
\begin{remark}
The decay assumption on $\rho(R)$ in Theorems \ref{thm-interior-Lip} and \ref{thm-bound-Lip} is to ensure that we have Dini-type rates for homogenization.
\end{remark}
The convergence rates and uniform regularity estimates are the central issues in quantitative homogenization and have been studied extensively in the various setting. For solutions of $(\partial_t+\mathcal{L}_\e)u_\e=F$ with periodic coefficients, under the assumption that $A$ satisfies (\ref{ellipticity}) and $\Omega$ is $C^{1,1}$, the sharp order convergence rates of both initial-Dirichlet and initial-Neumann problems were obtained in \cite{GS-2017}. The uniform boundary H\"older and interior Lipschitz estimates as well as boundary $W^{1,p}$ estimates for any $1<p<\infty$ were established on $C^1$ domains by using the compactness argument \cite{GS-2013}, introduced to the study of homogenization problems in \cite{AL-1987,AL-1987-2,AL-1991}(see also \cite{KS-2011,KLS-2012} for Neumann data case). The uniform boundary Lipschitz estimates and the size estimates of Green function in $C^{1,\alpha}$ domains for parabolic systems with Dirichlet boundary conditions were obtained in \cite{Geng}. We also refer the reader to see related results in the recent work \cite{Suslina-2004,GS-2013,MS-2015,MS-2016,GS-2017,Geng,GS-2020-ARMA}.

\begin{remark}
For parabolic equations, in the case of single equation $(m=1)$, if the coefficients are time-dependent, the convergence rate $\|u_\varepsilon-u_0\|_{L^2(\Omega_T)}\leq C\varepsilon\|u_0\|_{L^2(\Omega_T)}$ was obtained  in \cite{bensoussan-1978} by using the maximum principle.
\end{remark}


In recent years, there has been a great amount of interest in the homogenization for operators with almost periodic coefficients. Major progress has been made for elliptic equations in the almost-periodic setting (see \cite{Kozlov-1979,PV-1979,AGK-2016,SZ-2018} and references therein). In the case that $A$ is almost periodic in the sense of H. Bohr, by introducing the so-called approximate corrector $\chi_T$, the uniform H\"older estimates and convergence rates for elliptic systems with Dirichlet boundary conditions were obtained in \cite{Shen-2014} on $C^{1,\alpha}$ domains. Consequently, the uniform boundary Lipschitz estimates for elliptic systems with H\"older continuous coefficients were established in \cite{AS-2016} for both Dirichlet and Neumann boundary conditions. The boundary $W^{1,p}$ estimates were also obtained there in $C^{1,\alpha}$ domains for $1<p<\infty$.
For elliptic systems with almost periodic coefficients in Lipschitz domains, in \cite{Geng-Shi-PAMS}, we proved that the uniform $W^{1,p}$ estimates hold for $\frac{2d}{d+1}-\delta<p<\frac{2d}{d-1}+\delta$, and the ranges are sharp for $d=2$ and $d=3$. For elliptic equations with almost periodic and complex coefficients, the interior H\"older estimate was obtained in \cite{DtR-2001} by using the compactness argument. In the case that $A$ is almost-periodic in the
sense of H. Weyl, uniform estimates for approximate correctors and convergence rates in $H^1$ and $L^2$ for elliptic systems with Dirichlet boundary conditions were obtained in \cite{SZ-2018}.

Very few quantitative results were known for parabolic equations with almost periodic coefficients, though
the qualitative homogenization theory (such as $G$-convergence) has been known for many years \cite{ZKO-1981,ZKO}.
In contrast to periodic setting, the main difficulty is that the exact corrector equation
\begin{equation}\label{corrector-eq}
(\partial_s+\mathcal{L}_1)(\chi^\beta_{j})=-\mathcal{L}_1(P^\beta_j)~~{\rm in}~~\R^{d+1}
\end{equation}
may not be solvable in the almost periodic setting for linear functions $P(y)$.
In \cite{ZKO}, V. V. Zhikov, S. M. Kozlov, and O. A. Ole{\u\i}nik constructed a solution $\chi_j^\beta$ of \eqref{corrector-eq}, which is defined by series, and does not converge in $B^2(\mathbb{R}^{d+1})$, only the formal derivative $\nabla \chi_j^\beta$ of the series converges in $B^2(\mathbb{R}^{d+1})$. However, by using trigonometric polynomials to approximate \eqref{corrector-eq}, they were able to show that almost solutions of \eqref{corrector-eq} can be constructed in the sense of distribution in $\mathbb{R}^{d+1}$, as a result, homogenization result was obtained.

We now describe the main ideas in the proof of Theorems \ref{Holder-1}-\ref{rate} and Theorems \ref{thm-interior-Lip}-\ref{thm-bound-Lip}. The key ingredient in the proof of Theorem \ref{Holder-1} and Theorem \ref{boundary holder} is to establish the size estimates of fundamental solution and Green's function. It is known that a fundamental solution in $\R^{d+1}$ can be constructed for any scalar parabolic operator
in divergence form with real, bounded measurable coefficients. In fact, the
construction can be extended to any system of second-order parabolic operators in divergence form with complex, bounded measurable coefficients, provided that the solutions of
the system and its adjoint satisfy the De Giorgi -Nash type local H\"older continuity estimates. To do this, we first use the three-step compactness argument (i.e., improvement, iteration and blow-up argument) to obtain the boundary H\"older estimates (and interior H\"older estimates) for local solutions of
\begin{equation*}
(\partial_t+\mathcal{L_\varepsilon}) u_\varepsilon=0 ~~~{\rm in}~~\Omega_{2r}(x_0,t_0),
~~~~~~~\,u_\varepsilon=0 ~~~ {\rm on}~~\Delta_{2r}(x_0,t_0).
\end{equation*}
With this, the existences of fundamental solution and Green's function and their size estimates were obtained (see Section \ref{Section-4} and Section \ref{Section-5}). We then proceed to use the Poisson's representation formula to obtain the boundary
H\"older estimates for weak solution of
$(\partial_t +\mathcal{L_\varepsilon})u_\varepsilon= F$ in $\Omega_T$, $u_\varepsilon=g$ on $S_T$ and $u_\varepsilon=h$ on $\Omega\times\{t=0\}$.

Our approach to the proof of Theorem \ref{rate} involves two important ingredients. The first one is the so-called approximate corrector equation which has been employed in \cite{Shen-2014} to obtain the convergence rates for second order elliptic systems with almost periodic coefficients. The second one is the two scale expansion method. Basically, consider the approximate corrector equation corresponding to $\partial_t+\mathcal{L}_\e$, which is defined by
\begin{equation*}
(\partial_s+\mathcal{L}_1)(\chi^\beta_{S,j})+S^{-2}\chi^\beta_{S,j}=-\mathcal{L}_1(P^\beta_j)~~{\rm in}~~\R^{d+1}.
\end{equation*}
 Let $u_S=\chi_{S,j}^\beta$, then we have
\begin{equation}\label{mean-equ}
-\left\langle \partial_t v\cdot u_S\right\rangle+\left\langle a^{\alpha\gamma}_{ik}\frac{\partial u^\gamma_S}{\partial x_k}\frac{\partial v^\alpha}{\partial x_i}\right\rangle+S^{-2}\left\langle u_S \cdot v\right\rangle=-\left\langle a_{ij}^{\alpha\beta}\frac{\partial v^\alpha}{\partial x_i}\right\rangle
\end{equation}
for any $v=(v^\alpha)\in H^1_{{\rm loc}}(\mathbb{R}^{d+1})$ and $v^\alpha,\nabla v^\alpha, \partial_t v^\alpha\in B^2(\mathbb{R}^{d+1})$ (the space of almost periodic functions), where $\langle f \rangle$ denotes the mean value of $f$ in the sense of $B^2$.

One may notice that in the elliptic case \cite{Shen-2014}, one takes the test function $v=\chi_S$ in \eqref{mean-equ}. However, this is no longer available in the parabolic setting since $\partial_t \chi_S$ may not in $B^2(\mathbb{R}^{d+1})$. To overcome this difficulty, we seek to choose a sequence of real trigonometric polynomials, which converges to $\chi_S$  in the appropriate space, as the test function $v$ in \eqref{mean-equ}. This, together with the observation that $\langle u\cdot v\rangle=(u,v)$ for any $u\in B^2(\mathbb{R}^{d+1})\cap\mathcal{K}^*$, $v\in B^2(\mathbb{R}^{d+1})\cap\mathcal{K}$ and the fact that $(\partial_t f, f)=0$ for any $f\in D(\partial_t)$ (See Section \ref{section 2} for definitions), yields that
$$\|\nabla\chi-\nabla\chi_S\|_{B^2(\mathbb{R}^{d+1})}\to 0, ~~~~\text{as}~~S\rightarrow \infty.$$
By using the uniform H\"older estimates established in Theorem \ref{Holder-1}, we show that for any $S\geq 1$ and $\sigma \in(0,1]$,
\begin{equation*}
|\chi_S(x,t)-\chi_S(y,s)|\leq CS^{1-\sigma}(|x-y|+|t-s|^{1/2})^\sigma,
\end{equation*}
and for any $\sigma_1,\sigma_2\in(0,1)$, $2<p<\infty$ and $1\leq r\leq S$,
\begin{equation*}
\sup_{(x,t)\in\R^{d+1}}\bigg(\dashint_{Q_r(x,t)}|\nabla\chi_S|^p\bigg)^{1/p}\leq C S^{\sigma_1}r^{-\sigma_2} S^{\sigma_2}.
\end{equation*}
Moreover, we show that
\begin{equation*}
\|\chi_S\|_{L^\infty(\R^{d+1})}\leq CS \Theta_\sigma (S),
\end{equation*}
where $\Theta_\sigma (S)\to 0$ as $S\to \infty$ is defined by
\begin{equation*}
\Theta_\sigma (S)=\inf_{0<R\leq S}\bigg\{\rho(R)+\bigg(\frac{R}{S}\bigg)^\sigma\bigg\}
\end{equation*}
and $\rho(R)$ is defined as in (\ref{AP-1}) which converges to $0$ as $R\to \infty$.
As mentioned above, another key ingredient in the proof of Theorem \ref{rate} is the two scale expansion, to see this, let
\begin{equation*}
\aligned
w_\varepsilon^\alpha (x, t)  = u_\varepsilon^\alpha (x, t) -u_0^\alpha  (x, t) -\varepsilon \chi_{S,j}^{\alpha\beta} (x/\varepsilon, t/\varepsilon^2)
K_\varepsilon \left(\frac{\partial u_0^\beta}{\partial x_j}\right)\\
-\varepsilon^2 \phi_{S,(d+1) ij}^{\alpha\beta} (x/\varepsilon, t/\varepsilon^2)
\frac{\partial}{\partial x_i} K_\varepsilon \left(\frac{\partial u_0^\beta}{\partial x_j}\right),
\endaligned
\end{equation*}
where $S=\varepsilon^{-1}$ and $\phi_S$ denote the approximate flux correctors. We remark here that in the periodic setting, the flux corrector satisfies $\frac{\partial}{\partial y_k}\phi^{\alpha\beta}_{kij}=b_{ij}^{\alpha\beta}$ with $B=A+A\nabla\chi-\hat{A}$ by solving a $(d+1)$-dimensional Poisson equation. In the almost periodic setting, this will be replaced by solving the approximate $(d+1)$-dimensional Poisson equation
\begin{equation*}
  -\Delta_{d+1} f_S+S^{-2}f_S =\langle b_S \rangle-b_S ~~~{\rm in} ~~ \mathbb{R}^{d+1}.
\end{equation*}
By computing $(\partial_t+\mathcal{L}_\varepsilon)w_\varepsilon$ and using the estimates for $\chi_S$ and $\phi_S$, we are able to show that
\begin{align*}
\|\nabla w_\varepsilon\|_{L^2(\Omega_T)}\leq C\delta^{1/2}\bigg\{\|u_0\|_{L^2(0,T;H^2(\Omega))}
+\|F\|_{L^2(\Omega_T)}
+\sup_{\delta^2<t<T}
\left(\frac{1}{\delta}\int_{t-\delta^2}^t\int_{\Omega}|\nabla u_0|^2\right)^{1/2}\bigg\},
\end{align*}
where $\delta=\varepsilon+\langle|\nabla \chi-\nabla \chi_S|\rangle+\Theta_\sigma(S)+\Theta_1(S).$

For the proof of Theorems \ref{thm-interior-Lip} and \ref{thm-bound-Lip}, we appeal to a convergence rate argument, which was first introduced in \cite{Armstrong-Smart-2016} to study convex integral functionals in stochastic homogenization, and has been further developed in \cite{AS-2016} to study the Lipschitz estimates for elliptic systems with almost-periodic coefficients. According to this argument (see Lemma \ref{lem-standard}), the basic idea is that to establish the large-scale Lipschitz estimates for solutions $u_\varepsilon$ of \eqref{equ-loc}, one only needs to prove that $u_\varepsilon$ can be well approximated by a $C^{1+\alpha}$ function $u_0$ at every scale large that $\varepsilon$. To show Theorem \ref{thm-bound-Lip}, we first prove that
\begin{equation*}
\left(\dashint_{T_r}|u_\varepsilon-u_0|^2\right)^{\frac{1}{2}}\leq
C\left[\eta\left(\frac{\varepsilon}{r}\right)\right]^\gamma\left\{
\left(\dashint_{T_{2r}}| u_\varepsilon|^2\right)^{\frac{1}{2}}
+r^2\left(\dashint_{Q_{2r}}|F|^2\right)^{\frac{1}{2}}
+\|f\|_{C^{1+\alpha}(I_{2r})}\right\},
\end{equation*}
whenever $u_\varepsilon$ is a weak solution of
\begin{equation*}
(\partial_t+\mathcal{L}_\varepsilon)u_\varepsilon=F\quad{\rm in}~~T_{2r}\quad{\rm and}\quad u_\varepsilon=f\quad{\rm on}~~I_{2r}
\end{equation*}
for some $0<r\leq 1$. With this, and the fact that if $u_0$ is a weak solution of $(\partial_t+\mathcal{L}_0)u_0=F$ in $T_r$, then there exists $\theta\in(0,1/4)$, such that
\begin{equation*}
\Psi(\theta r;u_0)+\theta r\left(\dashint_{T_{\theta r}}|F|^p\right)^{1/p}\leq\frac{1}{2}\left\{\Psi(r;u_0)
+r\left(\dashint_{T_r}|F|^p\right)^{1/p}\right\},
\end{equation*}
where
\begin{equation*}
\Psi(r;u)=\frac{1}{r}\inf_{\substack{E\in\mathbb{R}^{m\times d}\\ \beta\in\mathbb{R}^m}}\left\{\left(\dashint_{T_r}|u-E\cdot x-\beta|^2\right)^{1/2}+\|u-E\cdot x-\beta\|_{C^{1+\alpha}(I_r)}\right\},
\end{equation*}
the large-scale boundary Lipschitz estimate for $u_\varepsilon$ then follows from a general scheme. The proof of Theorem \ref{thm-interior-Lip} is similar.

We end this section with some notation and definitions that will be used throughout the paper. For $(x_0,t_0)\in\mathbb{R}^{d+1}$ and $r>0$, let $Q_r(x_0,t_0)=B(x_0,r)\times(t_0-r^2,t_0)$,
where $B(x_0,r)=\{x\in \mathbb{R}^d:|x-x_0|<r\}$. If $Q=Q_r(x_0,t_0)$ and $\alpha>0$, we use $\alpha Q$ to denote $Q_{\alpha r}(x_0,t_0)$.
We will use $\dashint_E u$ to denote $\frac{1}{|E|}\int_E u$, the average of $u$ over $E$.
For a function $f$ defined on $E\subset\mathbb{R}^{d+1}$, for any $0<\sigma\leq 1$, we denote the parabolic H\"older semi-norm of $f$ by
\begin{equation*}
\|f\|_{\Lambda^{\sigma,\sigma/2} (E)}=\sup_{\substack{(x,t),(y,s)\in E\\ (x,t)\neq(y,s)}}\frac{|f(x,t)-f(y,s)|}{(|x-y|+|t-s|^{1/2})^\sigma}.
\end{equation*}
Also, we define
\begin{align*}
\| f \|_{X^{\sigma} (E)}=\|\nabla_{\tan} f\|_{\Lambda^{\sigma,\sigma/2} (E)}+\langle f \rangle_{1+\sigma;E},
\end{align*}
where $\nabla_{\tan} f$ denotes the tangential gradient of $f$ in the space variables and
$$\langle f \rangle_{1+\sigma;E}=\sup_{\substack{(x,t), (x,s)\in E\\ t\neq s}}
\frac{| f(x,t)-f(x,s)|}{|t-s|^{\frac{1+\sigma}{2}}}.$$

For $x_0\in \overline{\Omega}$, $t_0\in \mathbb{R}$, and $0<r<r_0=\rm{diam(\Omega)}$, set
\begin{equation*}
\begin{aligned}
&\Omega_r(x_0,t_0)=[B(x_0,r)\cap\Omega]\times(t_0-r^2,t_0),\\
&\Delta_r(x_0,t_0)=[B(x_0,r)\cap\partial \Omega]\times(t_0-r^2,t_0).\\
\end{aligned}
\end{equation*}
For $\alpha\in(0,1)$ and $\Delta_r=\Delta_r(x_0,t_0)$, we use $C^{1+\alpha}(\Delta_r)$ to denote the parabolic $C^{1+\alpha}$ space on $\Delta_r$ with the scaling-invariant norm,
\begin{equation*}
\|f\|_{C^{1+\alpha}(\Delta_r)}=\|f\|_{L^\infty(\Delta_r)}
+r\|\nabla_{\rm tan}f\|_{L^\infty(\Delta_r)}
+r^{1+\alpha}\|\nabla_{\rm tan}f\|_{\Lambda^{\alpha,\alpha/2}(\Delta_r)}
+r^{1+\alpha}\langle f \rangle_{1+\alpha;\Delta_r}.
\end{equation*}

\section{\bf  The corrector equation and function spaces \texorpdfstring{$\mathcal{K}$}{K}, \texorpdfstring{$\mathcal{K}^*$}{K*} and \texorpdfstring{$B^2$}{B2}}\label{section 2}
In this section we introduce some preliminaries of homogenization theory for parabolic systems with almost-periodic coefficients. A detailed presentation may be found in \cite{JKO,ZKO}.

Denote by ${\rm Trig}(\mathbb{R}^{d+1})$ the set of all trigonometric polynomials in $\mathbb{R}^{d+1}$, i.e., the set of finite sums of the form
\begin{equation*}
u(x,t)=\sum_{\xi, \lambda}u_{\xi,\lambda}e^{i(x\cdot\xi+t\lambda)},\quad \xi,x\in\mathbb{R}^d;\quad t,\lambda \in\mathbb{R},
\end{equation*}
where $x\cdot\xi=\sum_{i=1}^{d}x_i\xi_i$.
The completion of the set with respect to the $L^\infty$ norm is called the Bohr space of almost periodic functions.

A function $f\in L_{\rm{loc}}^2(\mathbb{R}^{d+1})$ is called almost-periodic in the sense of Besicovitch if there is a sequence of trigonometric polynomials converging to $f$ in the Besicovitch norm
\begin{equation*}
\|f\|_{B^2}=\limsup_{R\to\infty}\left(\dashint_{Q_R(0,0)}|f|^2\right)^{\frac{1}{2}}.
\end{equation*}
The space of such functions is denoted by $B^2(\mathbb{R}^{d+1})$. For any $f$ with finite Besicovitch norm, define its mean value $\langle f \rangle$ by
\begin{equation*}
\lim_{\varepsilon\to 0}\int_{\mathbb{R}^{d+1}}f\big(\frac{x}{\varepsilon},\frac{t}{\varepsilon^2}\big)\phi(x,t)
=\langle f \rangle\int_{\mathbb{R}^{d+1}}\phi(x, t)~~{\rm for~~any}~~\phi\in C_0^\infty(\mathbb{R}^{d+1}).
\end{equation*}
One may notice that the trigonometric polynomials $\{e^{i(x\cdot\xi+t\lambda)}:(\xi,\lambda)\in\R^{d+1}\}\subset B^2(\R^{d+1})$ form an orthonormal basis in it. Hence any element $u\in B^2(\R^{d+1})$ can be identified with its Fourier series
\begin{equation*}
u=\sum_{\xi,\lambda}u_{\xi,\lambda} e^{i(x\cdot\xi+t\lambda)},
\end{equation*}
where the summation is over some countable subset of points $(\xi,\lambda)\in \R^{d+1}$ . In view of the Parseval's equality we obtain that
$$
\|u\|_{B^2}=\bigg(\sum_{\xi,\lambda}|u_{\xi,\lambda}|^2\bigg)^{1/2}.
$$

For convenience we shall only consider real elements in $B^2(\R^{d+1})$, i.e., for which $u_{\xi,\lambda}=\bar{u}_{-\xi,-\lambda}$. Consider formal real series of the form
\begin{equation}\label{2.3-1}
u=\sum_{\xi,\lambda}u_{\xi,\lambda} e^{i(x\cdot\xi+t\lambda)},~~~ u_{0,\lambda}=0, ~~~ \bar{u}_{\xi,\lambda}=u_{-\xi,-\lambda}.
\end{equation}
Unlike the elliptic case, to introduce the corrector for parabolic equations with almost periodic coefficients, we need more function spaces. Let ${\rm \tilde{T}rig}(\R^{d+1})$ denote the finite sum of the form (\ref{2.3-1}) and $\mathcal{K}$ the collection of series (\ref{2.3-1}) for which the quantity
$$
\|u\|_{\mathcal{K}}=\bigg(\sum_{\xi,\lambda}|u_{\xi,\lambda}|^2\xi^2\bigg)^{1/2}
$$
is finite. It is easy to see that in view of the condition $u_{0,\lambda}=0$, the above equality gives a norm on $\mathcal{K}$. Obviously, $\mathcal{K}$ is a Hilbert space. Also, one may observe that if $u\in {\rm \tilde{T}rig}(\R^{d+1})$, then
$$
\|u\|_{\mathcal{K}}=\|\nabla u\|_{B^2}.
$$
The space $\mathcal{K}^*$ dual to $\mathcal{K}$ is the set of series \eqref{2.3-1} with norm
$$
\|u\|_{\mathcal{K}^*}=\bigg(\sum_{\xi,\lambda}|u_{\xi,\lambda}|^2\xi^{-2}\bigg)^{1/2},
$$
where the duality between $\mathcal{K}$ and $\mathcal{K}^*$ is given by the equality
$$
(u,v)=\sum_{\xi,\lambda}\bar{u}_{\xi,\lambda}v_{\xi,\lambda}.
$$
For any $u\in \mathcal{K}$, the formal derivative
$$
\nabla u=\sum_{\xi,\lambda}u_{\xi,\lambda}i \xi e^{i(x\cdot \xi+t\lambda)}
$$
is an element of $B^2(\mathbb{R}^{d+1})$. For any $f\in B^2(\mathbb{R}^{d+1})$, the formal series for $-{\rm div} f$ defines an element of $\mathcal{K}^*$, and
$$\|{\rm div} f\|_{\mathcal{K}^*}\leq \|f\|_{B^2}.$$
Moreover, we have
\begin{equation}\label{2.3-6}
( -\text{div}f, u)=-\sum_{\xi,\lambda}\overline{i\xi}\cdot\bar{f}_{\xi,\lambda} u_{\xi,\lambda}=\sum_{\xi,\lambda}\bar{f}_{\xi,\lambda}\cdot i\xi u_{\xi,\lambda}=\langle \nabla u\cdot \bar{f}\rangle.
\end{equation}

To introduce the homogenized coefficient of $\partial_t+\mathcal{L}_\varepsilon$, we first recall the following lemma.

\begin{lemma}\label{P}
Let $\mathcal{P}$ be a closed operator with domain $D(\mathcal{P})\subset \mathcal{K}$, where $D(\mathcal{P})$ is dense in $\mathcal{K}$ and $\mathcal{P}: D(\mathcal{P})\to \mathcal{K}^*$. We assume that for any $v\in D(\mathcal{P})$,
\begin{equation}\label{P1}
  \mu\|v\|_{\mathcal{K}}^2\leq (\mathcal{P}(v), v),
\end{equation}
and for any $v\in D(\mathcal{P}^*)$,
\begin{equation}\label{P2}
  \mu\|v\|_{\mathcal{K}}^2\leq (\mathcal{P}^*v, v),
\end{equation}
where $\mathcal{P}^*$ is the adjoint of $\mathcal{P}$ and $\mu>0$. Then the equation $\mathcal{P}u=f$ has a unique solution $u\in D(\mathcal{P})$ for any $f\in \mathcal{K}^*$ and $\|\mathcal{P}^{-1}\|\leq \mu^{-1}$.
\end{lemma}
\begin{proof}
The proof was given in \cite{ZKO-1981}. We provide a proof here for the sake of completeness. We first prove that the range of $\mathcal{P}$ is closed in $\mathcal{K}^*$. Suppose that $\mathcal{P}u_k\to w$ in $\mathcal{K}^*$,
then, in view of (\ref{P1}) we have
\begin{equation*}
  \mu\|u_k-u_\ell\|^2_{\mathcal{K}}\leq \|\mathcal{P}(u_k-u_\ell)\|_{\mathcal{K}^*}\|u_k-u_\ell\|_{\mathcal{K}}.
\end{equation*}
Hence $\{u_k\}$ is a Cauchy sequence in $\mathcal{K}$ and $\lim_{k\to\infty}u_k=v$ for some $v\in \mathcal{K}$. It follows from the fact that $\mathcal{P}$ is closed we obtain $v\in D(\mathcal{P})$ and $w=\mathcal{P}v$. Thus the range of $\mathcal{P}$ is closed in $\mathcal{K}^*$.

Next we need to show that the range of $\mathcal{P}$ is indeed $\mathcal{K}^*$. Suppose not, then it follows from the well-known Hahn-Banach extension theorem, we can find $g\in \mathcal{K}$ such that $g\neq 0$ and $(\mathcal{P}v, g)=0$ for any $v\in D(\mathcal{P})$, then we have $\mathcal{P}^*g=0$. In view of (\ref{P2}), this implies that $g=0$. Hence the range of $\mathcal{P}$ is $\mathcal{K}^*$. Hence $\mathcal{P}: D(\mathcal{P})\to \mathcal{K}^*$ is one-to-one. Finally, by (\ref{P1}), we obtain
\begin{equation*}
\|v\|_{\mathcal{K}}\leq \mu^{-1} \|\mathcal{P}v\|_{\mathcal{K}^*}.
\end{equation*}
Thus $\|\mathcal{P}^{-1}\|\leq \mu^{-1}$ and we complete the proof.
\end{proof}

Next, we consider the parabolic operator $\mathcal{P}=\partial_t+\mathcal{L}_1$ where $\mathcal{L}_1=-\text{div}(A(x,t)\nabla)$, for any $u\in{\rm \tilde{T}rig}(\R^{d+1})$, we have $\partial_t u=\sum_{\xi,\lambda}i\lambda u_{\xi,\lambda} e^{i(x\cdot\xi+t\lambda)}$, and define the domain $D(\partial_t)$ by
$$
D(\partial_t)=\bigg\{u\in\mathcal{K},\,\sum_{\xi,\lambda}|u_{\xi,\lambda}|^2 \lambda^2 \xi^{-2}<\infty\bigg\}.
$$
It is clear that the operator $\partial_t:D(\partial_t)\to \mathcal{K}^* $ is closed and skew-symmetric. For each $u\in \mathcal{K}$, we know that $\mathcal{L}_1u=-\text{div}(A(x,t)\nabla u)\in \mathcal{K}^*$.  Hence $\mathcal{L}_1: \mathcal{K}\to \mathcal{K}^*$ and we then have
$$
\|\mathcal{L}_1u\|_{\mathcal{K}^*}\leq \|A(x,t)\nabla u\|_{B^2}\leq C\|\nabla u\|_{B^2}\leq C_1\|u\|_{\mathcal{K}},
$$
where we have used the fact that $fg\in B^2{(\R^{d+1})}$ if $f$ is uniformly almost periodic and $ g\in B^2{(\R^{d+1})}$, moreover, $\|fg\|_{B^2}\leq \|f\|_{L^\infty(\R^{d+1})}\|g\|_{B^2}$.

Next, it follows from \eqref{2.3-6} that
\begin{equation}\label{coercive}
  (\mathcal{L}_1u,u)=\langle A\nabla u\cdot\nabla u\rangle\geq \mu \|u\|^2_{\mathcal{K}},
\end{equation}
where the proof of the last inequality uses \eqref{ellipticity} and is given in \cite{Zhikov-1979}. From \eqref{coercive} we know that the operator $\mathcal{L}_1$ is coercive.

Since $\partial_t:D(\partial_t)\to \mathcal{K}^*$ is skew-symmetric, we have
\begin{equation}\label{skew-symmetric}
(\partial_t u,u)=0\quad {\rm for~~ any}\quad u\in D(\partial_t).
\end{equation}
Let $\mathcal{P}=\partial_t+\mathcal{L}_1:D(\partial_t)\to \mathcal{K}^*$. By \eqref{coercive} and \eqref{skew-symmetric} we know that $\mathcal{P}$ satisfies the conditions of Lemma \ref{P}.
Hence, Lemma \ref{P} is applicable, then we know that the auxiliary equation
\begin{equation}\label{corrector}
  (\partial_s +\mathcal{L}_1) (\chi_j^\beta)  = - \mathcal{L}_1(P_j^\beta)~~~~\text{in}~~~\R^{d+1}
\end{equation}
has a unique solution $u\in D(\partial_t)$, where $A(x,t)$ is uniformly almost periodic in $\mathbb{R}^{d+1}$ and $P_j^\beta(y)=y_je^\beta$ and $e^\beta=(0,...1,...,0)$ with $1$ in the $\beta-$th position.
However, in generally speaking, the series $u=\sum_{\xi,\lambda}u_{\xi,\lambda} e^{i(x\cdot\xi+t\lambda)}$ does not converge in $B^2(\mathbb{R}^{d+1})$, only the formal derivative $\nabla u$ of this series converges in $B^2(\mathbb{R}^{d+1})$.  Hence, in general, (\ref{corrector}) is not solvable in the class of almost periodic functions.


With the help of the solvability of \eqref{corrector} in $D(\partial_t)$, we define the homogenized coefficients $\hat{A}=(\hat{a}_{ij}^{\alpha\beta})$ by
\begin{equation*}
\hat{a}_{ij}^{\alpha\beta}=\langle a_{ij}^{\alpha\beta} \rangle+\langle a_{ik}^{\alpha\gamma}\frac{\partial}{\partial y_k}\chi_j^{\gamma\beta}\rangle.
\end{equation*}
Then
\begin{equation*}
\mu|\xi|^2\leq \hat{a}_{ij}^{\alpha\beta}\xi_i^\alpha \xi_j^\beta\leq \widetilde{\mu}|\xi|^2,~~{\rm for}~~{\rm any}~~\xi\in\mathbb{R}^{dm},
\end{equation*}
where $\widetilde{\mu}$ depends only on $m, d$ and $\mu$. Let $A^*$ denote the adjoint of $A$, then it is known that $\hat{A^*}=(\hat{A})^*$. Define $\mathcal{L}_0=-{\rm div}(\hat{A}\nabla)$. Then $\partial_t+\mathcal{L}_0$ is the homogenized operator associated with $\partial_t+\mathcal{L}_\varepsilon$.

The following theorem gives a homogenization result for a sequence of operators $\partial_t+\mathcal{L}_{\varepsilon_k}=\partial_t-\text{div}[A_k(x/\varepsilon_k,t/\varepsilon_k^2)\nabla]$ where $\varepsilon_k \to 0$.

\begin{theorem}
Let $\Omega\subset\mathbb{R}^d$ be a bounded Lipschitz domain, $-\infty<T_0<T_1<+\infty$ and $F \in L^2(T_0, T_1; H^{-1}(\Omega))$. Let $u_k\in L^2(T_0,T_1; W^{1,2}(\Omega))$ be a weak solution of
\begin{equation*}
\partial_t  u_k  -{\rm div}[A_k(x/\varepsilon_k,t/\varepsilon^2_k)\nabla u_k]=  F \quad {\rm in } ~~\Omega \times (T_0, T_1),
\end{equation*}
where $\varepsilon_k\to 0$ and the matrix $A_k(y,s)$ is uniformly almost periodic and satisfies (\ref{ellipticity}). Suppose that $\hat{A_k}\to \hat{A}$, and
\begin{align*}
\nabla u_k  \to \nabla u_0 \quad {\rm weakly~~ in } ~~L^2(T_0, T_1; L^2(\Omega)),\\
u_k  \to u_0 \quad {\rm strongly~~ in } ~~L^2(T_0, T_1; L^2(\Omega)).
\end{align*}
Then $u_0\in L^2(T_0,T_1; W^{1,2}(\Omega))$ is a weak solution of
\begin{equation*}
\partial_t  u_0  -{\rm div}(\hat{A} \nabla u_0)=  F \quad {\rm in } ~~\Omega \times (T_0, T_1),
\end{equation*}
and the constant matrix $\hat{A}$ satisfies the ellipticity condition (\ref{ellipticity}).
\end{theorem}

\begin{proof}
  The proof may be found in \cite{ZKO} for scalar case $(m=1)$ and $A_k$ is independent of $k$.  In the case of $m>1$ and $A_k$ depends on $k$, the proof is exactly the same.
\end{proof}

We end this section with the following Aubin-Lions lemma.
\begin{lemma}
  Let $X_0\subset X\subset X_1$ be three Banach spaces. Suppose that $X_0, X_1$ are reflexive and the injection $X_0\subset X$ is compact. Let $1<\alpha_0,\alpha_1<\infty$. Define
  $$Y=\{u:u\in L^{\alpha_0}(T_0,T_1;X_0)~~\text{and}~~\partial_t u\in L^{\alpha_1}(T_0,T_1;X_1)\}$$
  with norm
  $$\|u\|_Y=\|u\|_{L^{\alpha_0}(T_0,T_1;X_0)}+\|\partial_t u\|_{L^{\alpha_1}(T_0,T_1;X_1)}.$$
  Then, $Y$ is a Banach space, and the injection $Y\subset L^{\alpha_0}(T_0,T_1; X)$ is compact.
\end{lemma}

\section{\bf Interior H\"older estimates }

In this section, we establish the interior H\"older estimates for local solutions of
$(\partial_t+\mathcal{L}_\e) u_\varepsilon=0$ under the assumption that $A$ is uniformly almost periodic and satisfies \eqref{ellipticity}.

\begin{theorem}\label{Holder}
Suppose that $A$ is uniformly almost periodic and satisfies (\ref{ellipticity}). Let $u_\varepsilon\in L^2(t_0-4r^2,t_0;H^1(B(x_0,2r)))$ be a weak solution of
\begin{equation*}
(\partial_t+\mathcal{L}_\varepsilon)u_\varepsilon=0~~{\rm in}~~2Q=Q_{2r}(x_0,t_0).
\end{equation*}
Then for any $0<\alpha<1$,
\begin{equation*}
\|u_\varepsilon\|_{\Lambda^{\alpha,\alpha/2}(Q)}\leq C r^{-\alpha}\left(\dashint_{2Q} | u_\varepsilon|^2\right)^{1/2},
\end{equation*}
where $C$ depends at most on $d$, $m$, $\alpha$ and $A$.
\end{theorem}

\begin{theorem}\label{improvent}
   Suppose that $A$ is uniformly almost periodic and satisfies \eqref{ellipticity}. Let $0<\alpha<1$. Then there exist constants $\theta\in(0,1/4)$ and $\varepsilon_0>0$, depending only on $d$, $\mu$, $m$ and $\alpha$, such that
\begin{equation}\label{42}
\dashint_{Q_\theta}|u_\e-\dashint_{Q_\theta}u_\e|^2\leq \theta^{2\alpha}\dashint_{Q_1}|u_\e|^2,
\end{equation}
whenever $0<\e<\e_0$ and $u_\varepsilon$ is a weak solution of $\left(\partial_t +\mathcal{L}_\varepsilon\right) u_\varepsilon=0$ in $Q_1$.
\end{theorem}
\begin{proof}
  Without loss of generality, we may assume that $\dashint_{Q_1}|u_\e|^2\leq 1$. Let $A^0$ be a constant matrix satisfying the ellipticity condition (\ref{ellipticity}) and $u$ be a weak solution of $\partial_t u-\text{div}(A^0\nabla u)=0$ in $Q_{1/2}$, it follows that for any $\theta\in(0,1/4)$,
 \begin{equation}\label{43}
\dashint_{Q_\theta}|u-\dashint_{Q_\theta}u|^2\leq C_0\theta^2\dashint_{Q_{1/2}}|u|^2.
\end{equation}
  We argue by contradiction, assume (\ref{42}) is not true. Then there exist sequences $\{\e_k\}, \{A_k\}$ satisfying (\ref{ellipticity}) and \eqref{AP}, and $\{u_k\}\subset L^2\left(-1,0;W^{1,2}(B(0,1))\right)$ such that $\e_k\to 0$,
\begin{equation*}
\partial_t u_k-\text{div}[A_k(x/\e_k,t/\e^2_k)\nabla u_k]=0~~\text{in}~~Q_1,
\end{equation*}
  and
\begin{equation}\label{45}
\dashint_{Q_1}|u_k|^2\leq 1~~{\rm and}~~\dashint_{Q_\theta}|u_k-\dashint_{Q_\theta}u_k|^2> \theta^{2\alpha}.
\end{equation}
In view of Caccioppoli's inequality we know that $\{\nabla u_k\}$  is bounded in $L^2(Q_{1/2})$ since $\{ u_k\}$
 is bounded in $L^2(Q_1)$. Hence by taking a subsequence and passing to the limit we have
\begin{equation*}
\left\{\aligned
u_k\to u  \quad & {\rm weakly~~in } \quad L^2(Q_1),\\
\nabla u_k \to \nabla u \quad &{\rm weakly~~in } \quad  L^2(Q_{1/2}).
\endaligned
\right.
\end{equation*}

Note that $\{\partial_t u_k\}$ is bounded in $L^2(-1/4,0;W^{-1,2}(B(0,1/2)))$. In view of the Aubin-Lions lemma, we obtain that $u_k\to u$ strongly in $L^2(Q_{1/2})$. It then follows from (\ref{45}) that
\begin{equation*}
\dashint_{Q_1}|u|^2\leq 1~~{\rm and}~~\dashint_{Q_\theta}|u-\dashint_{Q_\theta}u|^2\geq \theta^{2\alpha}.
\end{equation*}

 Next, fix $0<\alpha<1$, choose $\theta\in(0,1/4)$ small enough such that $2^{d+2}C_0\theta^2<\theta^{2\alpha}$. Then by (\ref{43}) we obtain
\begin{equation*}
\theta^{2\alpha}\leq C_0\theta^2\dashint_{Q_{1/2}}|u|^2\leq 2^{d+2}C_0\theta^2\dashint_{Q_1}|u|^2\leq 2^{d+2}C_0\theta^2,
\end{equation*}
which is in contradiction with the choice of $\theta$.
\end{proof}

\begin{theorem}\label{iteration}
   Suppose that $A$ is uniformly almost periodic and satisfies \eqref{ellipticity}.  Fix $0<\alpha<1$. Let $\theta\in(0,1/4)$ and $\varepsilon_0$ be given by Theorem \ref{improvent}. Then
\begin{equation}\label{49}
\dashint_{Q_{\theta^k}}|u_\e-\dashint_{Q_{\theta^k}}u_\e|^2\leq \theta^{2k\alpha}\dashint_{Q_1}|u_\e|^2,
\end{equation}
whenever $0<\e<\e_0\theta^{k-1}$ for some $k\geq 1$ and $u_\varepsilon$ is a weak solution of $\left(\partial_t +\mathcal{L}_\varepsilon\right) u_\varepsilon=0$ in $Q_1$.
\end{theorem}
\begin{proof}
  The case $k=1$ was proved by Theorem \ref{improvent}. Suppose that the estimate (\ref{49}) holds for some $k\geq 1$. Let $0<\e<\e_0\theta^k$. Define
 \begin{equation*}
w(x,t)=\theta^{-\alpha k}\bigg\{u_\e(\theta^{k} x,\theta^{2k} t)-\dashint_{Q_{\theta^k}}u_\e\bigg\}/\bigg\{\dashint_{Q_1}|u_\e|^2\bigg\}^{1/2}.
\end{equation*}
  Then
  \begin{equation*}
(\partial_t+\mathcal{L}_{\frac{\e}{\theta^{k}}})w=0~~{\rm in}~~Q_1,
\end{equation*}
and by the induction assumption, we have that $\dashint_{Q_1}|w|^2\leq 1$. Since $\e/\theta^k<\e_0$, in view of Theorem \ref{improvent}, we have
\begin{equation*}
\dashint_{Q_\theta}|w-\dashint_{Q_\theta}w|^2\leq \theta^{2\alpha},
\end{equation*}
which leads to
\begin{equation*}
\dashint_{Q_{\theta^{k+1}}}|u_\e-\dashint_{Q_{\theta^{k+1}}}u_\e|^2\leq \theta^{2(k+1)\alpha}\dashint_{Q_1}|u_\e|^2.
\end{equation*}
This completes the proof.
\end{proof}

\begin{proof}[Proof of Theorem \ref{Holder}]
By Campanato's characterization of H\"older spaces, it suffices to prove
\begin{equation}\label{54}
\dashint_{Q_r}|u_\e-\dashint_{Q_r}u_\e|^2\leq Cr^{2\alpha}\dashint_{Q_1}|u_\e|^2
\end{equation}
for $0<r<1/2$, where $u_\varepsilon$ is a weak solution of $(\partial_t+\mathcal{L}_\varepsilon)u_\varepsilon=0$ in $Q_1$. Notice that in the case of $\varepsilon\geq \theta\varepsilon_0$, (\ref{54}) follows from the regularity theory for parabolic systems with $VMO_x$ coefficients in \cite{Krylov-2007}, hence we just need to handle the case $0<\varepsilon <\theta\varepsilon_0$, and we divide this into two cases. First, consider $\varepsilon/\varepsilon_0\leq r<\theta$. Set $\theta^{k+1}\leq r<\theta^k$ by choosing suitable $k\geq 1$, then in view of Theorem \ref{iteration}, we have

\begin{equation*}
\aligned
\dashint_{Q_r}|u_\e-\dashint_{Q_r}u_\e|^2&\leq C\dashint_{Q_{\theta^k}}|u_\e-\dashint_{Q_{\theta^k}}u_\e|^2
\\& \leq C\theta^{2k\alpha}\dashint_{Q_1}|u_\varepsilon|^2\leq Cr^{2\alpha}\dashint_{Q_1}|u_\varepsilon|^2.
\endaligned
\end{equation*}
Secondly, we use a blow-up argument to handle the case $r<\varepsilon/\varepsilon_0$. Let
\begin{equation*}
w(x,t)=u_\varepsilon(\varepsilon x,\varepsilon^2 t)-\dashint_{Q_{2\varepsilon/\varepsilon_0}}u_\varepsilon.
\end{equation*}
Note that $(\partial_t+\mathcal{L}_1)w=0$ in $Q_{2/\varepsilon_0}$, then we have the H\"older estimate
for second order parabolic systems with $VMO_x$ coefficients (see \cite{Krylov-2007}),
\begin{equation}\label{Holder-VMO}
\|w\|_{\Lambda^{\alpha,\alpha/2}(Q_{1/\varepsilon_0})}\leq C {\varepsilon_0}^{\alpha}\left(\dashint_{Q_{2/\varepsilon_0}} | w|^2\right)^{1/2}.
\end{equation}
For any $0<\rho<\frac{1}{\varepsilon_0}$, by (\ref{Holder-VMO}) we have
\begin{equation*}
\dashint_{Q_\rho}|w-\dashint_{Q_\rho}w|^2\leq \rho^{2\alpha}\dashint_{Q_{2/\varepsilon_0}}|w|^2.
\end{equation*}
Thus, if $r<\varepsilon/\varepsilon_0$, then
\begin{equation*}
\aligned
\dashint_{Q_r}|u_\e-\dashint_{Q_r}u_\e|^2&\leq C\bigg(\frac{r}{\varepsilon}\bigg)^{2\alpha}\dashint_{Q_{2\varepsilon/\varepsilon_0}}|u_\e-\dashint_{Q_{2\varepsilon/\varepsilon_0}}u_\e|^2
\\& \leq Cr^{2\alpha}\dashint_{Q_1}|u_\varepsilon|^2,
\endaligned
\end{equation*}
where we have used (\ref{54}) for $r=2\varepsilon/\varepsilon_0$ in the last inequality. Hence we complete the proof.
\end{proof}

\section{\bf Fundamental solutions and proof of Theorem \ref{Holder-1}}\label{Section-4}

Suppose that $A$ is uniformly almost periodic and satisfies (\ref{ellipticity}). Let $u_\varepsilon$ be a weak solution of
\begin{equation*}
(\partial_t+\mathcal{L}_\varepsilon)u_\varepsilon=0~~\text{in}~~Q_{r}=Q_r(x_0,t_0).
\end{equation*}
Then it follows from Theorem \ref{Holder} that, for any $0<\rho<r$ and $0<\alpha<1$,
\begin{equation*}
\dashint_{Q_\rho}|u_\varepsilon-\dashint_{Q_\rho}u_\varepsilon|^2\leq C\bigg(\frac{\rho}{r}\bigg)^{2\alpha}\dashint_{Q_r}|u_\varepsilon-\dashint_{Q_r}u_\varepsilon|^2,
\end{equation*}
where $C$ is independent of $\varepsilon$. In view of Caccioppoli's inequality and Poincar\'e's inequality, we have that for any $0<\rho<r$,

\begin{equation}\label{f-3}
\dashint_{Q_\rho}|\nabla u_\varepsilon|^2\leq C\bigg(\frac{\rho}{r}\bigg)^{2\alpha-2}\dashint_{Q_r}|\nabla u_\varepsilon|^2.
\end{equation}

Let $v_\varepsilon$ be a weak solution of $(-\partial_t +\mathcal{L}_\varepsilon^*)v_\varepsilon=0$ in $Q_r^+$, where
$$Q_r^+=Q_r^+(x_0,t_0)=B(x_0,r)\times (t_0,t_0+r^2)$$
and $\mathcal{L}_\varepsilon^*=-\text{div}(A^*(x/\varepsilon,t/\varepsilon^2)\nabla)$. Then $u_\varep(x,t)= v_\varepsilon(x,2t_0-t)$ is a solution of
$$
\big(\partial_t-\text{div}(\tilde{A}(x/\varepsilon,t/\varepsilon^2)\nabla)\big)u_\varepsilon=0 ~~~~\text{in} ~~Q_r(x_0,t_0)
$$ with $\tilde{A}(y,s)=A^*(y,2t_0/\varepsilon^2-s)$ is uniformly almost periodic and satisfies (\ref{ellipticity}), we thus obtain that for any $0<\rho<r$,

\begin{equation}\label{f-4}
\dashint_{Q_\rho^+}|\nabla v_\varepsilon|^2\leq C\bigg(\frac{\rho}{r}\bigg)^{2\alpha-2}\dashint_{Q_r^+}|\nabla v_\varepsilon|^2.
\end{equation}

By utilizing (\ref{f-3}) and (\ref{f-4}), one may construct an $m\times m$ matrix of fundamental solutions $\Gamma_\varepsilon(x,t;y,s)=\big(\Gamma^{\alpha\beta}_{\varepsilon,ij}(x,t;y,s)\big)$ for the parabolic operator $\partial_t+\mathcal{L}_\varepsilon$ in $\mathbb{R}^{d+1}$ (see \cite{CDK-2008}). Moreover, the matrix $\Gamma_\varepsilon(x,t;y,s)$ satisfies
\begin{equation*}
|\Gamma_\varepsilon(x,t;y,s)|\leq \frac{C}{(t-s)^{d/2}}{\rm exp}\left\{-\frac{\kappa|x-y|^2}{t-s}\right\}
\end{equation*}
for any $t>s$, $x,y\in \mathbb{R}^d$ and $\kappa,C>0$ are independent of $\varepsilon$. Furthermore, we have
\begin{equation}\label{f-6}
|\Gamma_\varepsilon(x,t;y,s)|\leq \frac{C}{(|x-y|+|t-s|^{1/2})^d}
\end{equation}
for any $(x,t)\neq (y,s)$. By using Theorem \ref{Holder} we also have
\begin{equation}\label{f-7}
|\Gamma_\varepsilon(x,t;y,s)-\Gamma_\varepsilon(x^\prime,t^\prime;y,s)|\leq \frac{C(|x-x^\prime|+|t-t^\prime|^{1/2})^\alpha}{(|x-y|+|t-s|^{1/2})^{d+\alpha}}
\end{equation}
for any $\alpha\in(0,1)$ and $(x,t), (x^\prime,t^\prime), (y,s)\in\mathbb{R}^{d+1}$ such that $(x,t)\neq (y,s)$ and $|x-x^\prime|+|t-t^\prime|^{1/2}\leq\frac{1}{2}\left(|x-y|+|t-s|^{1/2}\right)$. Since $$-\partial_s\Gamma^*_\varepsilon(y,s;x,t)-\text{div}\left[A^*(s/\varepsilon,y/\varepsilon^2)
\nabla_y\Gamma^*_\varepsilon(y,s;x,t)\right]=0~~~~{\rm in}~~ \mathbb{R}^{d+1}\backslash (x,t),
$$
and $\Gamma^*_\varepsilon(y,s;x,t)=\Gamma^t_\varepsilon(x,t;y,s)$, where $\Gamma^t_\varepsilon(x,t;y,s)$ denotes the transpose of $\Gamma_\varepsilon(x,t;y,s)$, by using Caccioppoli's inequality and (\ref{f-6}) we have
\begin{equation}\label{f-8}
\bigg(\dashint_{R\leq |x-y|+|t-s|^{1/2}\leq 2R}|\nabla_y\Gamma_\varepsilon(x,t;y,s)|^2dyds\bigg)^{1/2}\leq \frac{C}{R^{d+1}}.
\end{equation}
Moreover, by using Caccioppoli's inequality and (\ref{f-7}), we see that
\begin{equation}\label{f-9}
\aligned
&\bigg(\dashint_{R\leq|x-y|+|t-s|^{1/2}\leq 2R}|\nabla_y\{\Gamma_\varepsilon(x,t;y,s)- \Gamma_\varepsilon(x^\prime,t^\prime;y,s)\}|^2dyds\bigg)^{1/2}\\
&\qquad\qquad\leq \frac{C(|x-x^\prime|+|t-t^\prime|^{1/2})^\alpha}{R^{d+1+\alpha}}
\endaligned
\end{equation}
for any $\alpha\in(0,1)$ and $(x,t), (x^\prime,t^\prime)\in\mathbb{R}^{d+1}$ with $|x-x^\prime|+|t-t^\prime|^{1/2}\leq\frac{R}{2}$.
\bigskip

Next, we are ready to give the proof of Theorem \ref{Holder-1}.

\begin{proof}[Proof of Theorem \ref{Holder-1}]
Let $\psi\in C^\infty_0(\mathbb{R}^{d+1})$ such that $0\leq\psi\leq1$, $\psi=1$ in $\frac{3}{2}Q$, $\psi=0$ in $2Q\setminus\frac{7}{4}Q$, and $|\nabla \psi|^2+|\partial_t\psi|\leq \frac{C}{r^2}$. Then direct calculation yields that
\begin{equation*}
(\partial_t+\mathcal{L}_\varepsilon)(\psi u_\varepsilon)=(\partial_t \psi)u_\varepsilon+f\psi+\text{div}(g\psi)
-g\cdot\nabla\psi-\text{div}(A^\varepsilon(\nabla\psi)u_\varepsilon)-A^\varepsilon\nabla u_\varepsilon\cdot \nabla \psi.
\end{equation*}
Hence for each $(x,t)\in Q$, we have
\begin{equation*}
\aligned
u_\varepsilon(x,t)&=\int_{-\infty}^{+\infty}\int_{\mathbb{R}^d}\Gamma_\varepsilon(x,t;y,s)\{(\partial_t \psi)u_\varepsilon+f\psi-g\cdot\nabla\psi-A^\varepsilon\nabla u_\varepsilon \cdot\nabla \psi\}(y,s)\,dyds
\\& ~~~~+\int_{-\infty}^{+\infty}\int_{\mathbb{R}^d}\Gamma_\varepsilon(x,t;y,s)\{\text{div}(g\psi)
-\text{div}(A^\varepsilon(\nabla\psi)u_\varepsilon)\}(y,s)\,dyds.
\endaligned
\end{equation*}
 Applying integration by parts we obtain
 \begin{equation}\label{61}
\aligned
u_\varepsilon(x,t)&=\int_{-\infty}^{+\infty}\int_{\mathbb{R}^d}\Gamma_\varepsilon(x,t;y,s)\{(\partial_t \psi)u_\varepsilon+f\psi-g\cdot\nabla\psi-A^\varepsilon\nabla u_\varepsilon\cdot\nabla \psi\}(y,s)\,dyds
\\& ~~~~+\int_{-\infty}^{+\infty}\int_{\mathbb{R}^d}\nabla_y\Gamma_\varepsilon(x,t;y,s)\cdot A(y/\varepsilon,s/\varepsilon^2)
\nabla\psi(y,s)u_\varepsilon(y,s)\, dyds
\\&~~~~-\int_{-\infty}^{+\infty}\int_{\mathbb{R}^d}\nabla_y\Gamma_\varepsilon(x,t;y,s)\cdot g(y,s)\psi(y,s)\, dyds.
\endaligned
\end{equation}
Similarly, for any $(x^\prime,t^\prime)\in Q$,
 \begin{equation}\label{62}
\aligned
u_\varepsilon(x^\prime,t^\prime)&=\int_{-\infty}^{+\infty}\int_{\mathbb{R}^d}\Gamma_\varepsilon(x^\prime,t^\prime;y,s)\{(\partial_t \psi)u_\varepsilon+f\psi-g\cdot\nabla\psi-A^\varepsilon\nabla u_\varepsilon \cdot\nabla \psi\}(y,s)\,dyds
\\& ~~~~+\int_{-\infty}^{+\infty}\int_{\mathbb{R}^d}\nabla_y\Gamma_\varepsilon(x^\prime,t^\prime;y,s)\cdot A(y/\varepsilon,s/\varepsilon^2)
\nabla\psi(y,s)u_\varepsilon(y,s)\,dyds
\\&~~~~-\int_{-\infty}^{+\infty}\int_{\mathbb{R}^d}\nabla_y\Gamma_\varepsilon(x^\prime,t^\prime;y,s)\cdot g(y,s)\psi(y,s)\, dyds.
\endaligned
\end{equation}

In view of (\ref{61}) and (\ref{62}), for any $(x,t),(x^\prime,t^\prime)\in Q$,
 \begin{equation}\label{63}
\aligned
|u_\varepsilon(x,t)-u_\varepsilon(x^\prime,t^\prime)|
&\leq \int_{2Q}
|\Gamma_\varepsilon(x^\prime,t^\prime;y,s)-\Gamma_\varepsilon(x,t;y,s)\|\partial_t \psi \|u_\varepsilon|dyds
\\&~~~+ \int_{2Q}
|\Gamma_\varepsilon(x^\prime,t^\prime;y,s)-\Gamma_\varepsilon(x,t;y,s)\|f\|\psi|dyds
\\&~~~+ \int_{2Q}
|\Gamma_\varepsilon(x^\prime,t^\prime;y,s)-\Gamma_\varepsilon(x,t;y,s)\|g\|\nabla\psi|dyds
\\&~~~+ C\int_{2Q}
|\Gamma_\varepsilon(x^\prime,t^\prime;y,s)-\Gamma_\varepsilon(x,t;y,s)| |\nabla u_\varepsilon \|\nabla \psi|dyds
\\&~~~+C\int_{2Q}|\nabla_y\{\Gamma_\varepsilon(x^\prime,t^\prime;y,s)
-\Gamma_\varepsilon(x,t;y,s)\}\|\nabla\psi\|u_\varepsilon|dyds
\\&~~~+\int_{2Q}|\nabla_y\{\Gamma_\varepsilon(x^\prime,t^\prime;y,s)
-\Gamma_\varepsilon(x,t;y,s)\}\|g\|\psi| dyds.
\endaligned
\end{equation}

By using (\ref{f-7}) and H\"older's inequality, we know that the third term in the right-hand side of (\ref{63}) is bounded by
\begin{equation*}
Cr\bigg(\frac{|x-x^\prime|+|t-t^\prime|^{1/2}}{r}\bigg)^\alpha  \bigg(\dashint_{2Q}|g|^2 \bigg)^{1/2}.
\end{equation*}

In view of \eqref{f-7}, \eqref{f-9} and Caccioppoli's inequality, we know that the first, fourth and fifth terms are bounded by
\begin{equation*}
C\bigg(\frac{|x-x^\prime|+|t-t^\prime|^{1/2}}{r}\bigg)^\alpha\left\{ \bigg(\dashint_{2Q}|u_\varepsilon|^2 \bigg)^{1/2}+
r \bigg(\dashint_{2Q}|g|^2 \bigg)^{1/2}+r^2 \bigg(\dashint_{2Q}|f|^2 \bigg)^{1/2}\right\}.
\end{equation*}

To bound the second term, by using (\ref{f-6}) and (\ref{f-7}), it will be bounded by
\begin{equation}\label{66}
\aligned
&C\int_{Q_{4d}(x,t)}\frac{|f(y,s)|}{(|x-y|+|t-s|^{1/2})^{d}}dyds
+C\int_{Q_{5d}(x^\prime,t^\prime)}\frac{|f(y,s)|}{(|x^\prime-y|+|t^\prime-s|^{1/2})^{d}}dyds\\
&+C d^\alpha\int_{2Q\backslash Q_{4d}(x,t)}\frac{|f(y,s)|}{(|x-y|+|t-s|^{1/2})^{d+\alpha}}dyds,
\endaligned
\end{equation}
where $d=|x-x^\prime|+|t-t^\prime|^{1/2}$. By decomposing $Q_{4d}(x,t)$ as a union of sets $\{(y,s):2^{1-j}d\leq|y-x|+|t-s|^{1/2}< 2^{2-j}d\}$, it is easy to see that the first term in (\ref{66}) is bounded by
\begin{equation}\label{67}
Cd^\alpha \sup_{\substack{(y,s)\in Q\\0<\ell<r}}\ell^{2-\alpha} \bigg(\dashint_{Q_\ell(y,s)}|f|^2 \bigg)^{1/2}.
\end{equation}
The second and the third terms in (\ref{66}) could be handled by the same way. Also, by using \eqref{f-7} and \eqref{f-9}, a similar argument as (\ref{67}) yields that the last term in the right-hand side of (\ref{63}) will be bounded by
\begin{equation*}
Cd^\alpha \sup_{\substack{(y,s)\in Q\\0<\ell<r}}\ell^{1-\alpha} \bigg(\dashint_{Q_\ell(y,s)}|g|^2 \bigg)^{1/2}.
\end{equation*}
We thus complete the proof.
\end{proof}

\begin{lemma}\label{6.15}
  Suppose that $\partial_t u_\varepsilon-{\rm div}(A(x/\varepsilon,t/\varepsilon^2)\nabla u_\varepsilon)=f$ in $2Q=Q_{2r}(x_0,t_0)$. Let $f\in L^p(2Q)$ for some $p\geq 2$. Then
  \begin{align*}
\bigg(\dashint_{Q}|u_\varepsilon|^q\bigg)^{1/q}\leq C\bigg(\dashint_{2Q}|u_\varepsilon|^2\bigg)^{1/2}+C r^2\bigg(\dashint_{2Q}|f|^p\bigg)^{1/p},
\end{align*}
where $0<\frac{1}{p}-\frac{1}{q}\leq \frac{2}{d+2}$.
\end{lemma}
\begin{proof}
  Let $\psi$ be defined as in the proof of Theorem \ref{Holder-1}. Since
  \begin{equation*}
(\partial_t+\mathcal{L}_\varepsilon)(\psi u_\varepsilon)=(\partial_t \psi)u_\varepsilon+f\psi-\text{div}(A^\varepsilon(\nabla\psi)u_\varepsilon)-A^\varepsilon\nabla u_\varepsilon \cdot\nabla \psi,
\end{equation*}
by applying integration by parts we obtain
\begin{equation*}
\aligned
u_\varepsilon(x,t)&=\int_{-\infty}^{+\infty}\int_{\mathbb{R}^d}\Gamma_\varepsilon(x,t;y,s)\{(\partial_t \psi)u_\varepsilon+f\psi-A^\varepsilon\nabla u_\varepsilon\cdot\nabla \psi\}(y,s)\,dyds
\\& ~~~~+\int_{-\infty}^{+\infty}\int_{\mathbb{R}^d}\nabla_y\Gamma_\varepsilon(x,t;y,s)\cdot
A^\varepsilon(y,s)\nabla\psi(y,s)u_\varepsilon(y,s)\,dyds
\endaligned
\end{equation*}
for any $(x,t)\in Q$.
Using the size estimate (\ref{f-6}) and (\ref{f-8}) and H\"older's inequality, we have
 \begin{equation*}
\aligned
|u_\varepsilon(x,t)|&\leq C\left(\dashint_{2Q}|u_\varepsilon|^2\right)^{1/2}+C\int_{2Q}\frac{|f(y,s)|}{(|x-y|+|t-s|^{1/2})^d}
+Cr\left(\dashint_{\frac{7}{4}Q}|\nabla u_\varepsilon|^2\right)^{1/2}
\\&\leq C\bigg(\dashint_{2Q}|u_\varepsilon|^2\bigg)^{1/2}+\int_{2Q}\frac{|f(y,s)|}{(|x-y|+|t-s|^{1/2})^d}
+Cr^2\left(\dashint_{2Q}|f|^2\right)^{1/2},
\endaligned
\end{equation*}
where we have used Caccioppoli's inequality for the last step. Next by the boundness of fractional integral and H\"older's inequality,
we have
\begin{equation*}
\aligned
\bigg(\dashint_{Q}|u_\varepsilon|^q\bigg)^{1/q}\leq C\bigg(\dashint_{2Q}|u_\varepsilon|^2\bigg)^{1/2}
+Cr^2\bigg(\dashint_{2Q}|f|^p\bigg)^{1/p}.
\endaligned
\end{equation*}
Thus we complete the proof.
\end{proof}

\section{\bf Green's function and boundary H\"older estimates}\label{Section-5}
For $x_0\in \bar{\Omega}$, $t_0\in \mathbb{R}$ and $0<r<r_0=\text{diam}(\Omega)$, recall that
\begin{equation*}
\begin{aligned}
\Omega_r(x_0,t_0)&=[B(x_0,r)\cap\Omega]\times(t_0-r^2,t_0),\\
\Delta_r(x_0,t_0)&=[B(x_0,r)\cap\partial \Omega]\times(t_0-r^2,t_0).\\
\end{aligned}
\end{equation*}

The next theorem provides the uniform boundary H\"older estimates.

\begin{theorem}\label{Holder-0}
Let $\Omega$ be a bounded $C^{1,\eta}$ domain in $\mathbb{R}^{d}$ for some $\eta\in(0,1)$. Suppose that $A$ is uniformly almost periodic and satisfies (\ref{ellipticity}). Let $u_\varepsilon\in L^2(t_0-4r^2,t_0;H^1(B(x_0,2r)\cap\Omega))$ be a weak solution of
\begin{equation*}
\begin{cases}
(\mathcal{L_\varepsilon} +\partial_t) u_\varepsilon=0 ~~~{\rm in}~~\Omega_{2r}(x_0,t_0), \\
~~~~~~~~~~~u_\varepsilon=0 ~~~ {\rm on}~~\Delta_{2r}(x_0,t_0),
\end{cases}
\end{equation*}
then for any $0<\alpha<1$,
\begin{align*}
\|u_\varepsilon\|_{\Lambda^{\alpha,\alpha/2}(\Omega_r(x_0,t_0))}\leq
\frac{C}{r^\alpha}\left(\dashint_{\Omega_{2r}(x_0,t_0)}|u_\varepsilon|^2\right)^{1/2},
\end{align*}
where $C$ depends only on $d$, $m$, $\mu$, $\eta$, $\alpha$, $A$ and $\Omega$.
\end{theorem}
\begin{proof}
  Theorem \ref{Holder-0} follows by the three-step compactness argument, as in the periodic case and elliptic case with almost periodic coefficients. We refer the reader to \cite{GS-2013,Shen-2014} for details.
\end{proof}

We introduce the Green's matrix $ G_\varep (x,t; y,s)=\left( G_\varep ^{\alpha\beta} (x,t; y,s) \right)_{m\times m}$ with pole at $(y,s)$  for $\partial_t + \mathcal{L_\varepsilon}$
in $\mathcal{U}=\Omega \times \mathbb{R}$. Precisely, for each $(y,s)\in \mathcal{U}$,
$G_\varep (x,t; y,s)$ is defined by
\begin{equation*}
\left\{
\aligned
\left(\partial_t + \mathcal{L}_\varepsilon\right)G_\varepsilon(\cdot,\cdot;y,s)
 &=\delta_{(y,s)}(\cdot,\cdot)
 \quad \text{in } \  \mathcal{U}, \\
G_\varepsilon(\cdot,\cdot;y,s) & =0 \qquad\quad\quad  \text{on } \partial\mathcal{U},\\
\lim_{t\to s^+} G_\varep (x,t; \cdot,s)&  =\delta_x (\cdot),
\endaligned
\right.
\end{equation*}
where $\delta_{(y,s)}(\cdot,\cdot)$ and $\delta_x (\cdot)$ are Dirac delta functions with pole at $(y,s)$ and $x$, respectively. It is known that
$G_\varep (x,t;y,s)$ exists provided that the solutions of the systems and its adjoint satisfy the local H\"older continuity estimates (see \cite{CDK-2008}). Hence with interior and boundary H\"older estimates in Theorem \ref{Holder} and Theorem \ref{Holder-0}, we could construct an $m\times m$ matrix $G_\varepsilon(x,t;y,s)=(G^{\alpha\beta}_\varepsilon(x,t;y,s))$ of Green's function for $\partial_t + \mathcal{L_\varepsilon}$ in a bounded $C^{1,\eta}$ domain $\Omega$. Moreover, we have the size estimates, as follows.

\begin{theorem}
Assume that $\Omega$ and $A$ satisfy the same assumptions as in Theorem \ref{boundary holder}.
Then, for any $(x,t),(y,s)\in \mathcal{U}$ such that $(x,t)\neq (y,s)$ and $|t-s|^{1/2}<\text{\rm diam} (\Omega)$,
\begin{equation}\label{2.2}
|G_\varepsilon(x,t;y,s)|\leq
\frac{C}{(|x-y|+|t-s|^{1/2})^d},
\end{equation}
where $C$ depends only on $A$ and $\Omega$,
and for any $0<\sigma,\sigma_1,\sigma_2<1$,
\begin{equation}\label{2.3}
|G_\varepsilon(x,t;y,s)|\leq
\frac{C_\sigma\, [\delta(x)]^\sigma}{(|x-y|+|t-s|^{1/2})^{d+\sigma}},\\
\end{equation}
\begin{equation}\label{2.4}
|G_\varepsilon(x,t;y,s)|\leq
\frac{C_\sigma \, [\delta(y)]^\sigma}{(|x-y|+|t-s|^{1/2})^{d+\sigma}},\\
\end{equation}
\begin{equation}\label{2.5}
|G_\varepsilon(x,t;y,s)|\leq
\frac{C\, [\delta(x)]^{\sigma_1}[\delta(y)]^{\sigma_2}}{(|x-y|+|t-s|^{1/2})^{d+\sigma_1+\sigma_2}},
\end{equation}
where $|t-s|^{1/2}<\text{\rm diam} (\Omega)$ and
$\delta(x)=\text{\rm dist} (x, \partial\Omega)$. Moreover, for any $(x,t),(x^\prime,t^\prime),(y,s)\in \mathcal{U}$ such that $|x-x^\prime|+|t-t^\prime|^{1/2}<\frac{1}{2}\left(|x-y|+|t-s|^{1/2}\right)$, $(x,t)\neq (y,s)$ and $|t-s|^{1/2}<\text{\rm diam} (\Omega)$,
\begin{equation}\label{2.6}
	|G_\varepsilon(x,t;y,s)-G_\varepsilon(x^\prime,t^\prime;y,s)|\leq
	\frac{C_\sigma\, (|x-x^\prime|+|t-t^\prime|^{1/2})^\sigma}{(|x-y|+|t-s|^{1/2})^{d+\sigma}},
\end{equation}
The constant  $C_\sigma$ depends only on $\sigma$, $A$ and $\Omega$.
\end{theorem}

\begin{proof}
Suppose that $\left(\partial_t +\mathcal{L}_\varep\right) u_\varep =0$
in $\Omega_{\rho} (x_0, t_0)$ and $u_\varep =0$ on $\Delta_{\rho} (x_0, t_0)$, where
$x_0\in \overline{\Omega}$, $t_0\in \mathbb{R}$ and $0<\rho<\text{\rm diam}(\Omega)$.
It follows from Caccioppoli's inequality, Theorems \ref{Holder} and \ref{Holder-0} that
\begin{equation}\label{RH}
\int_{\Omega_r (x_0, t_0)}
|\nabla u_\varep |^2
\le C_0 \left(\frac{r}{\rho}\right)^{d+2\sigma} \int_{\Omega_{\rho} (x_0, t_0)}
|\nabla u_\varep|^2,
\end{equation}
where $0<r<\rho$ and $C_0$ depends only on $A$ and $\Omega$.
Since $A^*$ satisfies the same conditions as $A$,
it is clear that estimate (\ref{RH}) also holds for weak solutions of
$(-\partial_t +\mathcal{L}_\varep^*) u_\varep=0$ in $\Omega_{\rho} (x_0, t_0)$
and $u_\varep =0$ on $\Delta_\rho (x_0, t_0)$.
By \cite{CDK-2008} this implies that the matrix  $G_\varep (x,t; y,s)$ of Green function
exists, and the size estimate (\ref{2.2}) holds for any $(x,t), (y,s)\in \mathcal{U}$
satisfying
$$
|x-y| +|t-s|^{1/2} \le (1/2) \max \big( \delta (x), \delta (y)\big).
$$
However, since (\ref{RH}) is a boundary H\"older estimate, the same argument as in \cite{CDK-2008}
shows that (\ref{2.2}) in fact holds for any $(x,t), (y,s)\in \mathcal{U}$ satisfying
$|t-s|^{1/2} < \text{\rm diam} (\Omega)$.

Estimates (\ref{2.3})-(\ref{2.5}) follow from the size estimate (\ref{2.2}) by Theorem \ref{Holder-0}.
Indeed, to see (\ref{2.3}), we fix $(x_0, t_0), (y_0, s_0)\in \mathcal{U}$ such that
$|t_0-s_0|^{1/2}<\text{\rm diam}(\Omega)$.
We may assume that $\delta (x_0)<c\, r_0= c\, \big\{ |x_0-y_0| +|t_0-s_0|^{1/2} \big\}$,
for otherwise the desired estimate follows directly from (\ref{2.2}).
Consider $u_\varep (x,t)=G_\varep (x,t; y_0, s_0)$.
Since $\left( \partial_t + \mathcal{L}_\varep\right) u_\varep =0$
in $\Omega_{c\, r_0} (x_0, t_0)$ and $u_\varep =0$ on $\partial\Omega \times \mathbb{R}$,
we may apply Theorem \ref{Holder-0} to obtain
$$
\aligned
|u_\varep (x_0, t_0)| & =|u_\varep (x_0, t_0) -u_\varep (z_0, t_0)|\\
 &\le C \left( \frac{\delta (x_0)}{r_0} \right)^\sigma  \| u_\varepsilon\| _{L^\infty (\Omega_{c r_0} (z_0, t_0))}\\
&\le C\, \frac{\big[\delta(x_0)\big]^\sigma}{r_0^{d+\sigma}},
\endaligned
$$
where $z_0\in \partial\Omega$ and $|x_0-z_0|=\delta (x_0)$.
This gives (\ref{2.3}). Estimates (\ref{2.4})-(\ref{2.5}) follow in the same manner. Finally, estimate (\ref{2.6}) follows from Theorem \ref{Holder-0} and estimate \eqref{2.2}.
\end{proof}

\begin{theorem}\label{Holder-2}
  Assume that $\Omega$ and $A$ satisfy the same conditions as in Theorem \ref{boundary holder}. Suppose that
  \begin{equation*}
  (\partial_t+\mathcal{L}_\varepsilon) u_\varepsilon=F \quad \text{in}\quad\Omega_T ~~~~~\text{and}~~~ u_\varepsilon=0~~~~~\text{on}\quad \partial_p\Omega_T.
  \end{equation*}
Then
  \begin{equation}\label{5.1}
    \|u_\varepsilon\|_{\Lambda^{\alpha,\alpha/2}(\Omega_T)}\leq
     C\sup_{\substack{(x,t)\in\Omega_T\\0<r<r_0}}r^{2-\alpha}\dashint_{\Omega_{r}(x,t)}|F|
   \end{equation}
for any $0<\alpha<1$, where $r_0={\rm diam}(\Omega)$ and $C$ depends only on $A,\Omega$ and $\alpha$.
\end{theorem}
\begin{proof}
 By using the Poisson's representation formula we have
\begin{equation*}
    u_\varepsilon(x,t)=\int_0^T\int_\Omega G_\varepsilon(x,t;y,s)F(y,s)dyds.
\end{equation*}
Hence, we have that for any $(x,t), (x^\prime, t^\prime) \in \Omega_T$,
\begin{equation}\label{5.3}
    |u_\varepsilon(x,t)-u_\varepsilon(x^\prime, t^\prime)|\leq \int_0^T\int_\Omega |G_\varepsilon(x,t;y,s)-G_\varepsilon(x^\prime, t^\prime;y,s)||F(y,s)|dyds.
\end{equation}
Let $\ell=|x-x^\prime|+|t-t^\prime|^{1/2}$ and decompose $\Omega_T=[\Omega_T\backslash Q_{4\ell}(x,t)]\cup \Omega_{4\ell}(x,t)$. We use (\ref{2.2}) to estimate the integral $\int_{\Omega_{4\ell}(x,t)} |G_\varepsilon(x,t;y,s)-G_\varepsilon(x^\prime, t^\prime;y,s)||F(y,s)|dyds$, which gives
\begin{align}\label{5.4}
   &\int_{\Omega_{4\ell}(x,t)}|G_\varepsilon(x,t;y,s)-G_\varepsilon(x^\prime, t^\prime;y,s)||F(y,s)|dyds\nonumber
   \\&\leq C\int_{\Omega_{5\ell}(x^\prime,t^\prime)} \frac{|F(y,s)|}{(|x^\prime-y|+|t^\prime-s|^{1/2})^d} dyds+\int_{\Omega_{4\ell}(x,t)} \frac{|F(y,s)|}{(|x-y|+|t-s|^{1/2})^d}dyds\nonumber
   \\& \leq C \ell^\alpha \sup_{\substack{(x,t)\in\Omega_T\\0<r<r_0}}r^{2-\alpha}\dashint_{\Omega_{r}(x,t)}|F|,
\end{align}
where we have used the dyadic decomposition
$$\Omega_{r}(x,t)=\cup_{k=0}^\infty \big\{\Omega_{2^{-k}r}(x,t)\backslash\Omega_{2^{-k-1}r}(x,t)\big\}.$$

For the integral $\int_{\Omega_T\backslash Q_{4\ell}(x,t)} |G_\varepsilon(x,t;y,s)-G_\varepsilon(x^\prime, t^\prime;y,s)||F(y,s)|dyds$, we choose $\beta\in(\alpha, 1)$ and use (\ref{2.6}) to obtain that
\begin{align}\label{5.5}
   &\int_{\Omega_T\backslash Q_{4\ell}(x,t)}|G_\varepsilon(x,t;y,s)-G_\varepsilon(x^\prime, t^\prime;y,s)||F(y,s)|dyds\nonumber
   \\&\leq C\ell^{\beta}\int_{\Omega_T\backslash Q_{4\ell}(x,t)}\frac{|F(y,s)|}{(|x-y|+|t-s|^{1/2})^{d+\beta}}dyds\nonumber
   \\&\leq C\ell^{\alpha}\sup_{\substack{(x,t)\in\Omega_T\\0<r<r_0}} r^{2-\alpha}\dashint_{\Omega_r(x,t)}|F|.
\end{align}
Combining \eqref{5.3} with \eqref{5.4} and \eqref{5.5}, we obtain \eqref{5.1}.
\end{proof}

\begin{theorem}\label{Holder-3}
Assume that $\Omega$ and $A$ satisfy the same conditions as in Theorem \ref{boundary holder}. Suppose that
  \begin{equation*}
  (\partial_t+\mathcal{L}_\varepsilon) u_\varepsilon= {\rm div}(f)\quad \text{in}\quad\Omega_T ~~~~~\text{and}~~~ u_\varepsilon=0~~~~~\text{on}\quad \partial_p\Omega_T.
  \end{equation*}
  Then
  \begin{equation*}
    \|u_\varepsilon\|_{\Lambda^{\alpha,\alpha/2}(\Omega_T)}\leq
     C_\alpha\sup_{\substack{(x,t)\in\Omega_T\\0<r<r_0}}r^{1-\alpha}
     \left(\dashint_{\Omega_{r}(x,t)}|f|^2\right)^{1/2}
   \end{equation*}
for any $0<\alpha<1$, where $r_0={\rm diam}(\Omega)$ and $C_\alpha$ depends only on $A$, $\Omega$ and $\alpha$.
\end{theorem}
\begin{proof}
Notice that
  \begin{equation*}
    |u_\varepsilon(x,t)-u_\varepsilon(x^\prime, t^\prime)|\leq \int_0^T\int_\Omega |\nabla_y(G_\varepsilon(x,t;y,s)-G_\varepsilon(x^\prime, t^\prime;y,s))||f(y,s)|dyds.
\end{equation*}
Using Caccioppoli's inequality we have
\begin{equation}\label{5.8}
\aligned
&\left(\dashint_{R\leq |x-y|+|t-s|^{1/2}\leq 2R}|\nabla_y G_\varepsilon(x,t;y,s)|^2dyds\right)^{1/2}\\&
\qquad\qquad\leq \frac{C}{R}\left(\dashint_{R/2\leq |x-y|+|t-s|^{1/2}\leq 3R}|G_\varepsilon(x,t;y,s)|^2dyds\right)^{1/2},
\endaligned
\end{equation}
and
\begin{equation*}
\aligned
&\bigg(\dashint_{R\leq |x-y|+|t-s|^{1/2}\leq 2R}|\nabla_y\{G_\varepsilon(x,t;y,s)- G_\varepsilon(x^\prime,t^\prime;y,s)\}|^2dyds\bigg)^{1/2}
\\&\qquad\qquad\leq \frac{C}{R}\bigg(\dashint_{R/2\leq |x-y|+|t-s|^{1/2}\leq 3R}|G_\varepsilon(x,t;y,s)- G_\varepsilon(x^\prime,t^\prime;y,s)|^2dyds\bigg)^{1/2},
\endaligned
\end{equation*}
where $|x-x^\prime|+|t-t^\prime|^{1/2}<\frac{1}{4}\left(|x-y|+|t-s|^{1/2}\right)$. The rest of the proof follows from the same manner as Theorem \ref{Holder-2}.
\end{proof}

\begin{theorem}\label{Holder-4}
Assume that $\Omega$ and $A$ satisfy the same conditions as in Theorem \ref{boundary holder}. Suppose that $g\in X^\alpha(S_T)$ for some $\alpha\in(0,1)$. Let $u_\varepsilon \in
L^2(0,T;H^1(\Omega))$ be the weak solution to
\begin{equation*}
  (\partial_t+\mathcal{L}_\varepsilon) u_\varepsilon= 0\quad \text{in}~~\Omega_T, ~~~~ u_\varepsilon=g~~\text{on}~~S_T ~~~\text{and}~~ ~~u_\varepsilon=0~~\text{on}~~\Omega\times\{t=0\}.
  \end{equation*}
Then we have
\begin{equation*}
\|u_\varepsilon\|_{\Lambda^{\alpha,\alpha/2}(\Omega_T)}\leq C \|g\|_{X^\alpha(S_T)}
\end{equation*}
for any $0<\alpha<1$, where $C$ depends only on $\alpha$, $\Omega$ and $A$.
\end{theorem}

\begin{proof}
We may assume that $\|g\|_{X^\alpha(S_T)}=1$. Let $v$ be the weak solution of heat equation $(\partial_t-\Delta) v=0$ in $\Omega_T$, $v=g$ on $S_T$ and $v=0$ on $\Omega\times\{t=0\}$. It is well known that $\|v\|_{\Lambda^{\alpha,\alpha/2}(\Omega_T)}\leq C \|g\|_{X^\alpha(S_T)}=C$. In view of the interior estimate of heat equation we obtain that for any $(x,t)\in\Omega_T$,
\begin{equation}\label{4.2}
|\nabla v(x,t)|\leq C \delta(x)^{\alpha-1}.
\end{equation}
Since $(\partial_t+\mathcal{L_\varepsilon})(u_\varepsilon-v)=-(\partial_t+\mathcal{L_\varepsilon})v$ in $\Omega_T$ and $u_\varepsilon-v=0$ on $\partial_p\Omega_T$, we obtain
\begin{equation}\label{4.3}
\aligned
u_\varepsilon(x,t)-v(x,t)&=-\int_0^t\int_\Omega \nabla_y G_\varep (x,t;y,s)\cdot A(\frac{y}{\varep},\frac{s}{\varep^2})\nabla_y v(y,s)dyds
\\&~~~-\int_0^t\int_\Omega G_\varep (x,t;y,s)\partial_s v(y,s)dyds
\\&=I+II.
\endaligned
\end{equation}

We claim that for any $(x,t)\in\Omega_T$,
\begin{equation}\label{est-I}
	|I|\leq Cr^\alpha.
\end{equation}
To show \eqref{est-I}, note that by \eqref{4.2}, we have
\begin{align*}
|I|\leq C\int_0^t\int_\Omega |\nabla_y G_\varep (x,t;y,s)|\delta(y)^{\alpha-1}dyds.
\end{align*}
Fix $(x,t)\in\Omega_T$ and let $r=\delta(x)/2$. Write $\Omega_T =E_r^1 \cup E_r^2$, where
$E_r^1 =\Omega_T\cap Q_r (x,t)$
and
$E_r^2=\Omega_T\setminus Q_r (x,t)$.
By H\"older's inequality and (\ref{5.8}) we have
\begin{align}\label{2.23}
& \iint_{E^1_r}|\nabla_y G_\varepsilon(x,t; y,s)|\delta(y)^{\alpha-1}\, dyds\nonumber
\\& \leq C r^{\alpha-1}\sum_{j=0}^{\infty}\int_{2^{-j-1}r\leq |x-y|+|t-s|^{1/2}\leq 2^{-j}r}
|\nabla_y G_\varepsilon(x,t; y,s)|\, dyds \nonumber
\\& \leq C r^{\alpha-1}\sum_{j=0}^{\infty}(2^{-j}r)^{\frac{d+2}{2}}\bigg\{\int_{2^{-j-1}r\leq |x-y|+|t-s|^{1/2}\leq 2^{-j}r}
|\nabla_y G_\varepsilon(x,t; y,s)|^2\, dyds\bigg\}^{1/2} \nonumber
\\& \leq C r^{\alpha-1}\sum_{j=0}^{\infty}(2^{-j}r)^{\frac{d+2}{2}-1} \bigg\{\int_{2^{-j-2}r\leq |x-y|+|t-s|^{1/2}\leq 2^{-j+1}r}
|G_\varepsilon(x,t; y,s)|^2\, dyds\bigg\}^{1/2},
\end{align}
where the decomposition
$$E^1_r=\cup_{k=0}^\infty \{(y,s)\in\Omega_T:2^{-k-1}r\leq |x-y|+|t-s|^{1/2}\leq 2^{-k}r\}$$
was used in the first inequality.
This, together with the size estimate (\ref{2.2}), implies that the left-hand side of (\ref{2.23}) is bounded by
\begin{align}\label{2.23-1}
& C\, r^{\alpha-1}\sum_{j=0}^{\infty}(2^{-j}r)^{d+1}\bigg\{
\dashint_{2^{-j-2}r\leq |x-y|+|t-s|^{1/2}\leq 2^{-j+1}r}(|x-y|+|t-s|^{1/2})^{-2d}\, dyds\bigg\}^{1/2}\nonumber
\\ \leq &  C\, r^{\alpha-1}\sum_{j=0}^{\infty}(2^{-j}r)^{d+1}(2^{-j-2}r)^{-d}\nonumber
\\ \leq& C r^\alpha.
\end{align}

To estimate the integral on $E_r^2$, decomposing $E_r^2$ as a nonoverlapping union of cubes Q with the property that $3Q\subset\Omega_T\backslash(x,t)$ and $\ell(Q)\sim{\rm dist}(Q,\partial_p\Omega_T)$. Then by using Caccioppoli's inequality and \eqref{2.5}, a similar argument as in \cite{Shen-2014} for elliptic systems yields that
\begin{equation}\label{2.24}
\iint_{E_r^2}|\nabla_y G_\varepsilon(x,t;y,s)|\delta(y)^{\alpha-1}\, dyds
\leq C\,  r^\alpha.
\end{equation}
Combining (\ref{2.23-1}) with (\ref{2.24}), we obtain \eqref{est-I}.

To treat $II$, we first use the interior estimate of heat equation to obtain that for any $(x,t)\in\Omega_T$,
\begin{equation*}
|\partial_t v(x,t)|\leq C \delta(x)^{\alpha-2}.
\end{equation*}
Thus we have
\begin{align*}
|II|\leq C\int_0^t\int_\Omega |G_\varep (x,t;y,s)|\delta(y)^{\alpha-2}dyds.
\end{align*}
Note that by \eqref{2.3} we have
\begin{align}\label{4.13}
 \iint_{E^1_r}&|G_\varepsilon(x,t; y,s)|\delta(y)^{\alpha-2}\, dyds\nonumber\\
& \leq C r^{\alpha-2}\iint_{\Omega_r (x,t)}
\frac{[\delta(x)]^\alpha}{(|x-y|+|t-s|^{1/2})^{d+\alpha}}\, dyds \nonumber
\\& \leq C\, r^{\alpha-2} r^{\alpha}\sum_{j=0}^{\infty}\int_{2^{-j-1}r\leq |x-y|+|t-s|^{1/2}\leq 2^{-j}r}
\frac{1}{(|x-y|+|t-s|^{1/2})^{d+\alpha}}\, dyds \nonumber
\\& \leq C\, r^{2\alpha-2}\sum_{j=0}^{\infty}\left(2^{-j-1}r\right)^{-(d+\alpha)}\left(2^{-j}r\right)^{d+2}\nonumber
\\&\leq Cr^\alpha.
\end{align}

To estimate the integral on $E_r^2$, in view of (\ref{2.5}), we see that
\begin{align*}
& \iint_{E_r^2}|G_\varepsilon(x,t;y,s)|\delta(y)^{\alpha-2}\, dyds
\leq C\,  r^{\alpha_1}\iint_{E_r^2}
\frac{ [\delta(y)]^{\alpha+\alpha_2-2}}{(|x-y|+|t-s|^{1/2})^{d+\alpha_1+\alpha_2}}\, dyds.
\end{align*}
A similar argument as in \eqref{2.24} yields that
\begin{align}\label{4.14}
	& \iint_{E_r^2}|G_\varepsilon(x,t;y,s)|\delta(y)^{\alpha-2}\, dyds
	\leq C  r^{\alpha}.
\end{align}
Combining (\ref{4.13}) with (\ref{4.14}), we obtain
$$|II|\leq Cr^\alpha.$$
This, together with \eqref{4.3} and \eqref{est-I}, leads that
\begin{align}\label{4.7}
	|u_\varepsilon(x,t)-v(x,t)|\leq C[\delta(x)]^{\alpha}
\end{align}
for any $(x,t)\in\Omega_T$. It follows that
\begin{align}\label{4.8}
\|u_\varepsilon\|_{L^\infty(\Omega_T)}\leq C\|v\|_{L^\infty(\Omega_T)}+C\leq C.
\end{align}

To prove the desired estimate $|u_\varep(x,t)-u_\varep(y,s)|\leq C(|x-y|+|t-s|^{1/2})^\alpha$, we consider three cases: (1) $|x-y|+|t-s|^{1/2}<\frac{1}{4}\delta(x)$; (2) $|x-y|+|t-s|^{1/2}<\frac{1}{4}\delta(y)$; (3) $|x-y|+|t-s|^{1/2}\geq \max\{\frac{1}{4}\delta(x),\frac{1}{4}\delta(y)\}$. For the first case, it follows from interior H\"older estimates that
\begin{align*}
|u_\varep(x,t)-u_\varep(y,s)|&\leq C(|x-y|+|t-s|^{1/2})^\alpha \|u_\varep\|_{L^\infty(Q_{\delta(x)}(x,t))}\nonumber
\\&\leq C(|x-y|+|t-s|^{1/2})^\alpha,
\end{align*}
where we have used \eqref{4.8} in the last inequality.
The second case will be handled in the same manner. For the third case, we observe that
\begin{align*}
|u_\varep(x,t)-u_\varep(y,s)|&\leq |u_\varep(x,t)-v(x,t)|+|v(x,t)-v(y,s)|+|v(y,s)-u_\varep(y,s)|\nonumber
\\&\leq C\left\{[\delta(x)]^\alpha+(|x-y|+|t-s|^{1/2})^\alpha+[\delta(y)]^\alpha\right\}\nonumber
\\&\leq C(|x-y|+|t-s|^{1/2})^\alpha,
\end{align*}
where we have used \eqref{4.7} in the second inequality. Thus we complete the proof.
\end{proof}

We are ready to give the proof of Theorem \ref{boundary holder}.

\begin{proof}[Proof of Theorem \ref{boundary holder}]
  In view of Theorem \ref{Holder-2}, Theorem \ref{Holder-3} and Theorem \ref{Holder-4}, Theorem \ref{boundary holder} follows directly.
\end{proof}

\section{\bf Approximate correctors and approximate dual correctors}

In this section we carry out a quantitative study of the approximate correctors $\chi_S=\chi^\beta_{S,j}$ for $\partial_t+\mathcal{L}_\varepsilon$. For $1\leq j\leq d$ and $1\leq \beta \leq m $, consider
\begin{equation}\label{6.10}
\partial_s \chi^\beta_{S,j}+\mathcal{L}_1(\chi^\beta_{S,j})+S^{-2}\chi^\beta_{S,j}=-\mathcal{L}_1(P^\beta_j)~~{\rm in}~~\R^{d+1},
\end{equation}
where $P_j^\beta(y)=y_j(0,...0,1,0,...0)$ with $1$ in the $\beta-$th position.

\begin{lemma}\label{6.1}
Let $f,g\in L_{{\rm loc}}^2(\mathbb{R}^{d+1})$. Assume that
\begin{equation*}
\sup_{(x,t)\in \mathbb{R}^{d+1}}\int_{Q_1(x,t)}(|f|^2+|g|^2)<\infty.
\end{equation*}
Then for $S>0$, there exists a unique $u\in L^2_{{\rm loc}}(\mathbb{R}; H^1_{\rm loc}(\mathbb{R}^d))$ such that
\begin{equation}\label{6.2}
\partial_t u-{\rm div}(A(x,t)\nabla u)+ S^{-2}u =f+{\rm div}(g)~~{\rm in} ~~\mathbb{R}^{d+1},
\end{equation}
and
\begin{equation}\label{6.2-1}
\sup_{(x,t)\in \mathbb{R}^{d+1}}\int_{Q_1(x,t)}(|u|^2+|\nabla u|^2)<\infty.
\end{equation}
Moreover, the solution $u$ satisfies the estimate
\begin{equation}\label{6.3}
\sup_{(x,t)\in \mathbb{R}^{d+1}}\dashint_{Q_S(x,t)}(|\nabla u|^2+S^{-2}|u|^2)\leq C\sup_{(x,t)\in \mathbb{R}^{d+1}}\dashint_{Q_S(x,t)}(S^{2}|f|^2+|g|^2),
\end{equation}
where $C$ depends only on $d, m$ and $\mu$.
\end{lemma}
\begin{proof}
  By rescaling we may assume $S=1$. The proof is similar to the elliptic case \cite{Shen-2014}, we give a proof here for the sake of completeness.

  For the existence, note that if $g,h\in L^2(\mathbb{R}^{d+1})$, then the equation \eqref{6.2} with $S=1$
is solvable in $L^2(\mathbb{R};H^{1}(\mathbb{R}^{d}))$, and the solution $u$ satisfies
\begin{equation}\label{6.4}
-\int_{\mathbb{R}^{d+1}}  u \partial_t\phi +\int_{\mathbb{R}^{d+1}}  A\nabla u\cdot\nabla \phi+ \int_{\mathbb{R}^{d+1}}  u\phi=\int_{\mathbb{R}^{d+1}}  f\phi-\int_{\mathbb{R}^{d+1}}  g\cdot\nabla \phi
\end{equation}
for any $\phi\in C^\infty(\mathbb{R}^{d+1})$.
For $f,g \in L^2(\mathbb{R}^{d+1})$ with compact support, let $u\in L^2(\mathbb{R};H^{1}(\mathbb{R}^{d}))$ be the solution of \eqref{6.2} with $S=1$. We claim that there exist $\theta$ and $C$, such that for any $\lambda<\theta$,
\begin{equation}\label{claim}
\int_{\mathbb{R}^{d+1}}e^{\lambda |x|}(|\nabla u|^2+|u|^2)\leq C \int_{\mathbb{R}^{d+1}}e^{\lambda |x|}(|f|^2+|g|^2).
\end{equation}
We assume the claim (\ref{claim}) for a while. Then for $f,g\in L_{{\rm loc}}^2(\mathbb{R}^{d+1})$, the solvability of (\ref{6.2})  in $L^2_{{\rm loc}}(\mathbb{R};H^1_{{\rm loc}}(\mathbb{R}^d))$ and estimate (\ref{6.3}) with $S=1$ follow from the same line as in \cite{PY} for elliptic case. We omit the details.

It remains to prove (\ref{claim}). Taking $\phi=u(x,t) e^{\lambda \xi}$ as test function in (\ref{6.4}), where $\xi=\xi_\varepsilon=\frac{|x|}{1+\varepsilon|x|}$ and $\lambda>0$ is sufficiently small. To deal with the first two terms in the left-hand side of (\ref{6.4}), one has
\begin{equation}\label{6.7}
\int_{\mathbb{R}^{d+1}} u \partial_t u e^{\lambda \xi}=\frac{1}{2}\int_{\mathbb{R}^{d+1}}  \partial_t (|u|^2) e^{\lambda \xi}=0,
\end{equation}
and
\begin{align}\label{6.8}
\int_{\mathbb{R}^{d+1}}  A(x,t)\nabla u\cdot\nabla \phi&= \int_{\mathbb{R}^{d+1}}  A(x,t)\nabla u\cdot\nabla u e^{\lambda \xi}+\lambda\int_{\mathbb{R}^{d+1}} A(x,t) \nabla u \cdot u\nabla \xi  e^{\lambda \xi} \nonumber
\\& \geq \mu\int_{\mathbb{R}^{d+1}}  |\nabla u|^2 e^{\lambda \xi}+\lambda\int_{\mathbb{R}^{d+1}} A(x,t) \nabla u \cdot u\nabla \xi e^{\lambda \xi},
\end{align}
where the ellipticity condition \eqref{ellipticity} was used in the last step.  Moreover, by Cauchy's inequality, we see that
\begin{equation}\label{6.8-1}
\left|\int_{\mathbb{R}^{d+1}} g\cdot\nabla \phi\right|\leq C\int_{\mathbb{R}^{d+1}} e^{\lambda \xi}|g|^2+\varepsilon \lambda^2 \int_{\mathbb{R}^{d+1}} |u|^2 e^{\lambda \xi} |\nabla \xi|^2+\frac{\mu}{4}\int_{\mathbb{R}^{d+1}}  |\nabla u|^2 e^{\lambda \xi}.
\end{equation}
The first term in the right-hand side of \eqref{6.4} could be handled easily. Finally, one may notice that $\xi_\varepsilon\to |x|$ as $\varepsilon\to 0$. This, together with (\ref{6.4}), (\ref{6.7}), (\ref{6.8}) and (\ref{6.8-1}), implies that (\ref{claim}) holds for $\lambda$ sufficiently small.

For the uniqueness, assume that $u\in L^2_{{\rm loc}}(\mathbb{R};H^1_{{\rm loc}}(\mathbb{R}^d))$ is a weak solution to $\partial_t u-{\rm div}(A\nabla u)+u=0$ in $\mathbb{R}^{d+1}$. Then for any $L\geq1$, by Caccioppoli's inequality,
\begin{equation*}
\int_{Q_L(0,0)}(|u|^2+|\nabla u|^2)\leq \frac{C}{L^2}\int_{Q_{2L}(0,0)}|u|^2,
\end{equation*}
which implies that
\begin{equation*}
\int_{Q_L(0,0)}|u|^2\leq \frac{C}{L^{2(d+2)}}\int_{Q_{2^{d+2}L}(0,0)}|u|^2
\end{equation*}
for any $L\geq 1$. However, \eqref{6.2-1} implies that $\int_{Q_{2^{d+2}L}(0,0)}|u|^2\leq CL^{d+2}$. Hence $$\int_{Q_L(0,0)}|u|^2\leq \frac{C}{L^{d+2}}\quad {\rm for~any}~~L\geq 1. $$
By letting $L\rightarrow \infty$, we obtain that $u\equiv0$. Thus we complete the proof.
\end{proof}

As a direct consequence of Lemma \ref{6.1}, we have
\begin{corollary}\label{6.9}
For $S> 0$, let $\chi_S=(\chi^\beta_{S,j})$ be the weak solution of \eqref{6.10}. Then we have
\begin{equation*}
\sup_{(x,t)\in\R^{d+1}}\dashint_{Q_S(x,t)}(|\nabla\chi_S|^2+S^{-2}|\chi_S|^2)\leq C,
\end{equation*}
where $C$ depends only on $d$, $m$ and $\mu$.
\end{corollary}

\begin{lemma}\label{6.11}
Let $(x,t), (y, s), (z,\tau)\in\R^{d+1}$. Then
\begin{equation*}
\aligned
\bigg(\dashint_{Q_S(x,t)}&|\nabla\chi_S(\lambda+y, \theta+s)-\nabla\chi_S(\lambda+z, \theta+\tau)|^2d\lambda d\theta\bigg)^{\frac{1}{2}}\\&~~~~\leq C\|A(\cdot+y, \cdot+s)-A(\cdot+z, \cdot+\tau)\|_{L^\infty(\R^{d+1})},
\endaligned
\end{equation*}
and
\begin{equation*}
\aligned
\bigg(\dashint_{Q_S(x,t)}&|\chi_S(\lambda+y, \theta+s)-\chi_S(\lambda+z, \theta+\tau)|^2d\lambda d\theta\bigg)^{\frac{1}{2}}\\
&~~~~\leq C\,S\|A(\cdot+y, \cdot+s)-A(\cdot+z, \cdot+\tau)\|_{L^\infty(\R^{d+1})},
\endaligned
\end{equation*}
where $C$ depends only on $d$, $m$ and $\mu$.
\end{lemma}
\begin{proof}
Let $u(\lambda,\theta)=\chi^\beta_{S,j}(\lambda+y, \theta+s)$, $v(\lambda,\theta)=\chi^\beta_{S,j}(\lambda+z, \theta+\tau)$ and $w=u-v$. Then direct computation shows that
\begin{align*}
\partial_t w&-\text{div}(A(\lambda+y, \theta+s)\nabla w)+S^{-2}w\nonumber
\\&=\partial_t (u-v)-\text{div}(A(\lambda+y, \theta+s)\nabla u)\nonumber
+\text{div}(A(\lambda+y, \theta+s)\nabla v)+S^{-2}(u-v)\nonumber
\\&~~~+\text{div}(A(\lambda+z, \theta+\tau)\nabla v)-\text{div}(A(\lambda+z, \theta+\tau)\nabla v)\nonumber
\\&=\text{div}\big\{[A(\lambda+y, \theta+s)-A(\lambda+z, \theta+\tau)](\nabla P^\beta_j+\nabla v)\big\}.
\end{align*}
Hence, in view of Theorem \ref{6.1} and Corollary \ref{6.9}, we see that
\begin{align*}
\dashint_{Q_S(x,t)}|\nabla w|^2+S^{-2}|w|^2\leq &C\sup_{(x,t)\in\R^{d+1}}\dashint_{Q_S(x,t)}|A(\lambda+y, \theta+s)-A(\lambda+z, \theta+\tau)|^2\nonumber
\\&+C\sup_{(x,t)\in\R^{d+1}}\dashint_{Q_S(x,t)}|A(\lambda+y, \theta+s)-A(\lambda+z, \theta+\tau)|^2|\nabla v|^2\nonumber
\\ \leq &C\|A(\cdot+y, \cdot+s)-A(\cdot+z, \cdot+\tau)\|^2_{L^\infty(\R^{d+1})}\nonumber\\
&~~\cdot\left\{1+\sup_{(x,t)\in\R^{d+1}}\dashint_{Q_S(x,t)}|\nabla v|^2\right\}\nonumber
\\ \leq &C\|A(\cdot+y, \cdot+s)-A(\cdot+z, \cdot+\tau)\|^2_{L^\infty(\R^{d+1})}.
\end{align*}
Thus we complete the proof.
\end{proof}

\begin{remark}
Since $A$ is uniformly almost-periodic, for any $\varepsilon>0$, the set
	$$
	\left\{(z,\tau)\in\mathbb{R}^{d+1}:\|A(\cdot+z, \cdot+\tau)-A(\cdot, \cdot)\|_{L^\infty(\R^{d+1})}<\varepsilon\right\}
	$$
is relatively dense in $\mathbb{R}^{d+1}$. For $f\in L^2_{\rm loc}(\mathbb{R}^{d+1})$, define
$$\|f\|_{W^2}=\limsup_{L\rightarrow \infty}\sup_{(x,t)\in\mathbb{R}^{d+1}}\left(\dashint_{Q_L(x,t)}|f|^2\right)^\frac{1}{2}.
$$
By Lemma \ref{6.11}, we know that
\begin{equation*}
	\aligned
&\|\nabla\chi_S(\cdot+z, \cdot+\tau)-\nabla\chi_S(\cdot, \cdot)\|_{W^2}
+S^{-1}\|\chi_S(\cdot+z, \cdot+\tau)-\chi_S(\cdot, \cdot)\|_{W^2}	\\
&\leq C\|A(\cdot+z, \cdot+\tau)-A(\cdot, \cdot)\|_{L^\infty(\R^{d+1})},
	\endaligned
\end{equation*}
which implies that
\begin{equation*}
	\aligned
	\big\{(z,\tau)\in\mathbb{R}^{d+1}:\|\nabla\chi_S(\cdot+z, \cdot+\tau)-\nabla\chi_S(\cdot, \cdot)\|_{W^2}
	+S^{-1}\|\chi_S(\cdot+z, \cdot+\tau)-\chi_S(\cdot, \cdot)\|_{W^2}<\varepsilon\big\}
	\endaligned
\end{equation*}
is relatively dense in $\mathbb{R}^{d+1}$. Hence $\nabla \chi_S$ and $\chi_S$ are limits of sequences of trigonometric polynomials w.r.t. the seminorm $\|\cdot\|_{W^2}$. As a result, $\chi_S,\,\nabla \chi_S\in B^2(\mathbb{R}^{d+1})$ for any $S>0$.

Since $\nabla \chi_S\in B^2(\mathbb{R}^{d+1})$ and $A$ is uniformly almost periodic, we see that $-{\rm div}(A\nabla\chi_S)\in \mathcal{K}^*$. This, together with $\chi_S\in B^2(\mathbb{R}^{d+1})$ and \eqref{6.10}, implies that $\partial_t\chi_S\in \mathcal{K}^*+B^2(\mathbb{R}^{d+1})$. Moreover,
\begin{equation}\label{time-der}
\aligned
\|\partial_t\chi_S\|_{\mathcal{K}^*+B^2}&\leq \|{\rm div}[A(\nabla P+\nabla\chi_S)]\|_{\mathcal{K}^*}+S^{-2}\|\chi_S\|_{B^2}\\
&\leq C\|\nabla P+\nabla\chi_S\|_{B^2}+S^{-2}\|\chi_S\|_{B^2}.
\endaligned
\end{equation}
\end{remark}
\begin{lemma}\label{Lem-6.4}
Let $u_S=\chi_{S,j}^\beta$ for some $ S\geq 0$, $1\leq j\leq d$ and $1\leq \beta\leq m$. Then
\begin{equation}\label{6.19-1}
-\left\langle \partial_t v\cdot u_S\right\rangle+\left\langle a^{\alpha\gamma}_{ik}\frac{\partial u^\gamma_S}{\partial x_k}\frac{\partial v^\alpha}{\partial x_i}\right\rangle+S^{-2}\left\langle u_S \cdot v\right\rangle=-\left\langle a_{ij}^{\alpha\beta}\frac{\partial v^\alpha}{\partial x_i}\right\rangle,
\end{equation}
where $v=(v^\alpha)\in H^1_{{\rm loc}}(\mathbb{R}^{d+1})$ and $v^\alpha,\nabla v^\alpha, \partial_t v^\alpha\in B^2(\mathbb{R}^{d+1})$.
\end{lemma}
\begin{proof}
  Let $1\leq j\leq d$ and $1\leq \beta\leq m$. For any $\phi\in H^1(\mathbb{R}^{d+1})$ with compact support, we have
\begin{equation}\label{6.19-2}
-\int_{\mathbb{R}^{d+1}} \partial_t\phi\cdot u_S+\int_{\mathbb{R}^{d+1}} a^{\alpha\gamma}_{ik}\frac{\partial u^\gamma_S}{\partial x_k}\frac{\partial \phi^\alpha}{\partial x_i}+S^{-2}\int_{\mathbb{R}^{d+1}} u_S \cdot \phi=-\int_{\mathbb{R}^{d+1}}a_{ij}^{\alpha\beta}\frac{\partial \phi^\alpha}{\partial x_i}.
\end{equation}

Suppose that $v=(v^\alpha)\in B^2(\mathbb{R}^{d+1})$ and $\nabla v^\alpha,\partial_t v^\alpha\in B^2(\mathbb{R}^{d+1})$. By choosing $\phi(x,t)=\psi(\varepsilon x, \varepsilon^2 t)v(x,t)$ with $\psi\in C_0^\infty(\R^{d+1})$ in \eqref{6.19-2} and using a change of variables, we obtain
\begin{equation*}
\aligned
&\int_{\mathbb{R}^{d+1}} \left(
-\partial_t v\cdot u_S+a_{ik}^{\alpha\beta}\frac{\partial u_S^{\gamma}}{\partial x_k}\frac{\partial v^{\alpha}}{\partial x_i}
+S^{-2} u_S\cdot v\right)\left(\frac{x}{\varepsilon},\frac{t}{\varepsilon^2}\right)\psi(x,t)\\
&-\varepsilon^2\int_{\mathbb{R}^{d+1}} (v\cdot u_S)\left(\frac{x}{\varepsilon},\frac{t}{\varepsilon^2}\right)\partial_t\psi(x,t)
+\varepsilon\int_{\mathbb{R}^{d+1}} \left(a_{ik}^{\alpha\beta}\frac{\partial u_S^{\gamma}}{\partial x_k}v^{\alpha}\right)\left(\frac{x}{\varepsilon},\frac{t}{\varepsilon^2}\right)\frac{\partial \psi}{\partial x_i}(x,t)\\
&=-\varepsilon\int_{\mathbb{R}^{d+1}}\left(a_{ij}^{\alpha\beta}v^{\alpha}\right)
\left(\frac{x}{\varepsilon},\frac{t}{\varepsilon^2}\right)
\frac{\partial \psi}{\partial x_i}(x,t)
-\int_{\mathbb{R}^{d+1}}\left(a_{ij}^{\alpha\beta}\frac{\partial v^{\alpha}}{\partial x_i}\right)\left(\frac{x}{\varepsilon},\frac{t}{\varepsilon^2}\right)
\psi(x,t).
\endaligned
\end{equation*}
Then the desired result follows by letting $\varepsilon\rightarrow 0$.
\end{proof}

\begin{remark}
Let $v$ be a constant in \eqref{6.19-1}, we obtain that $\langle \chi_{S,j}^\beta \rangle=0$. Let $u_S$ be the same as in Lemma \ref{Lem-6.4}. Choosing a sequence of trigonometric polynomials $u_{S,n}\in {\rm \tilde{T}rig}(\R^{d+1})$ such that
 \begin{equation}\label{6.19-11}
  \|u_{S,n}-u_S\|_{B^2}+ \|\nabla u_{S,n}-\nabla u_S\|_{B^2}+\|\partial_t u_{S,n}-\partial_t u_S\|_{\mathcal{K}^*+B^2}\rightarrow 0
  \end{equation}
as $n\rightarrow \infty$. By taking $u_{S,n}$ as the test function in \eqref{6.19-1}, we obtain
\begin{equation}\label{6.19-1-1}
-\left\langle \partial_t u_{S,n}\cdot u_S\right\rangle+\left\langle a^{\alpha\gamma}_{ik}\frac{\partial u^\gamma_S}{\partial x_k}\frac{\partial u_{S,n}^\alpha}{\partial x_i}\right\rangle+S^{-2}\left\langle u_S \cdot u_{S,n}\right\rangle=-\left\langle a_{ij}^{\alpha\beta}\frac{\partial u_{S,n}^\alpha}{\partial x_i}\right\rangle.
\end{equation}
By \eqref{6.19-11}, we see that
\begin{equation}\label{6.19-1-2}
\left\langle a^{\alpha\gamma}_{ik}\frac{\partial u^\gamma_S}{\partial x_k}\frac{\partial u_{S,n}^\alpha}{\partial x_i}\right\rangle+S^{-2}\left\langle u_S \cdot u_{S,n}\right\rangle
\rightarrow\left\langle a^{\alpha\gamma}_{ik}\frac{\partial u^\gamma_S}{\partial x_k}\frac{\partial u_S^\alpha}{\partial x_i}\right\rangle+S^{-2}\left\langle |u_S|^2\right\rangle
\end{equation}
and
\begin{equation}\label{6.19-1-3}
\left\langle a_{ij}^{\alpha\beta}\frac{\partial u_{S,n}^\alpha}{\partial x_i}\right\rangle
\rightarrow\left\langle a_{ij}^{\alpha\beta}\frac{\partial u_S^\alpha}{\partial x_i}\right\rangle
\end{equation}
as $n\rightarrow \infty$.
To handle the first term in the left-hand side of \eqref{6.19-1-1}, write
\begin{equation}\label{est-u,S,n}
\left\langle \partial_t u_{S,n}\cdot u_S\right\rangle=
\left\langle \partial_t u_{S,n}\cdot (u_S-u_{S,n})\right\rangle+\left\langle \partial_t u_{S,n}\cdot u_{S,n}\right\rangle.
\end{equation}
Since $u_{S,n}\in {\rm \tilde{T}rig}(\R^{d+1})$, by \eqref{skew-symmetric} we obtain that
\begin{equation}\label{est-u,S,n-1}
\left\langle \partial_t u_{S,n}\cdot u_{S,n}\right\rangle=0.
\end{equation}
 Moreover, since $u_S\in \mathcal{K}\cap B^2(\mathbb{R}^{d+1})$ and $u_{S,n}\in {\rm \tilde{T}rig}(\mathbb{R}^{d+1})$, if $\partial_t u_{S,n}=v+h$ with $v\in \mathcal{K}^*$ and $h\in B^2(\mathbb{R}^{d+1})$, then $v=\partial_t u_{S,n}-h\in B^2(\mathbb{R}^{d+1})$ and
\begin{equation}\label{est-u,S,n-2}
\aligned
\left|\left\langle \partial_t u_{S,n}\cdot (u_S-u_{S,n})\right\rangle\right|
&\leq \left|\left\langle v\cdot (u_S-u_{S,n})\right\rangle\right|
+\left|\left\langle h\cdot (u_S-u_{S,n})\right\rangle\right|\\
&=\left|\left( v, u_S-u_{S,n}\right)\right|
+\left|\left\langle h\cdot (u_S-u_{S,n})\right\rangle\right|\\
&\leq\|v\|_{\mathcal{K}^*}\|u_S-u_{S,n}\|_{\mathcal{K}}
+\|h\|_{B^2}\|u_S-u_{S,n}\|_{B^2}\\
&\leq\{\|v\|_{\mathcal{K}^*}+\|h\|_{B^2}\}\{\|u_S-u_{S,n}\|_{\mathcal{K}}
+\|u_S-u_{S,n}\|_{B^2}\}.
\endaligned
\end{equation}
Since \eqref{est-u,S,n-2} holds for any $v\in \mathcal{K}^*$ and $h\in B^2(\mathbb{R}^{d+1})$ such that $\partial_t u_{S,n}=v+h$, it follows that
\begin{equation}\label{est-u,S,n-3}
\aligned
\left|\left\langle \partial_t u_{S,n}\cdot (u_S-u_{S,n})\right\rangle\right|
&\leq\|\partial_t u_{S,n}\|_{\mathcal{K}^*+B^2}\{\| u_S- u_{S,n}\|_{\mathcal{K}}
+\|u_S-u_{S,n}\|_{B^2}\}\\
&\leq C\{\|\nabla u_S-\nabla u_{S,n}\|_{B^2}
+\|u_S-u_{S,n}\|_{B^2}\}\rightarrow 0
\endaligned
\end{equation}
as $n\rightarrow\infty$, where we have used \eqref{6.19-11}. Combining \eqref{est-u,S,n} with \eqref{est-u,S,n-1} and \eqref{est-u,S,n-3}, we obtain $\left\langle \partial_t u_{S,n}\cdot u_S\right\rangle\rightarrow 0$ as $n\rightarrow \infty$. This, together with \eqref{6.19-1-1}-\eqref{6.19-1-3}, gives
\begin{equation}\label{approximate-cor}
\left\langle a^{\alpha\gamma}_{ik}\frac{\partial u^\gamma_S}{\partial x_k}\frac{\partial u_S^\alpha}{\partial x_i}\right\rangle+S^{-2}\left\langle |u_S|^2\right\rangle=-\left\langle a_{ij}^{\alpha\beta}\frac{\partial u_S^\alpha}{\partial x_i}\right\rangle.
\end{equation}
In view of the ellipticity condition \eqref{ellipticity}, this implies that
\begin{equation}\label{est-AC}
\left\langle |\nabla \chi_S|^2\right\rangle+S^{-2}\left\langle |\chi_S|^2\right\rangle\leq C,
\end{equation}
where $C$ depends only on $d$, $m$, and $\mu$.
\end{remark}

\begin{lemma}
As $S\to \infty$,
\begin{equation*}
 \frac{\partial }{\partial x_i}(\chi_{S,j}^{\alpha\beta}) \rightarrow \frac{\partial }{\partial x_i}(\chi_j^{\alpha\beta})~~{\rm weakly~in}~~B^2(\mathbb{R}^{d+1}).
\end{equation*}
\end{lemma}
\begin{proof}
For $1\leq j\leq d$ and $1\leq \beta\leq m$, since $\{\nabla\chi_{S,j}^{\beta}\}$ is bounded in $B^2(\mathbb{R}^{d+1})$, we can assume that $\frac{\partial }{\partial x_i}(\chi_{S,j}^{\alpha\beta}) \rightarrow \phi_{ij}^{\alpha\beta}$ weakly in $B^2(\mathbb{R}^{d+1})$ as $S\rightarrow \infty$ for some $\phi=(\phi_{ij}^{\alpha\beta})\in B^2(\mathbb{R}^{d+1})$. Moreover, since  $\nabla \chi_{S,j}^{\beta}\in V^2_{{\rm pot}}$, where $V^2_{{\rm pot}}$ denotes the closure of potential trigonometric polynomials with mean value zero in $B^2(\mathbb{R}^{d+1})$, we obtain that $\phi\in V^2_{{\rm pot}}$. Hence, $\phi_{ij}^{\alpha\beta}=\frac{\partial \psi_j^{\alpha\beta}}{\partial x_i}$ for some $\psi=(\psi_j^{\alpha\beta})\in \mathcal{K}$.

Let $u_S=\chi_{S,j}^\beta$. Since $S^{-2}\left\langle |\chi_S|^2\right\rangle\leq C$, we see that
\begin{equation}\label{6.19-4}
	\left\langle a^{\alpha\gamma}_{ik}\frac{\partial u^\gamma_S}{\partial x_k}\frac{\partial v^\alpha}{\partial x_i}\right\rangle+S^{-2}\left\langle u_S \cdot v\right\rangle
\rightarrow	\left\langle a^{\alpha\gamma}_{ik}\frac{\partial \psi^{\gamma\beta}_{j}}{\partial x_k}\frac{\partial v^\alpha}{\partial x_i}\right\rangle
\end{equation}
as $S\rightarrow\infty$, for any $v\in {\rm \tilde{T}rig}(\mathbb{R}^{d+1})$.
Moreover, since $u_S\in \mathcal{K}\cap B^2(\mathbb{R}^{d+1})$, we have
$\left\langle \partial_t v\cdot u_S\right\rangle=\left(\partial_t v, u_S\right)$. Note that
$\left(\partial_t v, u_S\right)\rightarrow
\left(\partial_t v, \psi\right)$ as $S\rightarrow \infty$.
This, together with \eqref{6.19-1} and \eqref{6.19-4}, gives
\begin{equation}\label{psi-equ}
-\left(\partial_t v, \psi\right)+\left\langle a^{\alpha\gamma}_{ik}\frac{\partial \psi^{\gamma\beta}_{j}}{\partial x_k}\frac{\partial v^\alpha}{\partial x_i}\right\rangle
=-\left\langle a_{ij}^{\alpha\beta}\frac{\partial v^\alpha}{\partial x_i}\right\rangle.
\end{equation}
By \eqref{time-der} and \eqref{est-AC}, we know that $\|\partial_t \chi_S\|_{\mathcal{K}^*+B^2}
\leq C$ for any $S\geq 1$. Hence $\partial_t \chi_{S}\rightarrow w$ weakly$^*$ in $\mathcal{K}^*+B^2(\mathbb{R}^{d+1})$ for some $w=(w_{j}^{\alpha\beta})\in \mathcal{K}^*+B^2(\mathbb{R}^{d+1})$. Since $\frac{\partial }{\partial x_i}(\chi_{S,j}^{\alpha\beta}) \rightarrow \frac{\partial \psi^{\alpha\beta}_j}{\partial x_i}$ weakly in $B^2(\mathbb{R}^{d+1})$, we obtain that $w=\partial_t \psi$.

In view of \eqref{est-AC}, we know that $S^{-2}\chi_S\rightarrow 0$ in $B^2(\mathbb{R}^{d+1})$ as $S\rightarrow\infty$. This, together with \eqref{6.10}, implies that
$\partial_t \chi_S-{\rm div}[A(\nabla \chi_S+\nabla P)]=-S^{-2}\chi_S\rightarrow 0$ in $B^2(\mathbb{R}^{d+1})$. Consequently, $\partial_t \psi={\rm div}[A(\nabla \psi+\nabla P)]\in \mathcal{K}^*$, which, together with (\ref{psi-equ}), implies that $\psi\in D(\partial_t)$ is a solution to
\begin{equation*}
(\partial_s +\mathcal{L}_1)(\psi_j^\beta)=-\mathcal{L}_1(P^\beta_j)~~{\rm in}~~\R^{d+1}.
\end{equation*}
By the uniqueness, we obtain that $\psi_j^\beta=\chi_j^{\beta}$. This completes the proof.
\end{proof}

\begin{theorem}
As $S\to \infty$, $S^{-2}\langle |\chi_S|^2\rangle \to 0$ and $\|\nabla\chi-\nabla\chi_S\|_{B^2}\to 0$.
\end{theorem}
\begin{proof}
It follows from \eqref{ellipticity} that
\begin{align}\label{6.19-6}
\mu\left\langle |\nabla \chi-\nabla\chi_S|^2\right\rangle
&\leq \left
\langle A(\nabla \chi-\nabla \chi_S)\cdot(\nabla \chi-\nabla \chi_S)\right\rangle\nonumber
\\&=\left\langle A\nabla \chi\cdot\nabla \chi\right\rangle
+\left\langle A\nabla \chi_S\cdot\nabla \chi_S \right\rangle
-\left\langle A\nabla \chi \cdot\nabla \chi_S\right\rangle-
\left\langle A\nabla \chi_S\cdot\nabla \chi\right\rangle.
\end{align}
In view of (\ref{corrector}), we have
\begin{equation*}
\left(\partial_t \chi, \chi_S\right)
+\left\langle A\nabla \chi\cdot\nabla\chi_S\right\rangle
=-\left\langle A\nabla P\cdot\nabla\chi_S\right\rangle.
\end{equation*}
Note that
\begin{equation*}
\left(\partial_t \chi, \chi_S\right)
=\left(\partial_t \chi, \chi\right)
+\left(\partial_t \chi, \chi_S-\chi\right).
\end{equation*}
By \eqref{skew-symmetric} we know that $\left(\partial_t \chi, \chi\right)=0$. Hence we obtain
\begin{equation}\label{6.19-6-1}
\left\langle A\nabla \chi\cdot\nabla\chi_S\right\rangle
=-\left\langle A\nabla P\cdot\nabla\chi_S\right\rangle-\left(\partial_t \chi, \chi_S-\chi\right).
\end{equation}
 Moreover, it follows from \eqref{approximate-cor} that
\begin{equation}\label{6.19-6-2}
\left\langle A\nabla \chi_S\cdot\nabla\chi_S\right\rangle=
-\left\langle A\nabla P\cdot\nabla\chi_S\right\rangle-S^{-2}\left\langle|\chi_S|^2\right\rangle.
\end{equation}
Combining (\ref{6.19-6}), (\ref{6.19-6-1}) with (\ref{6.19-6-2}), we see that
\begin{equation}\label{6.19-6-3}
\mu\left\langle |\nabla \chi-\nabla\chi_S|^2\right\rangle+S^{-2}\left\langle|\chi_S|^2\right\rangle
\leq \left\langle A(\nabla \chi-\nabla\chi_S)\cdot\nabla\chi\right\rangle+\left(\partial_t \chi, \chi_S-\chi\right).
\end{equation}
Moreover, since $\nabla\chi_S\rightarrow \nabla\chi$ weakly in $B^2(\mathbb{R}^{d+1})$, we have $\chi_S\rightarrow\chi$ weakly in $\mathcal{K}$, therefore $\left(\partial_t \chi, \chi_S-\chi\right)\rightarrow 0$ as $S\rightarrow \infty$. Hence by letting $S\to \infty$, \eqref{6.19-6-3} implies that $S^{-2}\langle |\chi_S|^2\rangle \to 0$ and $\|\nabla\chi-\nabla\chi_S\|_{B^2}\to 0$ as $S\to \infty$.
\end{proof}

\begin{remark}\label{6.19-10}
  For $S>0$, we define the approximate homogenized coefficients as
  \begin{equation*}
  \hat{a}_{S,ij}^{\alpha\beta}=\langle a_{ij}^{\alpha\beta} \rangle+\left\langle a_{ik}^{\alpha\gamma} \frac{\partial}{\partial x_k}(\chi_{S,j}^{\gamma\beta})\right\rangle.
  \end{equation*}
  Then
  \begin{equation*}
  |\hat{a}_{ij}^{\alpha\beta}-\hat{a}_{S,ij}^{\alpha\beta}|=\left|\left\langle a_{ik}^{\alpha\gamma} \left(\frac{\partial}{\partial x_k}(\chi_j^{\gamma\beta})-\frac{\partial}{\partial x_k}(\chi_{S,j}^{\gamma\beta})\right)\right\rangle\right|\leq C\|\nabla\chi-\nabla\chi_S\|_{B^2}.
  \end{equation*}

\end{remark}

\begin{lemma}\label{6.20}
For $S\geq 1$, we have
\begin{equation}\label{est-6.20}
\|\chi_S\|_{L^\infty(\R^{d+1})}\leq CS.
\end{equation}
Moreover, for any $(x,t),(y,s)\in\mathbb{R}^{d+1}$ and $\sigma\in(0,1)$,
\begin{equation}\label{6.21}
|\chi_S(x,t)-\chi_S(y,s)|\leq CS^{1-\sigma}(|x-y|+|t-s|^{1/2})^\sigma,
\end{equation}
where $C$ depends only on $d$, $m$, $\sigma$ and $\mu$.
\end{lemma}
\begin{proof}
  Let $1\leq j\leq d$ and $1\leq \beta\leq m$. Fix $(z,\tau)\in \R^{d+1}$ and consider
  $$u(x,t)=\chi_{S,j}^\beta(x,t)+P_j^\beta(x-z).$$
Since
\begin{equation*}
(\partial_s+\mathcal{L}_1) (\chi^\beta_{S,j})+S^{-2}\chi^\beta_{S,j}=-\mathcal{L}_1(P^\beta_j)~~{\rm in}~~\R^{d+1},
\end{equation*}
we see that
 \begin{equation}\label{6.23}
(\partial_s+\mathcal{L}_1)(u)=-S^{-2}\chi^\beta_{S,j}.
\end{equation}
By Corollary \ref{6.9}, we have
  \begin{equation*}
\bigg(\dashint_{Q_{2S}(z,\tau)}|u|^2\bigg)^{1/2}\leq CS.
\end{equation*}
Hence by Lemma \ref{6.15} with an iteration argument, we obtain
\begin{equation*}
\bigg(\dashint_{Q_{2S}(z,\tau)}|u|^p\bigg)^{1/p}\leq CS
\end{equation*}
for any $2<p<\infty$. This, together with the H\"older estimate, gives
 \begin{equation}\label{6.26}
\|u\|_{L^\infty(Q_S(z,\tau))}\leq C S
\end{equation}
for any $(z,\tau)\in \mathbb{R}^{d+1}$. Estimate (\ref{est-6.20}) follows from \eqref{6.26} directly. Finally, \eqref{6.21} follows from \eqref{6.26} and the H\"older estimate.
\end{proof}
\begin{lemma}
Let $\sigma_1,\sigma_2\in(0,1)$ and $2<p<\infty$. Then for any $1\leq r\leq S$,
\begin{equation}\label{6.37}
\sup_{(x,t)\in\R^{d+1}}\bigg(\dashint_{Q_r(x,t)}|\nabla\chi_S|^p\bigg)^{1/p}\leq C S^{\sigma_1}r^{-\sigma_2} S^{\sigma_2},
\end{equation}
where $C$ depends only on $p,\sigma_1,\sigma_2$ and $A$.
\end{lemma}
\begin{proof}
  Let $1\leq j\leq d$ and $1\leq \beta\leq m$. Fix $(z,\tau)\in \R^{d+1}$ and consider
  $$u(x,t)=\chi_{S,j}^\beta(x,t)+P_j^\beta(x-z).$$
Then by Caccioppoli's inequality we have
\begin{equation}\label{6.28}
\dashint_{Q_r(z,\tau)}|\nabla u|^2\leq C r^{-2} \dashint_{Q_{2r}(z,\tau)}|u-u(z,\tau)|^2+Cr^2\|S^{-2}\chi_S\|^2_{L^\infty},
\end{equation}
where $0<r\leq S$. It follows from Lemma \ref{6.20} that
\begin{align}\label{6.28-1}
\sup_{(z,\tau)\in\R^{d+1}}\bigg(\dashint_{Q_r(z,\tau)}|\nabla\chi_S|^2\bigg)^{1/2}&\leq C r^{-1} S^{1-\sigma}r^\sigma+CrS^{-2}S+C\nonumber
\\&=C r^{\sigma-1}S^{1-\sigma}+CrS^{-1}+C\nonumber
\\&\leq C\left(\frac{S}{r}\right)^{\sigma_1},
\end{align}
where $\sigma_1=1-\sigma$. Since $A$ is uniformly H\"older continuous, by the local $W^{1,p}$ estimate for parabolic systems, in view of (\ref{6.23}), for any $2<p<\infty$ and $(z,\tau)\in\R^{d+1}$,
\begin{equation*}
\bigg(\dashint_{Q_1(z,\tau)}|\nabla u|^p\bigg)^{1/p}\leq C \bigg(\dashint_{Q_2(z,\tau)}|\nabla u|^2\bigg)^{1/2}+C\|S^{-2}\chi_S\|_{L^\infty}.
\end{equation*}
Hence for any $S\geq 1$,
\begin{equation*}
\sup_{(z,\tau)\in\R^{d+1}}\bigg(\dashint_{Q_1(z,\tau)}|\nabla \chi_S|^p\bigg)^{1/p}\leq C S^\sigma+CS^{-1}\leq C S^\sigma,
\end{equation*}
where we have used \eqref{est-6.20}. Consequently, for $1\leq r\leq S$,
\begin{equation*}
\sup_{(z,\tau)\in\R^{d+1}}\bigg(\dashint_{Q_r(z,\tau)}|\nabla \chi_S|^p\bigg)^{1/p}\leq C S^\sigma.
\end{equation*}
This, together with \eqref{6.28-1} and an interpolation argument, gives \eqref{6.37}.
\end{proof}

\begin{theorem}\label{6.32}
Let $S\geq 1$ and $\chi_S$ be the approximate corrector.  Then for any $(y,s), (z,\tau)\in\mathbb{R}^{d+1}$,
\begin{equation*}
\|\chi_S(\cdot+y, \cdot+s)-\chi_S(\cdot+z, \cdot+\tau)\|_{L^\infty(\R^{d+1})}\leq C S \|A(\cdot+y, \cdot+s)-A(\cdot+z, \cdot+\tau)\|_{L^\infty(\R^{d+1})},
\end{equation*}
where $C$ is independent of $(y,s), (z,\tau)$ and $S$.
\end{theorem}
\begin{proof}
Let $u(\lambda,\theta)=\chi^\beta_{S,j}(\lambda+y, \theta+s)$, $v(\lambda,\theta)=\chi^\beta_{S,j}(\lambda+z, \theta+\tau)$ and $w=u-v$. Then direct computation shows that
\begin{align}\label{6.33}
&\partial_t w-\text{div}(A(\lambda+y, \theta+s)\nabla w)\nonumber
\\&~~~=-S^{-2}w+\text{div}\{[A(\lambda+y, \theta+s)-A(\lambda+z, \theta+\tau)](\nabla P^\beta_j+\nabla v)\}.
\end{align}
Let $Q=Q_S(x_0,t_0)$ and $\psi\in C^\infty_0(2Q)$ such that $\psi=1$ in $\frac{3}{2}Q$ and $|\nabla\psi|^2+|\partial_t\psi|\leq CS^{-2}$.  (\ref{6.33}) implies that
\begin{align*}
(\partial_t-\text{div}(A(\lambda+y, \theta+s)\nabla))(\psi w)&=(\partial_t \psi)w-S^{-2}w\psi\nonumber
\\&~~-[A(\lambda+y, \theta+s)-A(\lambda+z, \theta+\tau)] (\nabla P_j^\beta+\nabla v) \cdot \nabla \psi\nonumber
\\&~~+\text{div}\{[A(\lambda+y, \theta+s)-A(\lambda+z, \theta+\tau)](\nabla P_j^\beta+\nabla v) \psi\}\nonumber
\\&~~-\text{div}\{A(\lambda+y, \theta+s)w\nabla\psi\}-A(\lambda+y, \theta+s)\nabla w\cdot \nabla \psi.
\end{align*}
Hence for any $(x,t)\in Q$, by using the representation formula, we have
\begin{align*}
|w(x,t)|\leq& CS^{-2}\int_{2Q}\left|\Gamma(x,t;\lambda,\theta)w(\lambda,\theta)\right|\,d\lambda d\theta\nonumber
\\&+\int_{2Q}\left|[A(\lambda+y, \theta+s)-A(\lambda+z, \theta+\tau)](\nabla P_j^\beta+\nabla v)\cdot \nabla_\lambda (\psi \Gamma(x,t;\lambda,\theta))\right|\,d\lambda d\theta\nonumber
\\&+\int_{2Q}\left|A(\lambda+y, \theta+s) w \nabla \psi \cdot \nabla_\lambda \Gamma(x,t;\lambda,\theta) \right|\,d\lambda d\theta
\\&+\int_{2Q}\left|A(\lambda+y, \theta+s) \nabla w\cdot  \nabla \psi  \Gamma(x,t;\lambda,\theta)\right|\,d\lambda d\theta.
\end{align*}
As a consequence, in view of \eqref{f-6} and \eqref{f-8} we obtain
\begin{align}\label{6.36}
|w(x,t)|\leq &CS^{-2}\int_{2Q}\left|\Gamma(x,t;\lambda,\theta)w(\lambda,\theta)\right|\,d\lambda d\theta\nonumber
\\&+C\|A(\cdot+y, \cdot+s)-A(\cdot+z, \cdot+\tau)\|_{L^\infty(\R^{d+1})}\int_{2Q}\left|\nabla(\psi\Gamma) \right| \nonumber
\\&+C\|A(\cdot+y, \cdot+s)-A(\cdot+z, \cdot+\tau)\|_{L^\infty(\R^{d+1})}\int_{2Q}|\nabla v||\nabla(\psi\Gamma)| \nonumber
\\&+C\bigg(\dashint_{2Q}|w|^2\bigg)^{1/2}+CS\bigg(\dashint_{2Q}|\nabla w|^2\bigg)^{1/2}.
\end{align}

Notice that by Lemma \ref{6.11} the last two terms in the right-hand side of (\ref{6.36}) are bounded by $CS\|A(\cdot+y, \cdot+s)-A(\cdot+z, \cdot+\tau)\|_{L^\infty(\R^{d+1})}$. And in view of \eqref{f-6} and \eqref{f-8} we know that the
second term in the right-hand side of (\ref{6.36}) is also bounded by $CS\|A(\cdot+y, \cdot+s)-A(\cdot+z, \cdot+\tau)\|_{L^\infty(\R^{d+1})}$.

To treat the third term in the right-hand side of (\ref{6.36}), we note that
\begin{align*}
&C\int_{2Q}|\nabla v||\nabla(\Gamma\psi)| \nonumber
\\&\leq C\sum_{\ell=0}^{\infty}\bigg(\dashint_{|\lambda-x|+|\theta-t|^{1/2}\sim 2^{-\ell}S}|\nabla v|^2\bigg)^{1/2}\bigg(\dashint_{|\lambda-x|+|\theta-t|^{1/2}\sim 2^{-\ell}S}|\nabla (\Gamma\psi)|^2\bigg)^{1/2}(2^{-\ell}S)^{d+2}\nonumber
\\&\leq C\sum_{\ell=0}^{\infty}(2^\ell)^\sigma (2^{-\ell}S)^{-1-d} (2^{-\ell}S)^{d+2}\nonumber
\\&=C\sum_{\ell=0}^{\infty}(2^\ell)^{\sigma-1}S\nonumber
\\&\leq CS,
\end{align*}
where we have used \eqref{6.28-1}, \eqref{f-6} and \eqref{f-8} in the second inequality.
As a result we have proved that
\begin{equation*}
|w(x,t)| \leq C S^{-2}\int_{2Q}\frac{|w(\lambda,\theta)|\,d\lambda d\theta}{(|\lambda-x|+|\theta-t|^{1/2})^d}+C S \|A(\cdot+y, \cdot+s)-A(\cdot+z, \cdot+\tau)\|_{L^\infty(\R^{d+1})}.
\end{equation*}
In view of the fractional integral estimates, this implies that
\begin{align*}
\bigg(\dashint_{Q} |w|^q\bigg)^{1/q}\leq C \bigg(\dashint_{2Q} |w|^p\bigg)^{1/p}+C S \|A(\cdot+y, \cdot+s)-A(\cdot+z, \cdot+\tau)\|_{L^\infty(\R^{d+1})},
\end{align*}
where $1<p<q\leq \infty$ and $\frac{1}{p}-\frac{1}{q}\leq\frac{2}{d+2}$.

Since
\begin{align*}
\bigg(\dashint_{2Q} |w|^2\bigg)^{1/2}\leq C S \|A(\cdot+y, \cdot+s)-A(\cdot+z, \cdot+\tau)\|_{L^\infty(\R^{d+1})}
\end{align*}
by Lemma \ref{6.11}, this, together with a simple iteration argument, gives
\begin{align*}
\|w\|_{L^\infty(Q)}\leq C S \|A(\cdot+y, \cdot+s)-A(\cdot+z, \cdot+\tau)\|_{L^\infty(\R^{d+1})}.
\end{align*}
Thus we complete the proof.
\end{proof}

Let $\rho(R)$ be defined by (\ref{AP-1}). For any $S\geq 1$ and $\sigma \in(0,1]$, recall that
\begin{equation*}
\Theta_\sigma (S)=\inf_{0<R\leq S}\bigg\{\rho(R)+\bigg(\frac{R}{S}\bigg)^\sigma\bigg\}
\end{equation*}
and notice that $\Theta_\sigma (S)$ is continuous and decreasing with respect to $S$. Moreover, $\Theta_\sigma (S)$ satisfies $\Theta_\sigma (S)\leq C_\sigma [\Theta_1 (S)]^\sigma$, and $\Theta_\sigma (S)\to 0$ as $S\to \infty$.

\begin{lemma}\label{6.42}
Let $S\geq 1$. Then
\begin{equation}\label{est-6.42}
\|\chi_S\|_{L^\infty(\R^{d+1})}\leq CS \Theta_\sigma (S)
\end{equation}
for any $\sigma \in(0,1]$.
\end{lemma}
\begin{proof}
  Let $Y=(y,s), Z=(z,\tau)\in \R^{d+1}$. Suppose that $|Z|\leq R$. Then it follows from Lemma \ref{6.20} and Theorem \ref{6.32} that
  \begin{align*}
  |\chi_S(Y)-\chi_S(0)|&\leq  |\chi_S(Y)-\chi_S(Z)|+ |\chi_S(Z)-\chi_S(0)|\nonumber
  \\&\leq CS \|A(\cdot+Y)-A(\cdot+Z)\|_{L^\infty(\R^{d+1})}+C S^{1-\sigma}R^\sigma,
  \end{align*}
 which gives
\begin{equation}\label{6.42-1}
\sup_{Y\in\R^{d+1}}|\chi_S(Y)-\chi_S(0)|\leq CS\bigg\{\rho(R)+\bigg(\frac{R}{S}\bigg)^\sigma\bigg\}.
\end{equation}

 Observe that
\begin{align*}
|\chi_S(0)|&\leq \left|\dashint_{B(0,L)}(\chi_S(Y)-\chi_S(0))\right|+\left|\dashint_{B(0,L)}\chi_S(Y)\right|\\&
\leq \sup_{Y\in\R^{d+1}}\left|\chi_S(Y)-\chi_S(0)\right|+\left|\dashint_{B(0,L)}\chi_S(Y)\right|.
\end{align*}
Since $\langle \chi_S\rangle=0$, by letting $L\to\infty$ we obtain that
\begin{equation}\label{6.42-2}
|\chi_S(0)|\leq  \sup_{Y\in\R^{d+1}}|\chi_S(Y)-\chi_S(0)|.
\end{equation}
Combining \eqref{6.42-1} and \eqref{6.42-2}, we obtain \eqref{est-6.42}.
\end{proof}

Let $\chi_S=(\chi_{S,j}^{\alpha\beta})$ be the approximate corrector defined by (\ref{6.10}). To introduce the dual approximate correctors, we consider the matrix-valued function
$$
b_S=A+A\nabla \chi_S-\hat{A}.
$$
More precisely, $b_S=(b_{S,ij}^{\alpha\beta})$ where $1\leq i, j \leq d$, $1\leq \alpha, \beta\leq m$, and
\begin{equation*}
b_{S,ij}^{\alpha\beta}=
a_{ij}^{\alpha\beta}+a_{ik}^{\alpha\gamma}\frac{\partial}{\partial y_k}\chi^{\gamma\beta}_{S,j}-\hat{a}^{\alpha\beta}_{ij},
\end{equation*}
for $i=d+1$, we define
\begin{equation*}
b^{\alpha\beta}_{S,(d+1)j}=-\chi^{\alpha\beta}_{S,j}.
\end{equation*}

Next, let $h(y,s)=h_S(y,s)=\langle b_S \rangle-b_S(y,s)$, we proceed to introduce $f_S=(f_{S,ij}^{\alpha\beta})\in H^2(\R^{d+1})$, which is a matrix-valued function and solves
\begin{equation}\label{7.34}
  -\Delta_{d+1} f_S+S^{-2}f_S =h ~~~{\rm in} ~~ \mathbb{R}^{d+1}.
\end{equation}

By the definition of $\chi_S$, we obtain
\begin{equation*}
  \frac{\partial}{\partial y_i}h_{ij}^{\alpha\beta}=-\sum_{i=1}^{d+1}\frac{\partial}{\partial y_i}b_{S,ij}^{\alpha\beta}=-\sum_{i=1}^{d}\frac{\partial}{\partial y_i}b_{S,ij}^{\alpha\beta}+\partial_s\chi_{S,j}^{\alpha\beta}=-S^{-2}\chi_{S,j}^{\alpha\beta}.
\end{equation*}
It follows that
\begin{equation*}
  -\Delta_{d+1} \frac{\partial}{\partial y_i}f_{S,ij}^{\alpha\beta}+S^{-2}\frac{\partial}{\partial y_i}f_{S,ij}^{\alpha\beta} =-S^{-2}\chi_{S,j}^{\alpha\beta}.
\end{equation*}

\begin{theorem}\label{7.35}
  Assume that $f=(f_{S,ij}^{\alpha\beta})\in H^2(\mathbb{R}^{d+1})$ solves (\ref{7.34}). Then
\begin{equation*}
  \sup_{(x,t)\in \mathbb{R}^{d+1}}\bigg(\dashint_{Q_S(x,t)}|\nabla^2 f_S|^2\bigg)^{1/2}\leq C,
\end{equation*}
where $C$ depends only on $d$.
\end{theorem}
\begin{proof}
By using \eqref{6.28-1}, the proof follows from the same line of argument as in \cite{Shen-2014}.
\end{proof}

\begin{lemma}\label{7.47}
Assume that $f_S=(f_{S,ij}^{\alpha\beta})\in H^2(\mathbb{R}^{d+1})$ solves (\ref{7.34}). Then
\begin{equation*}
\|f_S\|_{L^\infty(\R^{d+1})}\leq C S^2 \Theta_1 (S)
\end{equation*}
and
\begin{equation*}
\|\nabla f_S\|_{L^\infty(\R^{d+1})}\leq C S \Theta_\sigma (S)
\end{equation*}
as well as
\begin{equation*}
\bigg\|\nabla \frac{\partial}{\partial x_i}f_{S,ij}^{\alpha\beta}\bigg\|_{L^\infty(\R^{d+1})}\leq C \Theta_\sigma (S)
\end{equation*}
for any $S\geq 1$ and $\sigma \in(0,1)$.
\end{lemma}
\begin{proof}
  This proof follows from the same line of argument as in \cite{Shen-2014}.
\end{proof}

\section{\bf Convergence rates and proof of Theorem \ref{rate}}
In this section, we investigate convergence rates for initial-Dirichlet problem and establish the main results of this paper. For $1\leq k,i\leq d+1$ and $1\leq j\leq d$, let
\begin{equation*}
 \phi_{S,kij}^{\alpha\beta}=\frac{\partial}{\partial x_k}f_{S,ij}^{\alpha\beta}-\frac{\partial}{\partial x_i}f_{S,kj}^{\alpha\beta}
\end{equation*}
be the approximate flux corrector. One may notice that $\phi^{\alpha\beta}_{S,kij}$ is skew-symmetric by the definition.

 We shall consider the function
$w_\varepsilon=(w_\varepsilon^\alpha)$, defined by
\begin{equation}\label{w}
\aligned
w_\varepsilon^\alpha (x, t)  = u_\varepsilon^\alpha (x, t) -u_0^\alpha  (x, t) -\varepsilon \chi_{S,j}^{\alpha\beta} (x/\varepsilon, t/\varepsilon^2)
K_\varepsilon \left(\frac{\partial u_0^\beta}{\partial x_j}\right)\\
-\varepsilon^2 \phi_{S,(d+1) ij}^{\alpha\beta} (x/\varepsilon, t/\varepsilon^2)
\frac{\partial}{\partial x_i} K_\varepsilon \left(\frac{\partial u_0^\beta}{\partial x_j}\right),
\endaligned
\end{equation}
where $K_\varepsilon : L^2(\Omega_T) \to C_0^\infty(\Omega_T)$ is a linear operator to be chosen later and the repeated indices $i, j$ in (\ref{w}) are summed from $1$ to $d$.

We begin by introducing a parabolic  smoothing operator.
Fix $\theta=\theta (y,s)=\theta_1(y)\theta_2(s)$, where $\theta_1\in C_0^\infty(B(0,1))$, $\theta_2\in C_0^\infty(-1,1)$, $\theta_1,\theta_2\geq 0$, and $\int_{\mathbb{R}^{d}}\theta_1(y)\,dy=\int_{\mathbb{R}}\theta_2(s)\,ds=1$. Define
\begin{equation}\label{1.18}
\aligned
S_\varepsilon(f)(x,t)&=\frac{1}{\varepsilon^{d+2}}\int_{\mathbb{R}^{d+1}}f(x-y,t-s)
\theta(y/\varepsilon,s/\varepsilon^2)\, dyds
\\&=\int_{\mathbb{R}^{d+1}}f(x-\varepsilon y,t-\varepsilon^2 s)\theta(y,s)\, dyds.
\endaligned
\end{equation}

\begin{lemma}\label{lemma-S-1}
Let $S_\varepsilon$ be defined as in (\ref{1.18}). Then
\begin{align*}
\Vert S_\varepsilon (f)\Vert_{L^2({\mathbb{R}^{d+1}})}\leq \Vert f \Vert_{L^2({\mathbb{R}^{d+1}})},
\end{align*}
\begin{align*}
\varepsilon\, \Vert \nabla S_\varepsilon (f)\Vert_{L^2({\mathbb{R}^{d+1}})}
+\varepsilon^2 \Vert\nabla^2 S_\varepsilon (f)\Vert_{L^2(\mathbb{R}^{d+1})} \leq C\, \Vert  f \Vert_{L^2({\mathbb{R}^{d+1}})},
\end{align*}
\begin{align*}
\varepsilon^2 \Vert \partial_t S_\varepsilon (f)\Vert_{L^2({\mathbb{R}^{d+1}})}\leq C\, \Vert f \Vert_{L^2({\mathbb{R}^{d+1}})},
\end{align*}
where $C$ depends only on $d$.
\end{lemma}

\begin{proof}
See \cite{GS-2017}.
\end{proof}

\begin{lemma}\label{lemma-S-2}
Let $S_\varepsilon$ be defined as in (\ref{1.18}). Then
\begin{align*}
\| \nabla S_\varepsilon (f) -\nabla f \|_{L^2(\mathbb{R}^{d+1})}
\leq C \varepsilon \Big\{ \Vert \nabla^2 f \Vert_{L^2(\mathbb{R}^{d+1})}
+\|\partial_t f \|_{L^2(\mathbb{R}^{d+1})} \Big\},
\end{align*}
where $C$ depends only on $d$.
\end{lemma}

\begin{proof}
See \cite{GS-2017}.
\end{proof}

\begin{lemma}\label{lemma-S-3}
Let $S_\varepsilon$ be defined as in (\ref{1.18}). Let $g=g(y,s)\in L^p(\mathbb{R}^{d+1})$. Then
\begin{equation}\label{S-approx}
\aligned
\|g^\varepsilon S_\varepsilon(f) \|_{L^p(\mathbb{R}^{d+1})}
&\leq C \sup_{(x,t)\in \mathbb{R}^{d+1}}\bigg(\dashint_{Q_1(x,t)}|g|^p\bigg)^{1/p}\| f \|_{L^p(\mathbb{R}^{d+1})},\\
\|g^\varepsilon \nabla S_\varepsilon(f) \|_{L^p(\mathbb{R}^{d+1})}
&\leq C \varepsilon^{-1}\sup_{(x,t)\in \mathbb{R}^{d+1}}\bigg(\dashint_{Q_1(x,t)}|g|^p\bigg)^{1/p}\| f \|_{L^p(\mathbb{R}^{d+1})}
\endaligned
\end{equation}
for any $1\leq p<\infty$, where $g^\varepsilon=g(x/\varepsilon,t/\varepsilon^2)$ and $C$ depends only on $d$.
\end{lemma}

\begin{proof}
Note that $S_\varepsilon(f)(x,t)=S_1(f_\varepsilon)(\varepsilon^{-1}x,\varepsilon^{-2}t)$, where $f_\varepsilon(x,t)=f(\varepsilon x,\varepsilon^2 t)$. Consequently, it suffices for us to show the case $\varepsilon=1$ by a change of variable. It follows from H\"older's inequality we have that
\begin{equation*}
  |S_1(f)(x,t)|^p\leq \int_{\mathbb{R}^{d+1}}|f(y,s)|^p \theta(x-y,t-s)dyds,
\end{equation*}
where we have use the fact that $\int_{\mathbb{R}^{d+1}}\theta =1$.

By using Fubini's Theorem we have that
\begin{align*}
\int_{\mathbb{R}^{d+1}}|g(x,t)|^p|S_1(f)(x,t)|^p dx dt \leq C \sup_{(y,s)\in \mathbb{R}^{d+1}}\int_{Q_1(y,s)}|g(x,t)|^p dx dt \int_{\mathbb{R}^{d+1}}|f(y,s)|^p dyds.
\end{align*}
This gives the proof of the first estimate in (\ref{S-approx}). The proof of the second estimate in (\ref{S-approx}) is similar and we omit it.
\end{proof}

Let $\delta>\varepsilon$ be a small parameter to be determined later.
Choose  $\eta_1 \in C_0^\infty(\Omega)$ such that $0\leq \eta_1\leq 1$,
$\eta_1 (x)=1$ if dist$(x, \partial\Omega)\geq 2\delta$,
$\eta_1 (x)=0$ if dist$(x, \partial\Omega)\leq \delta$, and
$|\nabla_x \eta_1|\leq C \delta^{-1}$.
Similarly, we choose $\eta_2\in C_0^\infty(0, T)$ such that $0\leq \eta_2\leq 1$,
$\eta_2 (t) =1$ if $2\delta^2 \leq  t\le T-2\delta^2 $,
$\eta_2 (t)=0$ if $t\leq \delta^2 $ or $t>T-\delta^2 $,
and $| \eta_2^\prime (t)|\leq C \delta^{-2}$. Set $K_\varepsilon(f)(x,t)=S_\varepsilon(\eta_1\eta_2 f)(x,t)$.

\begin{theorem}\label{Theorem-2.1}
Let $\Omega$ be a bounded Lipschitz domain in $\mathbb{R}^d$ and
$0<T<\infty$.
Assume that $u_\varepsilon \in L^2(0, T; H^1(\Omega))$ and $u_0\in L^2(0, T; H^2(\Omega))$ are solutions of
the initial-Dirichlet problems (\ref{IDP}) and (\ref{IDP-0}), respectively.
Let $w_\varepsilon$ be defined by (\ref{w}).
Then we have
\begin{equation*}
\aligned
  & \int_0^T  \big\langle (\partial_t +\mathcal{L}_\varepsilon) w_\varepsilon,
 w_\varepsilon \big\rangle_{H^{-1}(\Omega) \times H^1_0(\Omega)} \, dt  \\&= \iint_{\Omega_T}(\widehat{a}_{ij} -a_{ij}^\varepsilon )
\bigg( \frac{\partial u_0}{\partial x_j} - K_\varepsilon \bigg( \frac{\partial u_0}{\partial x_j}\bigg) \bigg)\frac{\partial w_\varepsilon}{\partial x_i}\\
&~~~~~~~+\varepsilon^{-1} S^{-2}\iint_{\Omega_T}\chi_{S,j}^{\varepsilon}K_\varepsilon\bigg(\frac{\partial u_0}{\partial x_j}\bigg)w_\varepsilon\\
&~~~~~~~~~~+ \iint_{\Omega_T}b^\varepsilon_{S,ij}\frac{\partial}{\partial x_i}K_\varepsilon\left(\frac{\partial u_0}{\partial x_j}\right)w_\varepsilon
-\varepsilon\iint_{\Omega_T}\chi^\varepsilon_{S,j}\partial_tK_\varepsilon\left(\frac{\partial u_0}{\partial x_j}\right)w_\varepsilon\\
&~~~~~~~~~~~~-\varepsilon \iint_{\Omega_T}a_{ij}^\varepsilon \chi^\varepsilon_{S,k} \frac{\partial}{\partial x_ j}K_\varepsilon\left(\frac{\partial u_0}{\partial x_k}\right)\frac{\partial w_\varepsilon}{\partial x_i}\\
&~~~~~~~~~~~~~~-\varepsilon^2\iint_{\Omega_T}(\partial_t +\mathcal{L}_\varepsilon)\bigg\{
\phi^\varepsilon_{S,(d+1)ij}\frac{\partial}{\partial x_i}K_\varepsilon\left(\frac{\partial u_0}{\partial x_j}\right)\bigg\}w_\varepsilon,
\endaligned
\end{equation*}
where we have suppressed superscripts $\alpha, \beta$ for the simplicity of presentation.
The repeated indices  $i, j, k$ are summed from  $1$ to $d$.
\end{theorem}

\begin{proof}
In view of (\ref{IDP}) and (\ref{IDP-0}), we have
\begin{equation*}
\aligned
  (\partial_t +\mathcal{L}_\varepsilon) w_\varepsilon
&=(\mathcal{L}_0-\mathcal{L}_\varepsilon)u_0-(\partial_t +\mathcal{L}_\varepsilon)\bigg\{\varepsilon
\chi^\varepsilon_{S,j}K_\varepsilon\bigg(\frac{\partial u_0}{\partial x_j}\bigg)\bigg\}\\
&~~~-(\partial_t +\mathcal{L}_\varepsilon)\bigg\{\varepsilon^2
\phi^\varepsilon_{S,(d+1)ij}\frac{\partial}{\partial x_i}K_\varepsilon\bigg(\frac{\partial u_0}{\partial x_j}\bigg)\bigg\}\\
&= -\frac{\partial}{\partial x_i}\bigg\{(\widehat{a}_{ij} -a_{ij}^\varepsilon )
 \frac{\partial u_0}{\partial x_j} \bigg\}\\
&~~~~~-(\partial_t +\mathcal{L}_\varepsilon)\bigg\{\varepsilon
\chi^\varepsilon_{S,j}K_\varepsilon\left(\frac{\partial u_0}{\partial x_j}\right)\bigg\}\\
&~~~~~-(\partial_t +\mathcal{L}_\varepsilon)\bigg\{\varepsilon^2
\phi^\varepsilon_{S,(d+1)ij}\frac{\partial}{\partial x_i}K_\varepsilon\bigg(\frac{\partial u_0}{\partial x_j}\bigg)\bigg\}.\\
\endaligned
\end{equation*}

Notice that
\begin{equation*}
b_{S,ij}^{\alpha\beta}=
\begin{cases}
a_{ij}^{\alpha\beta}+a_{ik}^{\alpha\gamma}\frac{\partial}{\partial y_k}\chi^{\gamma\beta}_{S,j}-\hat{a}^{\alpha\beta}_{ij},~~~1\leq i\leq d, \\
~~-\chi^{\alpha\beta}_{S,j},~~~~~~~~~~~~~~~~~~~~~~~~i=d+1.
\end{cases}
\end{equation*}

Hence, direct computation shows that
\begin{equation*}
\aligned
  (\partial_t +\mathcal{L}_\varepsilon) w_\varepsilon
&= -\frac{\partial}{\partial x_i}\bigg\{(\widehat{a}_{ij} -a_{ij}^\varepsilon )
\bigg( \frac{\partial u_0}{\partial x_j} - K_\varepsilon \bigg( \frac{\partial u_0}{\partial x_j}\bigg) \bigg)\bigg\}\\
&~~~~~~+\frac{\partial}{\partial x_i} \bigg\{b^\varepsilon_{S,ij}K_\varepsilon\bigg(\frac{\partial u_0}{\partial x_j}\bigg)\bigg\}
-\varepsilon \partial_t
\bigg\{\chi^\varepsilon_{S,j}K_\varepsilon\bigg(\frac{\partial u_0}{\partial x_j}\bigg)\bigg\}\\
&~~~~~~~~+\varepsilon \frac{\partial}{\partial x_i}\bigg\{a_{ij}^\varepsilon \chi^\varepsilon_{S,k} \frac{\partial}{\partial x_j}K_\varepsilon\bigg(\frac{\partial u_0}{\partial x_k}\bigg)\bigg\}\\
&~~~~~~~~~~-(\partial_t +\mathcal{L}_\varepsilon)\bigg\{\varepsilon^2
\phi^\varepsilon_{S,(d+1)ij}\frac{\partial}{\partial x_i}K_\varepsilon\bigg(\frac{\partial u_0}{\partial x_j}\bigg)\bigg\}.\\
\endaligned
\end{equation*}
It follows that
\begin{equation}\label{7.5}
\aligned
  (\partial_t +\mathcal{L}_\varepsilon) w_\varepsilon
&= -\frac{\partial}{\partial x_i}\bigg\{(\widehat{a}_{ij} -a_{ij}^\varepsilon )
\left( \frac{\partial u_0}{\partial x_j} - K_\varepsilon \bigg( \frac{\partial u_0}{\partial x_j}\bigg) \right)\bigg\}\\
&~~~+\varepsilon^{-1}\left(\frac{\partial}{\partial x_i} b_{S,ij}\right)^\varepsilon K_\varepsilon\bigg(\frac{\partial u_0}{\partial x_j}\bigg)
-\varepsilon^{-1} \left(\partial_t
\chi_{S,j}\right)^\varepsilon K_\varepsilon\bigg(\frac{\partial u_0}{\partial x_j}\bigg)\\
&~~~~+ b^\varepsilon_{S,ij}\frac{\partial}{\partial x_i}K_\varepsilon\bigg(\frac{\partial u_0}{\partial x_j}\bigg)
-\varepsilon\chi^\varepsilon_{S,j}\partial_tK_\varepsilon\bigg(\frac{\partial u_0}{\partial x_j}\bigg)\\
&~~~~~~+\varepsilon \frac{\partial}{\partial x_i}\bigg\{a_{ij}^\varepsilon \chi^\varepsilon_{S,k} \frac{\partial}{\partial x_j}K_\varepsilon\bigg(\frac{\partial u_0}{\partial x_k}\bigg)\bigg\}\\
&~~~~~~~~-(\partial_t +\mathcal{L}_\varepsilon)\bigg\{\varepsilon^2
\phi^\varepsilon_{S,(d+1)ij}\frac{\partial}{\partial x_i}K_\varepsilon\bigg(\frac{\partial u_0}{\partial x_j}\bigg)\bigg\}.\\
\endaligned
\end{equation}

Notice that
\begin{equation}\label{7.4}
\aligned
\sum_{i=1}^{d}\frac{\partial}{\partial y_i}b_{S,ij}
-\frac{\partial}{\partial s}\chi_{S,j}&=\frac{\partial}{\partial y_i}\left\{a_{ij}(y,s)+a_{ik}\frac{\partial}{\partial y_k}\chi_{S,j}(y,s)\right\}-\frac{\partial}{\partial s}\chi_{S,j}\\
&=S^{-2}\chi_{S,j}.
\endaligned
\end{equation}
It follows from (\ref{7.4}), we know that the summation of the second and third terms in the right-hand side of (\ref{7.5}) is $$\varepsilon^{-1} S^{-2}\chi_{S,j}(x/\varepsilon,t/\varepsilon^2)K_\varepsilon\bigg(\frac{\partial u_0}{\partial x_j}\bigg).$$
This, together with (\ref{7.5}), yields that
\begin{equation}\label{equation-w}
\aligned
  (\partial_t +\mathcal{L}_\varepsilon) w_\varepsilon
&= -\frac{\partial}{\partial x_i}\bigg\{(\widehat{a}_{ij} -a_{ij}^\varepsilon )
\bigg( \frac{\partial u_0}{\partial x_j} - K_\varepsilon \bigg( \frac{\partial u_0}{\partial x_j}\bigg) \bigg)\bigg\}\\
&~~~~~~~+\varepsilon^{-1} S^{-2}\chi_{S,j}^{\varepsilon}K_\varepsilon\bigg(\frac{\partial u_0}{\partial x_j}\bigg)\\
&~~~~~~~~~~+ b^\varepsilon_{S,ij}\frac{\partial}{\partial x_i}K_\varepsilon\bigg(\frac{\partial u_0}{\partial x_j}\bigg)
-\varepsilon\chi^\varepsilon_{S,j}\partial_tK_\varepsilon\bigg(\frac{\partial u_0}{\partial x_j}\bigg)\\
&~~~~~~~~~~~~+\varepsilon \frac{\partial}{\partial x_i}\bigg\{a_{ij}^\varepsilon \chi^\varepsilon_{S,k} \frac{\partial}{\partial x_j}K_\varepsilon\bigg(\frac{\partial u_0}{\partial x_k}\bigg)\bigg\}\\
&~~~~~~~~~~~~~~-(\partial_t +\mathcal{L}_\varepsilon)\bigg\{\varepsilon^2
\phi^\varepsilon_{S,(d+1)ij}\frac{\partial}{\partial x_i}K_\varepsilon\bigg(\frac{\partial u_0}{\partial x_j}\bigg)\bigg\}.\\
\endaligned
\end{equation}
Thus we complete the proof.
\end{proof}

\begin{lemma}\label{lem-7.5}
Let $u_0\in L^2(0,T; H^2(\Omega))$ and $w_\varepsilon$ be defined by (\ref{w}) with $S=\varepsilon^{-1}$. Then we have
\begin{equation}\label{est-1}
\aligned
&\left|\iint_{\Omega_T}(\widehat{a}_{ij} -a_{ij}^\varepsilon )
\bigg(\frac{\partial u_0}{\partial x_j} - K_\varepsilon  \bigg(\frac{\partial u_0}{\partial x_j}\bigg)\bigg) \frac{\partial w_\varepsilon}{\partial x_i}
-\varepsilon \iint_{\Omega_T}a_{ij}^\varepsilon \chi^\varepsilon_{S,k} \frac{\partial}{\partial x_j}K_\varepsilon\left(\frac{\partial u_0}{\partial x_k}\right)\frac{\partial w_\varepsilon}{\partial x_i}\right|
\\&\leq C\left\{1+\delta^{-1}\Theta_\sigma(S)\right\}\|\nabla u_0\|_{L^2(\Omega_{T,4\delta})}\|\nabla w_\varepsilon\|_{L^2(\Omega_{T,4\delta})}\\
&~~~~+C\left\{\varepsilon+\Theta_\sigma(S)\right\}\left\{\|u_0\|_{L^2(0,T;H^2(\Omega))}+\|\partial_t u_0\|_{L^2(\Omega_T)}\right\}\|\nabla w_\varepsilon\|_{L^2(\Omega_T)},
\endaligned
\end{equation}
where $C$ depends only on $\mu$, $d$, $m$, $\Omega$ and $T$.
\end{lemma}

\begin{proof}
Let
\begin{equation*}
\Omega_{T,\delta}=\big(\left\{x\in\Omega:{\rm dist}(x,\partial\Omega)\leq\delta\right\}\times(0,T)\big)\cup
\left(\Omega\times(0,\delta^2)\right)\cup\left(\Omega\times(T-\delta^2,T)\right).
\end{equation*}
Observe that
\begin{align*}
&\left|\iint_{\Omega_T}(\widehat{a}_{ij} -a_{ij}^\varepsilon )
\bigg(\frac{\partial u_0}{\partial x_j} - K_\varepsilon  \bigg(\frac{\partial u_0}{\partial x_j}\bigg)\bigg) \frac{\partial w_\varepsilon}{\partial x_i}\right|\nonumber
\\&\leq C\iint_{\Omega_{T,3\delta}}\{|\nabla u_0|+S_\varepsilon(\eta_1\eta_2|\nabla u_0|)\}|\nabla w_\varepsilon|+C\iint_{\Omega_T\backslash \Omega_{T,3\delta}}|\nabla u_0-S_\varepsilon(\nabla u_0)||\nabla w_\varepsilon|\nonumber
\\&\leq C\bigg(\iint_{\Omega_{T,4\delta}}|\nabla u_0|^2\bigg)^{1/2}\bigg(\iint_{\Omega_{T,4\delta}}|\nabla w_\varepsilon|^2\bigg)^{1/2}\nonumber
\\&~~+C\bigg(\iint_{\Omega_T\backslash \Omega_{T,3\delta}}|\nabla u_0-S_\varepsilon(\nabla u_0)|^2\bigg)^{1/2}\bigg(\iint_{\Omega_T\backslash \Omega_{T,3\delta}}|\nabla w_\varepsilon|^2\bigg)^{1/2}.
\end{align*}

Using the Calder\'on extension theorem, we have
\begin{equation*}
  \bigg(\iint_{\R^{d+1}}|\nabla^2 \tilde{u}_0|^2\bigg)^{1/2}+\bigg(\iint_{\R^{d+1}}|\partial_t \tilde{u}_0|^2\bigg)^{1/2}\leq C\bigg\{\|u_0\|_{L^2(0,T;H^2(\Omega))}+\|\partial_t u_0\|_{L^2(\Omega_T)}\bigg\},
\end{equation*}
where $\tilde{u}_0$ is the extension of $u_0$ in $\R^{d+1}$.
This, together with Lemma \ref{lemma-S-2}, yields that
\begin{align*}
\bigg(\iint_{\Omega_T\backslash \Omega_{T,3\delta}}|\nabla u_0-S_\varepsilon(\nabla u_0)|^2\bigg)^{1/2}
&\leq C\bigg(\iint_{\R^{d+1}}|\nabla\tilde{u}_0 -S_\varepsilon(\nabla \tilde{u}_0)|^2\bigg)^{1/2}\nonumber
\\&\leq C\varepsilon\bigg\{\|\nabla^2 \tilde{u}_0\|_{L^2(\R^{d+1})}+\|\partial_t \tilde{u}_0\|_{L^2(\R^{d+1})}\bigg\}\nonumber
\\&\leq C\varepsilon \bigg\{\|u_0\|_{L^2(0,T;H^2(\Omega))}+\|\partial_t u_0\|_{L^2(\Omega_T)}\bigg\}.
\end{align*}

Next, It follows from Lemma \ref{6.42} that
\begin{align}\label{7.17-1}
\varepsilon \left|\iint_{\Omega_T}a_{ij}^\varepsilon \chi^\varepsilon_{S,k} \frac{\partial}{\partial x_j}K_\varepsilon\left(\frac{\partial u_0}{\partial x_k}\right)\frac{\partial w_\varepsilon}{\partial x_i}\right|
\leq C\Theta_\sigma(S)\iint_{\Omega_T}\left|\nabla K_\varepsilon\left(\nabla u_0\right)\cdot\nabla w_\varepsilon\right|,
\end{align}
where we have used $S=\varepsilon^{-1}$. Notice that
$$\nabla K_\varepsilon(\nabla u_0)=\nabla S_\varepsilon(\eta_1\eta_2\nabla u_0)=S_\varepsilon(\nabla(\eta_1\eta_2)\nabla u_0)+S_\varepsilon(\eta_1\eta_2\nabla^2 u_0),$$
then the right-hand side of (\ref{7.17-1}) is bounded by
\begin{align*}
C \delta^{-1}&\Theta_\sigma(S)\bigg(\iint_{\Omega_{T,4\delta}}|\nabla u_0|^2\bigg)^{1/2}\bigg(\iint_{\Omega_{T,3\delta}}|\nabla w_\varepsilon|^2\bigg)^{1/2}\nonumber
\\&+C\Theta_\sigma(S)\bigg(\iint_{\Omega_T}|\nabla^2 u_0|^2\bigg)^{1/2}\bigg(\iint_{\Omega_{T}}|\nabla w_\varepsilon|^2\bigg)^{1/2}.
\end{align*}
Above all, we obtain the desired estimate \eqref{est-1}.
\end{proof}

\begin{lemma}
  Let $u_0\in L^2(0,T; H^2(\Omega))$ and $w_\varepsilon$ be defined by (\ref{w}) with $S=\varepsilon^{-1}$. Then
\begin{equation}\label{7.19}
\aligned
&\varepsilon^2\left|\iint_{\Omega_T}\mathcal{L}_\varepsilon\left\{
\phi_{S,(d+1)ij}^\varepsilon\frac{\partial}{\partial x_i}K_\varepsilon\left(\frac{\partial u_0}{\partial x_j}\right)\right\}w_\varepsilon\right|\\
&\leq C\big\{\Theta_\sigma(S)+\varepsilon\big\}\big\{\delta^{-1}\|\nabla u_0\|_{L^2(\Omega_{T,3\delta})}\|\nabla w_\varepsilon\|_{L^2(\Omega_{T,3\delta})}+\|\nabla^2u_0\|_{L^2(\Omega_T)}\|\nabla w_\varepsilon\|_{L^2(\Omega_T)}\big\},
\endaligned
\end{equation}
where $C$ depends only on $\mu$, $d$, $m$, $\Omega$ and $T$.
\end{lemma}

\begin{proof}
Direct computation shows that
  \begin{align}\label{7.20}
\varepsilon^2&\left|\iint_{\Omega_T}\mathcal{L}_\varepsilon\bigg\{
\phi_{S,(d+1)ij}^\varepsilon\frac{\partial}{\partial x_i}K_\varepsilon\left(\frac{\partial u_0}{\partial x_j}\right)\bigg\}w_\varepsilon\right|\nonumber
\\ &~~~~ \leq
  \varepsilon\left|\iint_{\Omega_T}
 a_{ij}^\varepsilon \cdot \left(\frac{\partial}{\partial x_j} \phi_{S,(d+1) \ell k}  \right)^\varepsilon \cdot \frac{\partial}{\partial x_\ell}
K_\varepsilon \left(\frac{\partial u_0}{\partial x_k} \right) \cdot \frac{\partial w_\varepsilon}{\partial x_i}\right|\nonumber
\\&~~~~~~~~+\varepsilon^2\left| \iint_{\Omega_T}
 a_{ij}^\varepsilon \cdot \phi_{S,(d+1) \ell k } ^\varepsilon \cdot \frac{\partial^2 }{\partial x_j \partial x_\ell}
K_\varepsilon \left(\frac{\partial u_0}{\partial x_k}\right) \cdot \frac{\partial w_\varepsilon}{\partial x_i}\right|\nonumber
\\&~~~~=I_1+I_2.
\end{align}
In view of Lemma \ref{lemma-S-3} and Lemma \ref{7.47}, we have
\begin{align}\label{7.20-1}
I_2&\leq C\delta^{-1}\Theta_\sigma(S)\|\nabla u_0\|_{L^2(\Omega_{T,3\delta})}\|\nabla w_\varepsilon\|_{L^2(\Omega_{T,3\delta})}+C\Theta_\sigma(S)\|\nabla^2 u_0\|_{L^2(\Omega_{T})}\|\nabla w_\varepsilon\|_{L^2(\Omega_{T})}.
 \end{align}
 Next, by Lemma \ref{lemma-S-3} and Theorem \ref{7.35}, we obtain
\begin{equation*}
\aligned
I_1\leq &C\varepsilon\sup_{(x,t)\in\mathbb{R}^{d+1}}\left(\dashint_{Q_1(x,t)}\big|\frac{\partial}{\partial x_j}\phi_{S,(d+1)\ell k}\big|^2\right)^{1/2}\|S_\varepsilon(\nabla(\eta_1\eta_2)\nabla u_0)\|_{L^2(\Omega_{T,3\delta})}\|\nabla w_\varepsilon\|_{L^2(\Omega_{T,3\delta})}\\
&+C\varepsilon\sup_{(x,t)\in\mathbb{R}^{d+1}}
\left(\dashint_{Q_1(x,t)}\big|\frac{\partial}{\partial x_j}\phi_{S,(d+1)\ell k}\big|^2\right)^{1/2}
\|S_\varepsilon(\eta_1\eta_2\nabla^2 u_0)\|_{L^2(\Omega_{T})}\|\nabla w_\varepsilon\|_{L^2(\Omega_{T})}\\
\leq& C \varepsilon\delta^{-1}\|\nabla u_0\|_{L^2(\Omega_{T,3\delta})}\|\nabla w_\varepsilon\|_{L^2(\Omega_{T,3\delta})}+ C \varepsilon\|\nabla^2 u_0\|_{L^2(\Omega_{T})}\|\nabla w_\varepsilon\|_{L^2(\Omega_{T})}.
\endaligned
 \end{equation*}
This, together with \eqref{7.20} and \eqref{7.20-1}, gives \eqref{7.19}.
\end{proof}

\begin{lemma}
  Let $u_0\in L^2(0,T; H^2(\Omega))$ and $w_\varepsilon$ be defined by (\ref{w}) with $S=\varepsilon^{-1}$. Then
\begin{align}\label{w3}
& \left|\iint_{\Omega_T}b^\varepsilon_{S,ij}\frac{\partial}{\partial x_i}K_\varepsilon\left(\frac{\partial u_0}{\partial x_j}\right)w_\varepsilon-\iint_{\Omega_T}\varepsilon^2\partial_t \bigg\{\phi_{S,(d+1)ij}^\varepsilon\frac{\partial}{\partial x_i}K_\varepsilon\left(\frac{\partial u_0}{\partial x_j}\right)\bigg\}w_\varepsilon\right|\nonumber
\\&~~~~\leq C\left\{\langle|\nabla\chi-\nabla \chi_S|\rangle+\Theta_\sigma(S)+\Theta_1(S)\right\}\|u_0\|_{L^2(0,T;H^2(\Omega))}\|\nabla w_\varepsilon\|_{L^2(\Omega_T)}\nonumber\\
&~~~~~~~~+C\delta^{-1}\Theta_\sigma(S)\|\nabla u_0\|_{L^2(\Omega_{T,3\delta})}\|\nabla w_\varepsilon\|_{L^2(\Omega_{T,3\delta})},
\end{align}
where $C$ depends only on $\mu$, $d$, $m$, $\Omega$ and $T$.
\end{lemma}
\begin{proof}
We first notice that
\begin{align*}
b_{S,ij}&=\langle b_{S,ij}\rangle+\frac{\partial}{\partial y_k}\left(\frac{\partial}{\partial y_k}f_{S,ij}-\frac{\partial}{\partial y_i}f_{S,kj}\right)+\frac{\partial}{\partial y_i}\left(\frac{\partial}{\partial y_k}f_{S,kj}\right)-S^{-2}f_{S,ij}\nonumber
\\&=\langle b_{S,ij}\rangle+\sum_{k=1}^{d}\frac{\partial}{\partial y_k}\phi_{S,kij}+\frac{\partial}{\partial y_{d+1}}\phi_{S,(d+1)ij}+\sum_{k=1}^{d+1}\frac{\partial}{\partial y_i}\left(\frac{\partial}{\partial y_k}f_{S,kj}\right)-S^{-2}f_{S,ij}.
\end{align*}
Hence we have
\begin{equation}\label{7.10}
\aligned
b^\varepsilon_{S,ij}\frac{\partial}{\partial x_i}K_\varepsilon\left(\frac{\partial u_0}{\partial x_j}\right)
=&\langle b_{S,ij}\rangle\frac{\partial}{\partial x_i}K_\varepsilon\left(\frac{\partial u_0}{\partial x_j}\right)+\varepsilon\frac{\partial}{\partial x_k}\bigg\{\phi_{S,kij}^\varepsilon\frac{\partial}{\partial x_i}K_\varepsilon\left(\frac{\partial u_0}{\partial x_j}\right)\bigg\}\\
&-S^{-2}f^\varepsilon_{S,ij}\frac{\partial}{\partial x_i}K_\varepsilon\left(\frac{\partial u_0}{\partial x_j}\right)+\left(\frac{\partial}{\partial x_i} F_{S,j}\right)^\varepsilon\frac{\partial}{\partial x_i}K_\varepsilon\left(\frac{\partial u_0}{\partial x_j}\right)\\
&+\varepsilon^2\partial_t \bigg\{\phi_{S,(d+1)ij}^\varepsilon\frac{\partial}{\partial x_i}K_\varepsilon\left(\frac{\partial u_0}{\partial x_j}\right)\bigg\},
\endaligned
\end{equation}
where $F_{S,j}=\sum_{k=1}^{d+1}\frac{\partial}{\partial y_k}f_{S,kj}$.

In view of (\ref{7.10}), using integration by part, we obtain that the left-hand side of (\ref{w3}) equals to
\begin{align}\label{7.51}
&\bigg|\iint_{\Omega_T}\langle b_{S,ij}\rangle\frac{\partial}{\partial x_i}K_\varepsilon\left(\frac{\partial u_0}{\partial x_j}\right)w_\varepsilon
-\varepsilon\iint_{\Omega_T}\phi_{S,kij}^\varepsilon\frac{\partial}{\partial x_i}K_\varepsilon\left(\frac{\partial u_0}{\partial x_j}\right)\frac{\partial w_\varepsilon}{\partial x_k}\nonumber
\\&+\iint_{\Omega_T}\left(\frac{\partial}{\partial x_i}F_{S,j}\right)^\varepsilon\frac{\partial}{\partial x_i}K_\varepsilon\left(\frac{\partial u_0}{\partial x_j}\right) w_\varepsilon
-S^{-2}\iint_{\Omega_T}f^{\varepsilon}_{S,ij}\frac{\partial}{\partial x_i}K_\varepsilon\left(\frac{\partial u_0}{\partial x_j}\right)w_\varepsilon\bigg|.
\end{align}
By Lemma \ref{7.47}, we obtain that (\ref{7.51}) is bounded by
\begin{align}\label{7.52}
&  \left|\langle b_{S,ij}\rangle\right| \iint_{\Omega_T}\left|K_\varepsilon \left(\frac{\partial u_0}{\partial x_j}\right)\frac{\partial w_\varepsilon}{\partial x_i}\right|
+ C \Theta_\sigma(S) \iint_{\Omega_T}\left|\frac{\partial}{\partial x_i}K_\varepsilon\left(\frac{\partial u_0}{\partial x_j}\right)\frac{\partial w_\varepsilon}{\partial x_k}\right|\nonumber
\\&+ C \Theta_\sigma(S)\iint_{\Omega_T}\left|K_\varepsilon \left(\frac{\partial u_0}{\partial x_j}\right)\frac{\partial w_\varepsilon}{\partial x_i}\right|+ C \Theta_1(S)\iint_{\Omega_T}\left|K_\varepsilon \left(\frac{\partial u_0}{\partial x_j}\right)\frac{\partial w_\varepsilon}{\partial x_i}\right|.
\end{align}

Next, note that
$$
\left|\langle b_{S,ij}\rangle\right|\leq C \langle |\nabla \chi-\nabla \chi_S|\rangle.
$$
This, together with H\"older's inequality and Lemma \ref{lemma-S-1}, implies that (\ref{7.52}) is bounded by
\begin{align*}
&C\left\{\langle |\nabla \chi-\nabla \chi_S|\rangle+\Theta_\sigma(S)+\Theta_1(S)\right\}\bigg(\iint_{\Omega_T}|\nabla u_0|^2\bigg)^{1/2}\bigg(\iint_{\Omega_{T}}|\nabla w_\varepsilon|^2\bigg)^{1/2}\\
&+ C \Theta_\sigma(S) \bigg(\iint_{\Omega_T}|\nabla^2 u_0|^2\bigg)^{1/2}\bigg(\iint_{\Omega_{T}}|\nabla w_\varepsilon|^2\bigg)^{1/2}
\\&+ C \delta^{-1}\Theta_\sigma(S)\bigg(\iint_{\Omega_{T,3\delta}}|\nabla u_0|^2\bigg)^{1/2}\bigg(\iint_{\Omega_{T,3\delta}}|\nabla w_\varepsilon|^2\bigg)^{1/2}
\\&\leq C\left\{\langle|\nabla \chi-\nabla \chi_S|\rangle+\Theta_\sigma(S)+\Theta_1(S)\right\}\|u_0\|_{L^2(0,T;H^2(\Omega))}\|\nabla w_\varepsilon\|_{L^2(\Omega_T)}\\
&~~~~+C\delta^{-1}\Theta_\sigma(S)\|\nabla u_0\|_{L^2(\Omega_{T,3\delta})}\|\nabla w_\varepsilon\|_{L^2(\Omega_{T,3\delta})}.
\end{align*}
Thus we complete the proof.
\end{proof}

\begin{lemma}
Let $u_0\in L^2(0,T; H^2(\Omega))$ and $w_\varepsilon$ be defined by (\ref{w}) with $S=\varepsilon^{-1}$. Then
\begin{align}\label{w4}
\varepsilon&\bigg|\iint_{\Omega_T}
\chi^\varepsilon_{S,j}\partial_tK_\varepsilon\left(\frac{\partial u_0}{\partial x_j}\right)w_\varepsilon\bigg|\nonumber\\
&\leq C\varepsilon\delta^{-2}\left\{\Theta_\sigma(S)+\Theta_1(S)\right\}\|\nabla u_0\|_{L^2(\Omega_{T,3\delta})}\|\nabla w_\varepsilon\|_{L^2(\Omega_{T,3\delta})}\nonumber\\
&~~~~+C\left\{\Theta_\sigma(S)+\varepsilon\delta^{-1}\Theta_\sigma(S)+\varepsilon\delta^{-1}\Theta_1(S)\right\}\|\partial_t u_0\|_{L^2(\Omega_T)}\|\nabla w_\varepsilon\|_{L^2(\Omega_T)},
\end{align}
where $C$ depends only on $\mu$, $d$, $m$, $\Omega$ and $T$.
\end{lemma}

\begin{proof}
Note that $b_{S,(d+1)j}=-\chi_{S,j}$ and $\langle \chi_{S,j} \rangle=0$, it follows that
\begin{align}\label{B_S,(d+1)j}
b_{S,(d+1)j}&=\langle b_{S,(d+1)j}\rangle+\sum_{k=1}^{d}\frac{\partial}{\partial y_k}\phi_{S,k(d+1)j}+\sum_{k=1}^{d+1}\frac{\partial}{\partial y_{d+1}}\left(\frac{\partial}{\partial y_k}f_{S,kj}\right)-S^{-2}f_{S,(d+1)j}\nonumber\\
&=\sum_{k=1}^{d}\frac{\partial}{\partial y_k}\phi_{S,k(d+1)j}+\sum_{k=1}^{d+1}\frac{\partial}{\partial y_{d+1}}\left(\frac{\partial}{\partial y_k}f_{S,kj}\right)-S^{-2}f_{S,(d+1)j}.
\end{align}
Since $\phi_{S,k(d+1)j}$ is skew-symmetric, we then have
\begin{align}\label{7.11}
\varepsilon\chi^\varepsilon_{S,j}\partial_tK_\varepsilon\left(\frac{\partial u_0}{\partial x_j}\right)
=&-\varepsilon^2\frac{\partial}{\partial x_k}\left\{\phi^\varepsilon_{S,k(d+1)j}\partial_tK_\varepsilon\left(\frac{\partial u_0}{\partial x_j}\right)\right\}-\varepsilon\left(\frac{\partial}{\partial x_{d+1}}F_{S,j}\right)^{\varepsilon}\partial_tK_\varepsilon\left(\frac{\partial u_0}{\partial x_j}\right)\nonumber\\
&+\varepsilon S^{-2}f^{\varepsilon}_{S,(d+1)j}\partial_tK_\varepsilon\left(\frac{\partial u_0}{\partial x_j}\right).
\end{align}
By multiplying both sides of (\ref{7.11}) with $w_\varepsilon$ and then integrating on $\Omega_T$, integration by parts yields that the left-hand side of (\ref{w4}) equals to
\begin{align}\label{7.54}
&\bigg|\varepsilon^2\iint_{\Omega_T}\phi^\varepsilon_{S,k(d+1)j}\partial_tK_\varepsilon\left(\frac{\partial u_0}{\partial x_j}\right)\frac{\partial w_\varepsilon}{\partial x_k}-\varepsilon\iint_{\Omega_T}\left(\frac{\partial}{\partial x_{d+1}}F_{S,j}\right)^\varepsilon\partial_tK_\varepsilon\left(\frac{\partial u_0}{\partial x_j}\right)w_\varepsilon\nonumber
\\&
+\varepsilon S^{-2}\iint_{\Omega_T}f^{\varepsilon}_{S,(d+1)j}\partial_tK_\varepsilon\left(\frac{\partial u_0}{\partial x_j}\right)w_\varepsilon\bigg|.
\end{align}

Note that
\begin{equation*}
  \partial_t K_\varepsilon(\nabla u_0)=S_\varepsilon(\partial_t(\eta_1\eta_2)\nabla u_0)+\nabla S_\varepsilon(\eta_1\eta_2\partial_t u_0)-S_\varepsilon(\nabla(\eta_1\eta_2)\partial_t u_0).
\end{equation*}
This, together with H\"older's inequality and Lemma \ref{lemma-S-1}, gives
\begin{align}\label{7.18}
\iint_{\Omega_T}\left|\partial_tK_\varepsilon\left(\frac{\partial u_0}{\partial x_j}\right)\frac{\partial w_\varepsilon}{\partial x_i}\right|&\leq C\delta^{-2}\bigg(\iint_{\Omega_{T,3\delta}}|\nabla u_0|^2\bigg)^{1/2}\bigg(\iint_{\Omega_{T,3\delta}}|\nabla w_\varepsilon|^2\bigg)^{1/2}\nonumber
\\&~~~+ C\varepsilon^{-1}\bigg(\iint_{\Omega_T}|\partial_t u_0|^2\bigg)^{1/2}\bigg(\iint_{\Omega_T}|\nabla w_\varepsilon|^2\bigg)^{1/2}
\nonumber
\\&~~~+ C\delta^{-1}\bigg(\iint_{\Omega_{T,3\delta}}|\partial_t u_0|^2\bigg)^{1/2}\bigg(\iint_{\Omega_{T,3\delta}}|\nabla w_\varepsilon|^2\bigg)^{1/2}.
\end{align}

Similarly, we obtain
\begin{align}\label{7.16}
  \iint_{\Omega_T} \left|\partial_t K_\varepsilon\left(\frac{\partial u_0}{\partial x_j}\right)w_\varepsilon\right| &\leq C\delta^{-2}\bigg(\iint_{\Omega_{T,3\delta}}|\nabla u_0|^2\bigg)^{1/2}\bigg(\iint_{\Omega_{T,3\delta}}|w_\varepsilon|^2\bigg)^{1/2}\nonumber
  \\&~~~+ C\delta^{-1}\bigg(\iint_{\Omega_{T,3\delta}}|\partial_t u_0|^2\bigg)^{1/2}\bigg(\iint_{\Omega_{T,3\delta}}| w_\varepsilon|^2\bigg)^{1/2}\nonumber
  \\&~~~+ C\bigg(\iint_{\Omega_T}|\partial_t u_0|^2\bigg)^{1/2}\bigg(\iint_{\Omega_T}|\nabla w_\varepsilon|^2\bigg)^{1/2}.
  \end{align}

Since $w_\varepsilon=0$ on $\partial\Omega\times(0,T)$, by Poincar\'e's inequality, we have
$$\int_{\Omega_{T,3\delta}}\left|w_\varepsilon(x,t)\right|^2\leq C\int_{\Omega_{T,3\delta}}|\nabla w_\varepsilon(x,t)|^2.$$

This, together with (\ref{7.16}), implies that
\begin{align}\label{7.16-1}
   \iint_{\Omega_T} \left|\partial_t K_\varepsilon\left(\frac{\partial u_0}{\partial x_j}\right)w_\varepsilon\right|  &\leq C\delta^{-2}\bigg(\iint_{\Omega_{T,3\delta}}|\nabla u_0|^2\bigg)^{1/2}\bigg(\iint_{\Omega_{T,3\delta}}|\nabla w_\varepsilon|^2\bigg)^{1/2}\nonumber
  \\&~~~+ C\delta^{-1}\bigg(\iint_{\Omega_{T,3\delta}}|\partial_t u_0|^2\bigg)^{1/2}\bigg(\iint_{\Omega_{T,3\delta}}|\nabla w_\varepsilon|^2\bigg)^{1/2}
  \nonumber
  \\&~~~+ C\bigg(\iint_{\Omega_T}|\partial_t u_0|^2\bigg)^{1/2}\bigg(\iint_{\Omega_T}|\nabla w_\varepsilon|^2\bigg)^{1/2}.
\end{align}


In view of (\ref{7.16-1}), \eqref{7.18} and Lemma \ref{7.47}, (\ref{7.54}) is bounded by
\begin{align*}
& C\varepsilon\delta^{-2}\left\{\Theta_\sigma(S)+\Theta_1(S)\right\}\|\nabla u_0\|_{L^2(\Omega_{T,3\delta})}\|\nabla w_\varepsilon\|_{L^2(\Omega_{T,3\delta})}\\
&+C\left\{\Theta_\sigma(S)+\varepsilon\delta^{-1}\Theta_\sigma(S)+\varepsilon\delta^{-1}\Theta_1(S)\right\}\|\partial_t u_0\|_{L^2(\Omega_T)}\|\nabla w_\varepsilon\|_{L^2(\Omega_T)}.
\end{align*}
We hence complete the proof.
\end{proof}

\begin{lemma}\label{lem-7.9}
  Let $u_0\in L^2(0,T; H^2(\Omega))$ and $w_\varepsilon$ be defined by (\ref{w}) with $S=\varepsilon^{-1}$. Then
\begin{equation}\label{w5}
\aligned
&\left|\varepsilon^{-1} S^{-2} \iint_{\Omega_T}\chi_{S,j}^{\varepsilon}K_\varepsilon\left(\frac{\partial u_0}{\partial x_j}\right)w_\varepsilon\right|\\
&\quad\quad\leq C\varepsilon\big\{\Theta_\sigma(S)+\Theta_1(S)+1\big\}\cdot\|\nabla u_0\|_{L^2(\Omega_T)}
\|\nabla w_\varepsilon\|_{L^2(\Omega_T)},
\endaligned
\end{equation}
where $C$ depends only on $\mu$, $d$, $m$, $\Omega$ and $T$.
\end{lemma}

\begin{proof}
By \eqref{B_S,(d+1)j}, we know that
\begin{align*}
b_{S,(d+1)j}=\sum_{k=1}^{d}\frac{\partial}{\partial y_k}\phi_{S,k(d+1)j}+\frac{\partial}{\partial y_{d+1}}F_{S,j}-S^{-2}f_{S,(d+1)j}.
\end{align*}
Then it follows from integration by parts that the left-hand side of (\ref{w5}) can be bounded by
\begin{align}\label{7.56}
&\bigg|\varepsilon^{-1}S^{-2}\iint_{\Omega_T}\left(\frac{\partial}{\partial x_k}\phi_{S,k(d+1)j}\right)^\varepsilon K_\varepsilon\left(\frac{\partial u_0}{\partial x_j}\right) w_\varepsilon\bigg|+\varepsilon^{-1} S^{-4} \bigg|\iint_{\Omega_T} f^{\varepsilon}_{S,(d+1)j}K_\varepsilon\left(\frac{\partial u_0}{\partial x_j}\right)w_\varepsilon\bigg|\nonumber
\\&~~+\varepsilon^{-1} S^{-2} \bigg|\iint_{\Omega_T}\left(\frac{\partial}{\partial x_{d+1}}F_{S,j}\right)^\varepsilon K_\varepsilon\left(\frac{\partial u_0}{\partial x_j}\right)w_\varepsilon\bigg|=I_1+I_2+I_3.
\end{align}

To estimate $I_1$, in view of Lemma \ref{lemma-S-3} and Theorem \ref{7.35}, we have
\begin{align}\label{I1}
I_1&\leq C\varepsilon\sup_{(x,t)\in\mathbb{R}^{d+1}}\left(\dashint_{Q_1(x,t)}\bigg|\frac{\partial}{\partial x_k}\phi_{S,k(d+1)j}\bigg|^2\right)^{1/2}\|K_\varepsilon(\nabla u_0)\|_{L^2(\Omega_T)}\|w_\varepsilon\|_{L^2(\Omega_T)}\nonumber\\
&\leq C\varepsilon\|\nabla u_0\|_{L^2(\Omega_T)}\|\nabla w_\varepsilon\|_{L^2(\Omega_T)},
\end{align}
where we have used Lemma \ref{lemma-S-1} and Poincar\'e's inequality for the last step.

Next, by using Lemma \ref{7.47} and Lemma \ref{lemma-S-1}, we have
\begin{align}\label{I2+I3}
I_2+I_3&\leq C \varepsilon\left\{\Theta_\sigma(S)+\Theta_1(S)\right\} \iint_{\Omega_T} \left|K_\varepsilon\left(\frac{\partial u_0}{\partial x_j}\right)w_\varepsilon\right|
\nonumber\\& ~~~~\leq
C\varepsilon\left\{\Theta_\sigma(S)+\Theta_1(S)\right\}\|\nabla u_0\|_{L^2(\Omega_T)}\| w_\varepsilon\|_{L^2(\Omega_T)}
\nonumber\\& ~~~~\leq
C\varepsilon\left\{\Theta_\sigma(S)+\Theta_1(S)\right\}\|\nabla u_0\|_{L^2(\Omega_T)}\|\nabla w_\varepsilon\|_{L^2(\Omega_T)},
\end{align}
where we have used Poincar\'e's inequality for the last step. The desired estimate \eqref{w5} now follows from \eqref{7.56}, \eqref{I1} and \eqref{I2+I3}.
\end{proof}

\begin{theorem}\label{thm-w-e}
  Let $w_\varepsilon$ be defined as (\ref{w}), $S=\varepsilon^{-1}$ and $\varepsilon\leq \delta <1$. Then we have
  \begin{equation*}
\aligned
& \left|  \int_0^T  \big\langle (\partial_t +\mathcal{L}_\varepsilon) w_\varepsilon,
 w_\varepsilon \big\rangle_{H^{-1}(\Omega) \times H^1_0(\Omega)} \, dt  \right| \\
&~~~~\leq
C\left\{1+\delta^{-1}\Theta_\sigma(S)+\delta^{-1}\Theta_1(S)
+\varepsilon\delta^{-1}\right\}\cdot\|\nabla u_0\|_{L^2(\Omega_{T,4\delta})}\|\nabla w_\varepsilon\|_{L^2(\Omega_{T,4\delta})}\\
&~~~~~~+C\left\{\varepsilon+\langle|\nabla\chi-\nabla \chi_S|\rangle+\Theta_\sigma(S)+\Theta_1(S)\right\}\\
&~~~~\qquad\cdot\left\{\|u_0\|_{L^2(0,T;H^2(\Omega))}
+\|\partial_t u_0\|_{L^2(\Omega_T)}\right\}\|\nabla w_\varepsilon\|_{L^2(\Omega_T)},
\endaligned
\end{equation*}
where $C$ depends only on $d$, $\mu$, $m$, $\Omega$ and $T$.
\end{theorem}
\begin{proof}
  This follows directly from Theorem \ref{Theorem-2.1} and Lemmas \ref{lem-7.5}-\ref{lem-7.9}.
\end{proof}

We are ready to give the proof of Theorem \ref{rate}.

\begin{proof}[Proof of Theorem \ref{rate}]
Notice that $w_\varepsilon\in L^2(0,T;H^1_0(\Omega))$ and $w_\varepsilon=0$ on $\Omega\times\{t=0\}$. It follows from Theorem \ref{thm-w-e} that
\begin{equation}\label{est-w-e-1}
\aligned
&\mu\iint_{\Omega_T}|\nabla w_\varepsilon|^2\,dxdt\\
& \leq\left|  \int_0^T  \big\langle (\partial_t +\mathcal{L}_\varepsilon) w_\varepsilon,
 w_\varepsilon \big\rangle_{H^{-1}(\Omega) \times H^1_0(\Omega)} \, dt  \right| \\
&\leq
C\left\{1+\delta^{-1}\Theta_\sigma(S)+\delta^{-1}\Theta_1(S)
+\varepsilon\delta^{-1}\right\}\cdot\|\nabla u_0\|_{L^2(\Omega_{T,4\delta})}\|\nabla w_\varepsilon\|_{L^2(\Omega_{T,4\delta})}\\
&+C\left\{\varepsilon+\langle|\nabla\chi-\nabla \chi_S|\rangle+\Theta_\sigma(S)+\Theta_1(S)\right\}\\
&~~~~\cdot\left\{\|u_0\|_{L^2(0,T;H^2(\Omega))}
+\|\partial_t u_0\|_{L^2(\Omega_T)}\right\}\|\nabla w_\varepsilon\|_{L^2(\Omega_T)}.
\endaligned
\end{equation}
By \cite[Remark 3.6]{GS-2017},
$$
\|\nabla u_0\|_{L^2(\Omega_{T,4\delta})}\leq C\delta^{1/2}\left\{
\|u_0\|_{L^2(0,T;H^2(\Omega))}+\sup_{\delta^2<t<T}
\left(\frac{1}{\delta}\int_{t-\delta^2}^t\int_{\Omega}|\nabla u_0|^2\right)^{1/2}
\right\}.
$$
This, together with \eqref{est-w-e-1}, implies that
\begin{equation}\label{7.41-1}
\aligned
\|\nabla w_\varepsilon\|_{L^2(\Omega_T)}&\leq
C\left\{\delta^{1/2}+\delta^{-1/2}\Theta_\sigma(S)+\delta^{-1/2}\Theta_1(S)
+\varepsilon\delta^{-1/2}\right\}\\
&~~~~\cdot\left\{
\|u_0\|_{L^2(0,T;H^2(\Omega))}+\sup_{\delta^2<t<T}
\left(\frac{1}{\delta}\int_{t-\delta^2}^t\int_{\Omega}|\nabla u_0|^2\right)^{1/2}
\right\}\\
&~~~~+C\big\{\varepsilon+\langle|\nabla\chi-\nabla \chi_S|\rangle+\Theta_\sigma(S)+\Theta_1(S)\big\}\\
&~~~~\cdot\left\{\|u_0\|_{L^2(0,T;H^2(\Omega))}
+\|\partial_t u_0\|_{L^2(\Omega_T)}\right\}.
\endaligned
\end{equation}
By setting
  $$\delta=\varepsilon+\langle|\nabla \chi-\nabla \chi_S|\rangle+\Theta_\sigma(S)+\Theta_1(S),$$
it follows from \eqref{7.41-1} that
  \begin{align*}
\|\nabla w_\varepsilon\|_{L^2(\Omega_T)}&\leq C\delta^{1/2}\bigg\{\|u_0\|_{L^2(0,T;H^2(\Omega))}
+\|\partial_t u_0\|_{L^2(\Omega_T)}
+\sup_{\delta^2<t<T}
\left(\frac{1}{\delta}\int_{t-\delta^2}^t\int_{\Omega}|\nabla u_0|^2\right)^{1/2}\bigg\}\\
&\leq C\delta^{1/2}\bigg\{\|u_0\|_{L^2(0,T;H^2(\Omega))}
+\|F\|_{L^2(\Omega_T)}
+\sup_{\delta^2<t<T}
\left(\frac{1}{\delta}\int_{t-\delta^2}^t\int_{\Omega}|\nabla u_0|^2\right)^{1/2}\bigg\},
\end{align*}
where for the last step we have used
\begin{equation*}
\|\partial_t u_0\|_{L^2(\Omega_T)}\leq C\left\{\|\nabla^2 u_0\|_{L^2(\Omega_T)}+\|F\|_{L^2(\Omega_T)}\right\}.
\end{equation*}
One may observe that $\Theta_\sigma(S)\leq C[\Theta_1(S)]^\sigma$, hence we have
\begin{align*}
\delta&\leq C\left\{S^{-1}+\langle|\nabla \chi-\nabla \chi_S|\rangle+[\Theta_1(S)]^\sigma\right\}.
\end{align*}
And it is easy to see that $\delta\to 0$ as $S\to\infty$, one may find a modulus $\eta$ on $(0,1]$, depending only on $A, T$ such that $\eta(0+)=0$ and for any $S\geq 1$,
\begin{equation}\label{def-eta}
\bigg\{S^{-1}+\langle|\nabla \chi-\nabla \chi_S|\rangle+[\Theta_1(S)]^{\sigma}\bigg\}\leq [\eta(S^{-1})]^{2}.
\end{equation}
As a result, we obtain that
\begin{equation*}
\|\nabla w_\varepsilon\|_{L^2(\Omega_T)}\leq C\eta(\varepsilon)\left\{\|u_0\|_{L^2(0,T;H^2(\Omega))}+\|F\|_{L^2(\Omega_T)}
+\sup_{\delta^2<t<T}
\left(\frac{1}{\delta}\int_{t-\delta^2}^t\int_{\Omega}|\nabla u_0|^2\right)^{1/2}\right\}.
\end{equation*}
Since $w_\varepsilon\in L^2(0,T;H^1_0(\Omega))$, this, together with Poincar\'e's inequality, gives
\begin{equation*}
\|w_\varepsilon\|_{L^2(\Omega_T)}\leq C\eta(\varepsilon)\left\{\|u_0\|_{L^2(0,T;H^2(\Omega))}+\|F\|_{L^2(\Omega_T)}
+\sup_{\delta^2<t<T}
\left(\frac{1}{\delta}\int_{t-\delta^2}^t\int_{\Omega}|\nabla u_0|^2\right)^{1/2}\right\}.
\end{equation*}
In view of the definition of $w_\varepsilon$, we see that
\begin{equation*}
\aligned
\|&u_\varepsilon-u_0\|_{L^2(\Omega_T)}\\
&\leq\|w_\varepsilon\|_{L^2(\Omega_T)}+\left\|\varepsilon\chi_{S,j}^\varepsilon K_\varepsilon\left(\frac{\partial u_0}{\partial x_j}\right)+\varepsilon^2\phi_{S,(d+1)ij}^\varepsilon \frac{\partial}{\partial x_i}K_\varepsilon\left(\frac{\partial u_0}{\partial x_j}\right)\right\|_{L^2(\Omega_T)}\\
&\leq\|w_\varepsilon\|_{L^2(\Omega_T)}+C\Theta_\sigma(S)\left\{\left\|\nabla u_0\right\|_{L^2(\Omega_T)}+\left\|\nabla^2 u_0\right\|_{L^2(\Omega_T)}\right\}\\
&\leq C\eta(\varepsilon)\left\{\|u_0\|_{L^2(0,T;H^2(\Omega))}+\|F\|_{L^2(\Omega_T)}
+\sup_{\delta^2<t<T}
\left(\frac{1}{\delta}\int_{t-\delta^2}^t\int_{\Omega}|\nabla u_0|^2\right)^{1/2}\right\},
\endaligned
\end{equation*}
where we have used Lemma \ref{7.47} and Lemma \ref{lemma-S-3} in the second inequality. Thus we complete the proof.
\end{proof}

\section{Interior Lipschitz estimates}
In this section, we establish the uniform interior Lipschitz estimates for $\partial_t+\mathcal{L}_\varepsilon$. Our main tool is the convergence rate method first established in \cite{Armstrong-Smart-2016} and further developed in \cite{AS-2016}.

We begin with the following approximation result.
\begin{lemma}\label{lem-approx}
Suppose that $A$ satisfies \eqref{ellipticity} and is uniformly almost periodic. Let $u_\varepsilon$ be a weak solution of
\begin{equation*}
(\partial_t+\mathcal{L}_\varepsilon)u_\varepsilon=F\quad{\rm in}~~Q_{2r},
\end{equation*}
where $F\in L^2(Q_{2r})$. Let $S=\varepsilon^{-1}$. Then there exists a weak solution of
\begin{equation*}
(\partial_t+\mathcal{L}_0)u_0=F\quad{\rm in}~~Q_{r},
\end{equation*}
such that
\begin{equation}\label{est-approximation}
\left(\dashint_{Q_{r/2}}|u_\varepsilon-u_0|^2\right)^{\frac{1}{2}}\leq
C\left[\eta\left(\frac{\varepsilon}{r}\right)\right]^\gamma
\left\{\left(\dashint_{Q_{2r}}|u_\varepsilon|^2\right)^{\frac{1}{2}}
+r^2\left(\dashint_{Q_{2r}}|F|^2\right)^{\frac{1}{2}}\right\},
\end{equation}
where $\eta(t)$ is defined by
 \begin{equation}\label{def-eta-1}
 \eta(t)=[\Theta_1(t^{-1})]^\sigma+\sup_{S\geq t^{-1}}\langle|\nabla\chi_S-\nabla\chi|\rangle+t,
 \end{equation}
$\gamma=\frac{1}{2}-\frac{1}{q}$ for some $q>2$ and $C$ depends only on $d$, $m$ and $\mu$.
\end{lemma}

\begin{proof}
By rescaling, we may assume that $r=1$. Let $u_0$ be the weak solution of
\begin{equation*}
\begin{cases}
(\partial_t+\mathcal{L}_0)u_0=F\quad{\rm in}~~Q_1,\\
\qquad\qquad u_0=u_\varepsilon~~\,{\rm on}~~\partial_pQ_1,
\end{cases}
\end{equation*}
where $\partial_pQ_1$ denotes the parabolic boundary of $Q_1$. Then we have
\begin{equation*}
\begin{cases}
(\partial_t+\mathcal{L}_0)(u_0-u_\varepsilon)={\rm div}[(\hat{A}-A^\varepsilon)\nabla u_\varepsilon]\quad{\rm in}~~Q_1,\\
\qquad\qquad~~ \,u_0-u_\varepsilon=0\qquad\qquad\qquad\qquad\,{\rm on}~~\partial_pQ_1.
\end{cases}
\end{equation*}
By the standard regularity estimates for parabolic systems with constant coefficients, for any $2\leq q<\infty$,
\begin{equation*}
\dashint_{Q_1}|\nabla(u_0-u_\varepsilon)|^q\leq C\dashint_{Q_1}|\nabla u_\varepsilon|^q,
\end{equation*}
where $C$ depends only on $d$, $\mu$ and $q$. It follows that
\begin{equation}\label{approx-1}
\dashint_{Q_1}|\nabla u_0|^q\leq C\dashint_{Q_1}|\nabla u_\varepsilon|^q
\end{equation}
for any $2\leq q<\infty$.

To see \eqref{est-approximation}, let $K_\varepsilon(f)=S_\varepsilon(\eta_\delta f)$, where $\delta>\varepsilon$ is a small number to be determined later, $S_\varepsilon$ is defined by \eqref{1.18} with $\theta=\theta (y,s)=\theta_1(y)\theta_2(s)$, where $\theta_1\in C_0^\infty(B(0,1))$, $\theta_2\in C_0^\infty(-1,0)$, $\theta_1,\theta_2\geq 0$, and $\int_{\mathbb{R}^{d}}\theta_1(y)\,dy=\int_{\mathbb{R}}\theta_2(s)\,ds=1$, and $\eta_\delta\in C^\infty_0(\mathbb{R}^{d+1})$ such that $\eta_\delta=1$ in $Q_{1-3\delta}$, $\eta_\delta=0$ in $Q_1\setminus Q_{1-2\delta}$, $0\leq\delta\leq 1$, $|\nabla \eta_\delta|\leq C/\delta$, and $|\partial_t\eta_\delta|+|\nabla^2 \eta_\delta|\leq C/\delta^2$.
Let $w_\varepsilon$ be defined as in \eqref{w}. By \eqref{equation-w}, \eqref{7.10} and \eqref{7.11}, we have
\begin{equation*}
\aligned
  (\partial_t& +\mathcal{L}_\varepsilon) w_\varepsilon\\
&= -\frac{\partial}{\partial x_i}\bigg\{(\widehat{a}_{ij} -a_{ij}^\varepsilon )
\bigg( \frac{\partial u_0}{\partial x_j} - K_\varepsilon \bigg( \frac{\partial u_0}{\partial x_j}\bigg) \bigg)\bigg\}+\varepsilon^{-1} S^{-2}\chi_{S,j}^{\varepsilon}K_\varepsilon\bigg(\frac{\partial u_0}{\partial x_j}\bigg)\\
&+ \langle b_{S,ij}\rangle\frac{\partial}{\partial x_i}K_\varepsilon\bigg(\frac{\partial u_0}{\partial x_j}\bigg)
+\varepsilon \frac{\partial}{\partial x_k}\bigg\{\phi^\varepsilon_{S,kij} \frac{\partial}{\partial x_i}K_\varepsilon\bigg(\frac{\partial u_0}{\partial x_j}\bigg)\bigg\}-S^{-2}f_{S,ij}^\varepsilon\frac{\partial}{\partial x_i}K_\varepsilon\bigg(\frac{\partial u_0}{\partial x_j}\bigg)\\
&+\left(\frac{\partial}{\partial x_i}F_{S,j}\right)^\varepsilon\frac{\partial}{\partial x_i}K_\varepsilon\bigg(\frac{\partial u_0}{\partial x_j}\bigg)
+\varepsilon \frac{\partial}{\partial x_i}\bigg\{a_{ij}^\varepsilon \chi^\varepsilon_{S,k} \frac{\partial}{\partial x_j}K_\varepsilon\bigg(\frac{\partial u_0}{\partial x_k}\bigg)\bigg\}\\
&+\varepsilon^2 \frac{\partial}{\partial x_k}\bigg\{\phi^\varepsilon_{S,k(d+1)j} \partial_tK_\varepsilon\bigg(\frac{\partial u_0}{\partial x_j}\bigg)\bigg\}
+\varepsilon\left(\frac{\partial}{\partial x_{d+1}}F_{S,j}\right)^\varepsilon\partial_t K_\varepsilon\bigg(\frac{\partial u_0}{\partial x_j}\bigg)
\\
&-\varepsilon S^{-2} f^\varepsilon_{S,(d+1)j}\partial_t K_\varepsilon\bigg(\frac{\partial u_0}{\partial x_j}\bigg)+\varepsilon\frac{\partial}{\partial x_i}\bigg\{a_{ij}^\varepsilon\bigg(\frac{\partial}{\partial x_j}
\phi_{S,(d+1)kl}\bigg)^\varepsilon\frac{\partial}{\partial x_k}K_\varepsilon\bigg(\frac{\partial u_0}{\partial x_l}\bigg)\bigg\}
\\
&+\varepsilon^2\frac{\partial}{\partial x_i}\bigg\{a_{ij}^\varepsilon
\phi_{S,(d+1)kl}^\varepsilon\frac{\partial^2}{\partial x_j\partial x_k}K_\varepsilon\bigg(\frac{\partial u_0}{\partial x_l}\bigg)\bigg\}.\\
\endaligned
\end{equation*}
Since $w_\varepsilon=0$ on $\partial_pQ_1$, it follows that
\begin{equation*}
\aligned
\int_{Q_1}|\nabla w_\varepsilon|^2 \leq &C\int_{Q_1}|\nabla u_0-K_\varepsilon(\nabla u_0)||\nabla w_\varepsilon|+C\varepsilon\|\chi_S\|_{L^\infty}\int_{Q_1}|K_\varepsilon(\nabla u_0)||w_\varepsilon|\\
&+C\langle |\nabla\chi-\nabla\chi_S|
\rangle\int_{Q_1}|K_\varepsilon(\nabla u_0)||\nabla w_\varepsilon|+
C\varepsilon\|\phi_S\|_{L^\infty}
\int_{Q_1}|\nabla K_\varepsilon(\nabla u_0)||\nabla w_\varepsilon|\\
&+C\|\nabla F_S\|_{L^\infty}\int_{Q_1}|\nabla K_\varepsilon(\nabla u_0)||w_\varepsilon|+CS^{-2}\|f_S\|_{L^\infty}\int_{Q_1}|\nabla K_\varepsilon(\nabla u_0)|| w_\varepsilon|\\
&+C\varepsilon\|\chi_S\|_{L^\infty}\int_{Q_1}|\nabla K_\varepsilon(\nabla u_0)||\nabla w_\varepsilon|+
C\varepsilon^2\|\phi_S\|_{L^\infty}\int_{Q_1}|\partial_t K_\varepsilon(\nabla u_0)||\nabla w_\varepsilon|\\
&+C\varepsilon\|\partial_t F_S\|_{L^\infty}\int_{Q_1}|\partial_t K_\varepsilon(\nabla u_0)||w_\varepsilon|
+C\varepsilon S^{-2}\|f_S\|_{L^\infty}\int_{Q_1}|\partial_t K_\varepsilon(\nabla u_0)||w_\varepsilon|\\
&+C\varepsilon\int_{Q_1}|(\nabla\phi_S)^\varepsilon\nabla K_\varepsilon(\nabla u_0)||\nabla w_\varepsilon|
+C\varepsilon^2\|\phi_S\|_{L^\infty}\int_{Q_1}|\nabla^2 K_\varepsilon(\nabla u_0)||\nabla w_\varepsilon|\\
=&I_1+I_2+\cdots+I_{12}.
\endaligned
\end{equation*}
By the same argument as in the proof of \cite[Lemma 3.6]{GS-2020-ARMA}, we have
\begin{equation*}
I_1\leq C\left\{\int_{Q_1\setminus Q_{1-3\delta}}|\nabla u_0|^2+\varepsilon^2\int_{Q_{1-2\delta}}\left(|\nabla^2 u_0|^2+|\partial_t u_0|^2\right)\right\}^{1/2}\left(\int_{Q_1}|\nabla w_\varepsilon|^2\right)^{1/2},
\end{equation*}
where we have used the fact that $\delta\geq \varepsilon$. By the standard regularity estimates for parabolic systems with constant coefficients,
\begin{equation}\label{regularity}
\int_{Q_{1-2\delta}}\left(|\nabla^2 u_0|^2+|\partial_t u_0|^2\right)\leq C
\left\{\int_{Q_{1-\delta}}\frac{|\nabla u_0(y,s)|^2\,dyds}{|{\rm dist}_p((y,s),\partial_p Q_1)|^2}+\int_{Q_1}|F|^2\right\},
\end{equation}
where ${\rm dist}_p((y,s),\partial_p Q_1)$ denotes the parabolic distance from $(y,s)$ to $\partial_pQ_1$. It follows that
\begin{equation*}
\aligned
I_1&\leq C\left\{\int_{Q_1\setminus Q_{1-3\delta}}|\nabla u_0|^2+\varepsilon^2
\int_{Q_{1-\delta}}\frac{|\nabla u_0(y,s)|^2\,dyds}{|{\rm dist}_p((y,s),\partial_p Q_1)|^2}+\varepsilon^2\int_{Q_1}|F|^2\right\}^{1/2}\left(\int_{Q_1}|\nabla w_\varepsilon|^2\right)^{1/2}\\
&\leq C\left\{\delta^{\frac{1}{2}-\frac{1}{q}}\left(\dashint_{Q_1}|\nabla u_0|^q\right)^{1/q}+\varepsilon\left(\dashint_{Q_1}|F|^2\right)^{1/2}\right\}
\left(\int_{Q_1}|\nabla w_\varepsilon|^2\right)^{1/2},
\endaligned
\end{equation*}
where $q>2$ and we have used H\"older's inequality for the last step.

By Lemma \ref{6.42}, Lemma \ref{7.47} and Poincar\'e's inequality, we see that
\begin{equation*}
\aligned
I_2+\cdots+I_7\leq & C\{\Theta_\sigma(S)+\langle|\nabla\chi_S-\nabla\chi|\rangle\}
\|\nabla u_0\|_{L^2(Q_1)}\|\nabla w_\varepsilon\|_{L^2(Q_1)}\\
&+ C\{\Theta_\sigma(S)+\Theta_1(S)\}
\|\nabla K_\varepsilon(\nabla u_0)\|_{L^2(Q_1)}
\|\nabla w_\varepsilon\|_{L^2(Q_1)}.
\endaligned
\end{equation*}
Note that
\begin{equation*}
\aligned
\|\nabla K_\varepsilon(\nabla u_0)\|_{L^2(Q_1)}
&\leq \|S_\varepsilon(\nabla\eta_\delta\nabla u_0)\|_{L^2(Q_1)}
+\|S_\varepsilon(\eta_\delta\nabla^2 u_0)\|_{L^2(Q_1)}\\
&\leq C\delta^{-1}\|\nabla u_0\|_{L^2(Q_1\setminus Q_{1-3\delta})}
+C\|\nabla^2 u_0\|_{L^2(Q_{1-2\delta})}\\
&\leq C\delta^{-\frac{1}{2}-\frac{1}{q}}\left(\dashint_{Q_1}|\nabla u_0|^q\right)^{1/q}+C\left(\dashint_{Q_1}|F|^2\right)^{1/2}
\endaligned
\end{equation*}
for some $q>2$, where we have used H\"older's inequality and \eqref{regularity} for the last step.
As a result,
\begin{equation*}
\aligned
I_2+\cdots+I_7\leq &C\delta^{-1}\{\Theta_\sigma(S)+\Theta_1(S)+\langle|\nabla\chi_S-\nabla\chi|\rangle\}\\
&\cdot\left\{\delta^{\frac{1}{2}-\frac{1}{q}}\left(\dashint_{Q_1}|\nabla u_0|^q\right)^{1/q}+\delta\left(\dashint_{Q_1}|F|^2\right)^{1/2}\right\}
\left(\int_{Q_1}|\nabla w_\varepsilon|^2\right)^{1/2}.
\endaligned
\end{equation*}
Similarly, by Lemma \ref{7.47} and Poincar\'e's inequality, we have
\begin{equation*}
\aligned
I_8+I_9+I_{10}\leq & C\varepsilon\{\Theta_\sigma(S)+\Theta_1(S)\}
\|\partial_t K_\varepsilon(\nabla u_0)\|_{L^2(Q_1)}
\|\nabla w_\varepsilon\|_{L^2(Q_1)}.
\endaligned
\end{equation*}
Observe that
\begin{equation*}
\aligned
\partial_t K_\varepsilon(\nabla u_0)= S_\varepsilon(\partial_t\eta_\delta\nabla u_0)+ \nabla S_\varepsilon(\eta_\delta\partial_t u_0)-
S_\varepsilon(\nabla\eta_\delta \partial_t u_0).
\endaligned
\end{equation*}
Hence by Lemma \ref{lemma-S-1}, \eqref{regularity} and H\"older's inequality,
\begin{equation*}
\aligned
\|\partial_t K_\varepsilon(\nabla u_0)\|_{L^2(Q_1)}&\leq
C\left\{\delta^{-2}\left(\int_{Q_1\setminus Q_{1-3\delta}}|\nabla u_0|^2\right)^{1/2}
+\varepsilon^{-1}\left(\int_{Q_{1-2\delta}}|\partial_t u_0|^2\right)^{1/2}\right\}\\
&\leq C\varepsilon^{-1}\delta^{-1}\left\{\delta^{\frac{1}{2}-\frac{1}{q}}\left(\dashint_{Q_1}|\nabla u_0|^q\right)^{1/q}+\delta\left(\dashint_{Q_1}|F|^2\right)^{1/2}\right\}.
\endaligned
\end{equation*}
This gives
\begin{equation*}
\aligned
I_8+I_9+I_{10}\leq & C\delta^{-1}\{\Theta_\sigma(S)+\Theta_1(S)\}\\
&\cdot\left\{\delta^{\frac{1}{2}-\frac{1}{q}}\left(\dashint_{Q_1}|\nabla u_0|^q\right)^{1/q}+\delta\left(\dashint_{Q_1}|F|^2\right)^{1/2}\right\}
\left(\int_{Q_1}|\nabla w_\varepsilon|^2\right)^{1/2}.
\endaligned
\end{equation*}
To estimate $I_{12}$, by using Lemma \ref{lemma-S-1}, Lemma \ref{7.47} and \eqref{regularity}, we see that
\begin{equation*}
\aligned
I_{12}\leq & C\delta^{-1}\Theta_\sigma(S)\left\{\int_{Q_1\setminus Q_{1-3\delta}}|\nabla u_0|^2+\delta^2\int_{Q_{1-2\delta}}|\nabla^2 u_0|^2\right\}^{1/2}\left(\int_{Q_1}|\nabla w_\varepsilon|^2\right)^{1/2}\\
&\leq C\delta^{-1}\Theta_\sigma(S)\left\{\delta^{\frac{1}{2}-\frac{1}{q}}\left(\dashint_{Q_1}|\nabla u_0|^q\right)^{1/q}+\delta\left(\dashint_{Q_1}|F|^2\right)^{1/2}\right\}
\left(\int_{Q_1}|\nabla w_\varepsilon|^2\right)^{1/2}.
\endaligned
\end{equation*}
Finally, by using Lemma \ref{lemma-S-3} and \eqref{regularity}, we have
\begin{equation*}
\aligned
I_{11}
&\leq C\varepsilon\sup_{(x,t)\in\mathbb{R}^{d+1}}\left(\dashint_{Q_1(x,t)}
|\nabla\phi_S|^2\right)^{1/2}\left(\int_{Q_1}|\nabla w_\varepsilon|^2\right)^{1/2}\\
&~~\cdot\left\{
\delta^{-1}\left(\int_{Q_1\setminus Q_{1-3\delta}}|\nabla u_0|^2\right)^{1/2}+\left(\int_{Q_{1-2\delta}}|\nabla^2 u_0|^2\right)^{1/2}\right\}\\
&\leq C\varepsilon\delta^{-1}\left\{\delta^{\frac{1}{2}-\frac{1}{q}}\left(\dashint_{Q_1}|\nabla u_0|^q\right)^{1/q}+\delta\left(\dashint_{Q_1}|F|^2\right)^{1/2}\right\}
\left(\int_{Q_1}|\nabla w_\varepsilon|^2\right)^{1/2}.
\endaligned
\end{equation*}
Above all, we obtain that
\begin{equation*}
\aligned
\|\nabla w_\varepsilon\|_{L^2(Q_1)}\leq &C\{1+\delta^{-1}\{\Theta_\sigma(S)+\Theta_1(S)+
\langle|\nabla\chi_S-\nabla\chi|\rangle+\varepsilon\}\}\\
&\left\{\delta^{\frac{1}{2}-\frac{1}{q}}\left(\dashint_{Q_1}|\nabla u_0|^q\right)^{1/q}+\delta\left(\dashint_{Q_1}|F|^2\right)^{1/2}\right\}.
\endaligned
\end{equation*}
By choosing $\delta=\Theta_\sigma(S)+\Theta_1(S)+
\langle|\nabla\chi_S-\nabla\chi|\rangle+\varepsilon$ and $\gamma=\frac{1}{2}-\frac{1}{q}$, this gives
\begin{equation*}
\aligned
\|\nabla w_\varepsilon\|_{L^2(Q_1)}\leq C\delta^{\gamma}\left\{\left(\dashint_{Q_1}|\nabla u_0|^q\right)^{1/q}+\left(\dashint_{Q_1}|F|^2\right)^{1/2}\right\}.
\endaligned
\end{equation*}
Since $w_\varepsilon=0$ on $\partial_pQ_1$, by using Poincare's inequality, we obtain
\begin{equation}\label{est-w-8}
\aligned
\|w_\varepsilon\|_{L^2(Q_1)}\leq C\delta^{\gamma}\left\{\left(\dashint_{Q_1}|\nabla u_0|^q\right)^{1/q}+\left(\dashint_{Q_1}|F|^2\right)^{1/2}\right\}.
\endaligned
\end{equation}

It remains to estimate $\|w_\varepsilon-(u_\varepsilon-u_0)\|_{L^2(Q_1)}$. In view of Lemma \ref{6.42}, Lemma \ref{7.47} and Lemma \ref{lemma-S-1}, we see that
\begin{equation}\label{est-remain-term}
\aligned
\|w_\varepsilon-(u_\varepsilon-u_0)\|_{L^2(Q_1)}
&\leq\varepsilon\left\| \chi_{S,j}^\varepsilon K_\varepsilon\left(\frac{\partial u_0}{\partial x_j}\right)\right\|_{L^2(Q_1)}
+\varepsilon^2\left\| \phi_{S,(d+1)ij}^\varepsilon\frac{\partial}{\partial x_i}K_\varepsilon\left(\frac{\partial u_0}{\partial x_j}\right)\right\|_{L^2(Q_1)}
\\&\leq C\Theta_\sigma(S)\left(\dashint_{Q_1}|\nabla u_0|^2\right)^{1/2}.
\endaligned
\end{equation}

  Finally, by the Meyers-type estimates for parabolic systems \cite{ABM-2018},
\begin{equation*}
\left(\dashint_{Q_1}|\nabla u_\varepsilon|^q\right)^{1/q}\leq C\left\{\left(\dashint_{Q_{3/2}}|\nabla u_\varepsilon|^2\right)^{1/2}
+\left(\dashint_{Q_2}|F|^2\right)^{1/2}\right\},
\end{equation*}
for some $q>2$ and $C>0$, depending on $d$ and $\mu$. This, together with \eqref{approx-1}, gives
\begin{equation}\label{Meyer}
\left(\dashint_{Q_1}|\nabla u_0|^q\right)^{1/q}\leq C\left\{\left(\dashint_{Q_{3/2}}|\nabla u_\varepsilon|^2\right)^{1/2}
+\left(\dashint_{Q_2}|F|^2\right)^{1/2}\right\}
\end{equation}
for some $q>2$ and $C>0$, depending on $d$, $q$, and $\mu$.
Combining \eqref{est-w-8} with \eqref{est-remain-term} and H\"older's inequality, and then using \eqref{Meyer}, Caccioppoli's inequality as well as $\Theta_\sigma(S)\leq C_\sigma[\Theta_1(S)]^\sigma$, we obtain \eqref{est-approximation}.
\end{proof}

\begin{lemma}\label{lem-standard}
Let $H(r)$ and $h(r)$ be two nonnegative and continuous functions on the interval $[0,1]$. Let $0<\delta<\frac{1}{4}$. Suppose that there exists a constant $C_0$ such that
\begin{equation}\label{cond-1}
\max_{r\leq t\leq 2r}H(t)\leq C_0 H(2r)\quad {\rm and}\quad
\max_{r\leq t,s\leq 2r}|h(t)-h(s)|\leq C_0 H(2r)
\end{equation}
for any $r\in [\delta,1/2]$. Suppose further that
\begin{equation}\label{cond-2}
H(\theta r)\leq \frac{1}{2}H(r)+C_0\omega
\left(\frac{\delta}{r}\right)\left\{H(2r)+h(2r)\right\}
\end{equation}
for any $r\in [\delta,1/2]$, where $\theta\in (0,1/4)$ and $\omega(t)$ is a nonnegative and nondecreasing function on $[0,1]$ such that $\omega(0)=0$ and \begin{equation}\label{cond-3}
\int_0^1\frac{\omega(t)}{t}\,dt<\infty.
\end{equation}
Then
\begin{equation*}
\max_{\delta\leq r\leq 1}\left\{H(r)+h(r)\right\}\leq C\left\{H(1)+h(1)\right\},
\end{equation*}
where $C$ depends only on $C_0$, $\theta$, and the function $\omega(t)$.
\end{lemma}

\begin{proof}
See \cite{Shen-book}.
\end{proof}

\begin{theorem}
Let $S\geq 1$. Then for any $\sigma\in(0,1)$,
\begin{equation}\label{est-grad-correct}
\langle|\nabla\chi_S-\nabla\chi|^2\rangle^{1/2}\leq C_\sigma\int_{S/2}^{\infty}\frac{\Theta_\sigma(r)}{r}\,dr,
\end{equation}
where $C_\sigma$ depends only on $\sigma$ and $A$.
\end{theorem}
\begin{proof}
Fix $1\leq j\leq d$ and $1\leq \beta\leq m$. Let $u=\chi_{S,j}^\beta$, $v=\chi_{2S,j}^\beta$, and $w=u-v$. It follows from Lemma \ref{Lem-6.4} that
\begin{equation}\label{error}
-\langle \partial_t \varphi\cdot w\rangle+\langle A\nabla w\cdot\nabla\varphi\rangle=\frac{1}{4S^2}\langle v\cdot\varphi\rangle
-\frac{1}{S^2}\langle u\cdot\varphi\rangle
\end{equation}
 for any $\varphi=(\varphi^\alpha)\in H^1_{{\rm loc}}(\mathbb{R}^{d+1})$ such that $\varphi,\nabla\varphi, \partial_t \varphi\in B^2(\mathbb{R}^{d+1})$.

 Choosing a sequence of trigonometric polynomials $\varphi_{n}\in {\rm \tilde{T}rig}(\R^{d+1})$ such that
 \begin{equation}\label{est-9.2}
  \|\varphi_{n}-w\|_{B^2}+ \|\nabla \varphi_{n}-\nabla w\|_{B^2}+\|\partial_t \varphi_{n}-\partial_t w\|_{\mathcal{K}^*+B^2}\rightarrow 0
  \end{equation}
as $n\rightarrow \infty$. Since $\varphi_n\in {\rm \tilde{T}rig}(\R^{d+1})$,
\begin{equation*}
\langle \partial_t \varphi_n\cdot w\rangle
=(\partial_t \varphi_n, \varphi_n)
+\langle \partial_t \varphi_n\cdot (w-\varphi_n)\rangle.
\end{equation*}
 In view of \eqref{skew-symmetric}, we have $(\partial_t\varphi_n,\varphi_n)=0$. Moreover, by a similar argument as in \eqref{est-u,S,n-3} and using \eqref{est-9.2}, we see that
\begin{equation*}
|\langle \partial_t \varphi_n\cdot (w-\varphi_n)\rangle|\leq \|\partial_t\varphi_n\|_{\mathcal{K}^*+B^2}\{\|w-\varphi_n\|_{B^2}
+\|\nabla w-\nabla \varphi_n\|_{B^2}\}\to 0 \quad \text{as}~~n\to \infty.
\end{equation*}
As a result,
\begin{equation*}
\langle \partial_t \varphi_n\cdot w\rangle\to 0\quad{\rm as}~~n\to \infty.
\end{equation*}
Hence by taking $\varphi=\varphi_n$ as test function in \eqref{error} and letting $n\to \infty$, we obtain
\begin{equation*}
\langle A\nabla w\cdot\nabla w\rangle=\frac{1}{4S^2}\langle v\cdot w\rangle
-\frac{1}{S^2}\langle u\cdot w\rangle.
\end{equation*}
This, together with Cauchy's inequality and \eqref{est-6.42}, gives
\begin{equation}\label{est-9.3}
\langle |\nabla w|^2\rangle^{1/2}\leq CS^{-1}\{\langle |v|^2 \rangle^{1/2}+\langle |u|^2 \rangle^{1/2}\}\leq C\{\Theta_\sigma(S)+\Theta_\sigma(2S)\}.
\end{equation}
Since $\Theta_\sigma(S)$ is decreasing, we see that
\begin{equation}\label{est-9.4}
\Theta_\sigma(S)+\Theta_\sigma(2S)\leq C_\sigma\int_{S/2}^S\frac{\Theta_\sigma(r)}{r}\,dr.
\end{equation}
In view of \eqref{est-9.3} and \eqref{est-9.4}, we obtain
\begin{equation*}
\langle |\nabla \chi_S-\nabla\chi_{2S}|^2\rangle^{1/2}\leq C_\sigma\int_{S/2}^S\frac{\Theta_\sigma(r)}{r}\,dr.
\end{equation*}
It follows that
\begin{equation*}
\sum_{k=0}^{\infty}\langle |\nabla \chi_{2^kS}-\nabla\chi_{2^{k+1}S}|^2\rangle^{1/2}\leq C_\sigma\int_{S/2}^\infty\frac{\Theta_\sigma(r)}{r}\,dr.
\end{equation*}
Since $\langle|\nabla \chi_S-\nabla\chi|^2\rangle\to 0$ as $S\to \infty$, this gives \eqref{est-grad-correct}.
\end{proof}

\begin{lemma}\label{Lemma-decay}
Suppose that there exist $C_0>0$ and $N> 1+\frac{1}{\gamma}$ such that
\begin{equation*}
\rho(R)\leq C[\log R]^{-N}\quad{\rm for~any}~R\geq 2,
\end{equation*}
where $\rho(R)$ is defined by \eqref{AP-1}. Then there exists $\sigma\in(0,1)$, such that
\begin{equation*}
\int_0^1\frac{[\eta(t)]^\gamma}{t}\,dt<\infty,
\end{equation*}
where $\eta(t)$ is defined by \eqref{def-eta-1}.
\end{lemma}
\begin{proof}
By the definition of $\Theta_\sigma(S)$, we see that
\begin{equation*}
\Theta_\sigma(S)\leq \rho(\sqrt{S})+\left(\frac{1}{\sqrt{S}}
\right)^\sigma\leq C_\sigma[\log S]^{-N}
\end{equation*}
for $S\geq 2$. It follows from the above estimate and \eqref{est-grad-correct} that
\begin{equation*}
\aligned
\eta(t)&\leq t+[\Theta_1(t^{-1})]^\sigma+\sup_{S\geq t^{-1}}\langle|\nabla\chi_S-\nabla \chi|\rangle\\
&\leq t+[\Theta_1(t^{-1})]^\sigma+C_\sigma\int_{(2t)^{-1}}^{\infty}
\frac{\Theta_\sigma(r)}{r}\,dr\\
&\leq t+C_\sigma[\log(1/t)]^{1-N}
\endaligned
\end{equation*}
for $t\in(0,\frac{1}{2})$. Since $N> 1+\frac{1}{\gamma}$, we see that
\begin{equation*}
\int_0^1\frac{\left[\eta(t)\right]^\gamma}{t}\,dt\leq C+C\int_{0}^{1/2}\frac{\left[\log(1/t)\right]^{\gamma(1-N)}}{t}\,dt<\infty.
\end{equation*}
This completes the proof.
\end{proof}

\begin{lemma}\label{lem-psi-u0-in}
Suppose that $u_0$ is a weak solution of $(\partial_t+\mathcal{L}_0)u_0=F$ in $Q_r$, where $0<r\leq 1$ and $F\in L^p(Q_r)$ for some $p>d+2$. Then there exists $\theta\in(0,1/4)$, depending only on $d$, $\mu$, $m$ and $p$, such that
\begin{equation}\label{est-u0}
\aligned
\frac{1}{\theta r}\inf_{\substack{E\in\mathbb{R}^{m\times d}\\ \beta\in\mathbb{R}^m}}&\left(\dashint_{Q_{\theta r}}|u_0-E\cdot x-\beta|^2\right)^{1/2}
+\theta r\left(\dashint_{Q_{\theta r}}|F|^p\right)^{1/p}\\
&\leq\frac{1}{2}\left\{\frac{1}{r}\inf_{\substack{E\in\mathbb{R}^{m\times d}\\ \beta\in\mathbb{R}^m}}\left(\dashint_{Q_r}|u_0-E\cdot x-\beta|^2\right)^{1/2}
+r\left(\dashint_{Q_r}|F|^p\right)^{1/p}\right\}.
\endaligned
\end{equation}
\end{lemma}
\begin{proof}
By translation and dilation, we may assume that $Q_r=Q_1(0,0)$. Let $E_0=\nabla u_0(0,0)$ and $\beta_0=u_0(0,0)$. Then by the classical $C^{1+\alpha_p}$ estimates for $\partial_t+\mathcal{L}_0$,
\begin{equation*}
\aligned
|u_0(x,t)-E_0\cdot x-\beta_0|\leq C(|x|+|t|^{1/2})^{1+\alpha_p}
\left\{\left(\dashint_{Q_1}|u_0|^2\right)^{1/2}
+\left(\dashint_{Q_1}|F|^p\right)^{1/p}\right\}
\endaligned
\end{equation*}
for any $(x,t)\in Q_{1/2}$, where $\alpha_p=1-\frac{d+2}{p}$. It follows that
\begin{equation}\label{est-u0-1}
\aligned
\frac{1}{\theta }\inf_{\substack{E\in\mathbb{R}^{m\times d}\\ \beta\in\mathbb{R}^m}}&\left(\dashint_{Q_{\theta }}|u_0-E\cdot x-\beta|^2\right)^{1/2}
+\theta \left(\dashint_{Q_{\theta }}|F|^p\right)^{1/p}\\
&\leq C\theta^{\alpha_p}
\left\{\left(\dashint_{Q_1}|u_0|^2\right)^{1/2}
+\left(\dashint_{Q_1}|F|^p\right)^{1/p}\right\}.
\endaligned
\end{equation}
Since $(\partial_t+\mathcal{L}_0) (E\cdot x+\beta)=0$ in $Q_1$ for any $E\in\mathbb{R}^{m\times d}$ and $\beta\in\mathbb{R}^m$, by choosing $\theta$ so small that $C\theta^{\alpha_p}<\frac{1}{2}$ and replace $u_0$ in \eqref{est-u0-1} by $u_0-E\cdot x-\beta$, we obtain \eqref{est-u0}.
\end{proof}

\begin{lemma}\label{lem-8.6}
Suppose that $u_\varepsilon$ is a weak solution of $(\partial_t+\mathcal{L}_\varepsilon)u_\varepsilon=F$ in $Q_2$, where $F\in L^p(Q_2)$ for some $p>d+2$. Let $\theta\in(0,1/4)$ be given by Lemma \ref{lem-psi-u0-in}. Then for any $\varepsilon\leq r\leq1$,
\begin{equation}\label{est-ue-in}
\aligned
\frac{1}{\theta r}\inf_{\substack{E\in\mathbb{R}^{m\times d}\\ \beta\in\mathbb{R}^m}}&\left(\dashint_{Q_{\theta r}}|u_\varepsilon-E\cdot x-\beta|^2\right)^{1/2}+\theta r\left(\dashint_{Q_{\theta r}}|F|^p\right)^{1/p}\\
\leq&\frac{1}{2}\left\{\frac{1}{r}\inf_{\substack{E\in\mathbb{R}^{m\times d}\\ \beta\in\mathbb{R}^m}}\left(\dashint_{Q_r}|u_\varepsilon-E\cdot x-\beta|^2\right)^{1/2}
+r\left(\dashint_{Q_r}|F|^p\right)^{1/p}\right\}\\
&+C\left[\eta\left(\frac{\varepsilon}{r}\right)\right]^\gamma
\left\{\frac{1}{r}\left(\dashint_{Q_{2r}}|u_\varepsilon|^2\right)^{1/2}
+r\left(\dashint_{Q_{2r}}|F|^p\right)^{1/p}\right\},
\endaligned
\end{equation}
where $\gamma$ and $\eta(t)$ are given by Lemma \ref{lem-approx} and $C$ depends only on $d$, $\mu$, $m$ and $p$.
\end{lemma}
\begin{proof}
Fix $\varepsilon\leq r\leq 1$. Let $u_0$ be the weak solution of $(\partial_t+\mathcal{L}_0)u_0=F$ in $Q_r$, given by Lemma \ref{lem-approx}. Note that by Lemma \ref{lem-psi-u0-in},
\begin{equation*}
\aligned
&\frac{1}{\theta r}\inf_{\substack{E\in\mathbb{R}^{m\times d}\\ \beta\in\mathbb{R}^m}}\left(\dashint_{Q_{\theta r}}|u_\varepsilon-E\cdot x-\beta|^2\right)^{1/2}+\theta r\left(\dashint_{Q_{\theta r}}|F|^p\right)^{1/p}\\
&\leq\frac{1}{\theta r}\inf_{\substack{E\in\mathbb{R}^{m\times d}\\ \beta\in\mathbb{R}^m}}\left(\dashint_{Q_{\theta r}}|u_0-E\cdot x-\beta|^2\right)^{1/2}
+\theta r\left(\dashint_{Q_{\theta r}}|F|^p\right)^{1/p}
+\frac{1}{\theta r}\left(\dashint_{Q_{\theta r}}
|u_\varepsilon-u_0|^2\right)^{1/2}\\
&\leq\frac{1}{2}
\left\{\frac{1}{r}\inf_{\substack{E\in\mathbb{R}^{m\times d}\\ \beta\in\mathbb{R}^m}}\left(\dashint_{Q_{ r}}|u_0-E\cdot x-\beta|^2\right)^{1/2}
+r\left(\dashint_{Q_r}|F|^p\right)^{1/p}\right\}
+\frac{1}{\theta r}\left(\dashint_{Q_{\theta r}}
|u_\varepsilon-u_0|^2\right)^{1/2}\\
&\leq\frac{1}{2}\left\{\frac{1}{r}\inf_{\substack{E\in\mathbb{R}^{m\times d}\\ \beta\in\mathbb{R}^m}}\left(\dashint_{Q_r}|u_\varepsilon-E\cdot x-\beta|^2\right)^{1/2}
+r\left(\dashint_{Q_r}|F|^p\right)^{1/p}\right\}
+\frac{C_\theta}{r}\left(\dashint_{Q_{r}}
|u_\varepsilon-u_0|^2\right)^{1/2}\\
&\leq\frac{1}{2}\left\{\frac{1}{r}\inf_{\substack{E\in\mathbb{R}^{m\times d}\\ \beta\in\mathbb{R}^m}}\left(\dashint_{Q_r}|u_\varepsilon-E\cdot x-\beta|^2\right)^{1/2}
+r\left(\dashint_{Q_r}|F|^p\right)^{1/p}\right\}\\
&\quad+C\left[\eta\left(\frac{\varepsilon}{r}\right)\right]^\gamma
\left\{\frac{1}{r}\left(\dashint_{Q_{2r}}|u_\varepsilon|^2\right)^{1/2}
+r\left(\dashint_{Q_{2r}}|F|^p\right)^{1/p}\right\},
\endaligned
\end{equation*}
where we have used Lemma \ref{lem-approx} for the last step. This gives \eqref{est-ue-in}.
\end{proof}

We are now in a position to give the proof of Theorem \ref{thm-interior-Lip}.

\begin{proof}[Proof of Theorem \ref{thm-interior-Lip}]
By translation and dilation we may assume that $R=2$ and $(x_0,t_0)=(0,0)$.  We prove the Theorem by applying Lemma \ref{lem-standard}. Let $h(t)=|E_t|$, where $E_t\in \mathbb{R}^{m\times d}$ is a matrix such that
\begin{equation*}
\frac{1}{t}\inf_{\substack{E\in\mathbb{R}^{m\times d}\\ \beta\in\mathbb{R}^m}}\left(\dashint_{Q_t}|u_\varepsilon-E\cdot x-\beta|^2\right)^{1/2}=\frac{1}{t}\inf_{\beta\in\mathbb{R}^m}
\left(\dashint_{Q_t}|u_\varepsilon-E_t\cdot x-\beta|^2\right)^{1/2},
\end{equation*}
and
\begin{equation*}
H(t)=\frac{1}{t}\inf_{\substack{E\in\mathbb{R}^{m\times d}\\ \beta\in\mathbb{R}^m}}\left(\dashint_{Q_t}|u_\varepsilon-E\cdot x-\beta|^2\right)^{1/2}+t\left(\dashint_{Q_t}|F|^p\right)^{1/p}.
\end{equation*}
In view of \eqref{est-ue-in}, we have
\begin{equation*}
H(\theta r)\leq \frac{1}{2}H(r)+C_0\left[\eta
\left(\frac{\varepsilon}{r}\right)\right]^\gamma\left\{H(2r)+h(2r)\right\}
\end{equation*}
for $\varepsilon\leq r\leq 1$. This gives \eqref{cond-2} with $\omega(t)=[\eta(t)]^\gamma$. Since $A$ satisfies the decay assumption \eqref{decay-cond}, by Lemma \ref{Lemma-decay}, we know that $\omega(t)$ satisfies \eqref{cond-3}.

It is easy to see that $H(r)$ satisfies the first inequality in \eqref{cond-1}. To verify the second, observe that for any $r\leq t,s\leq 2r$,
\begin{equation}\label{sec-cond}
\aligned
|h(t)-h(s)|&\leq |E_t-E_s|
\\&\leq\frac{C}{r}\inf_{\beta\in\mathbb{R}^m}
\left(\dashint_{Q_r}|(E_t-E_s)\cdot x-\beta|^2\right)^{1/2}
\\&\leq\frac{C}{r}\inf_{\beta\in\mathbb{R}^m}
\left(\dashint_{Q_t}|u_\varepsilon-E_t\cdot x-\beta|^2\right)^{1/2}+
\frac{C}{r}\inf_{\beta\in\mathbb{R}^m}
\left(\dashint_{Q_s}|u_\varepsilon-E_s\cdot x-\beta|^2\right)^{1/2}\\
&\leq C\{H(t)+H(s)\}
\\&\leq CH(2r),
\endaligned
\end{equation}
where $C$ depends only on $d$. As a result, by Lemma \ref{lem-standard}, we obtain
\begin{equation}\label{8.20}
\aligned
\frac{1}{r}\inf_{\beta\in\mathbb{R}^m}
\left(\dashint_{Q_r}|u_\varepsilon-\beta|^2\right)^{1/2}
&\leq H(r)+h(r)\\
&\leq C\{H(1)+h(1)\}\\
&\leq C\left\{\left(\dashint_{Q_1}|u_\varepsilon|^2\right)^{1/2}
+\left(\dashint_{Q_1}|F|^p\right)^{1/p}\right\}.
\endaligned
\end{equation}
Combining \eqref{8.20} with Caccioppoli's inequality and then using the fact that $(\partial_t+\mathcal{L}_\varepsilon)\beta=0$ for any $\beta\in\mathbb{R}^m$, we see that
\begin{align*}
\left(\dashint_{Q_{r/2}}|\nabla u_\varepsilon|^2\right)^{1/2}
&\leq C\left\{\inf_{\beta\in\mathbb{R}^m}\left(\dashint_{Q_1}|u_\varepsilon-\beta|^2\right)^{1/2}
+\left(\dashint_{Q_1}|F|^p\right)^{1/p}\right\}.
\end{align*}
This, together with Poincar\'e's inequality, gives \eqref{est-Lip}.
\end{proof}

\section{Boundary Lipschitz estimates}
This section is devoted to the large-scale boundary Lipschitz estimates for $\partial_t+\mathcal{L}_\varepsilon$. To prove Theorem \ref{thm-bound-Lip}, we localize the boundary of $\Omega$. Let $\psi: \mathbb{R}^{d-1}\to\mathbb{R}$ be a $C^{1,\alpha}$ function such that $\psi(0)=0$ and $\|\psi\|_{C^{1,\alpha}(\mathbb{R}^{d-1})}\leq M$. Define
\begin{equation*}
\aligned
&T_r=\{(x^\prime,x_d):|x^\prime|<r~~{\rm and}~~\psi(x^\prime)<x_d<100\sqrt{d}(M+1)\}\times(-r^2,0),\\
&I_r=\{(x^\prime,\psi(x^\prime)):|x^\prime|<r\}\times(-r^2,0),
\endaligned
\end{equation*}
where $0<r<\infty$.

We begin with an approximation result, which is a boundary version of Lemma \ref{lem-approx}.
\begin{lemma}\label{lem-approx-bound}
Suppose that $A$ satisfies \eqref{ellipticity} and is uniformly almost periodic. Let $u_\varepsilon$ be a weak solution of
\begin{equation*}
(\partial_t+\mathcal{L}_\varepsilon)u_\varepsilon=F\quad{\rm in}~~T_{2r}\quad{\rm and}\quad u_\varepsilon=f\quad{\rm on}~~I_{2r}
\end{equation*}
for some $0<r\leq 1$, where $F\in L^2(T_{2r})$ and $f\in C^{1+\alpha}(I_{2r})$. Then there exists a weak solution of
\begin{equation*}
(\partial_t+\mathcal{L}_0)u_0=F\quad{\rm in}~~T_{r}\quad{\rm and}\quad u_0=f\quad{\rm on}~~I_{r},
\end{equation*}
such that
\begin{equation}\label{est-approx-boundary}
\left(\dashint_{T_r}|u_\varepsilon-u_0|^2\right)^{1/2}\leq
C\left[\eta\left(\frac{\varepsilon}{r}\right)\right]^\gamma\left\{
\left(\dashint_{T_{2r}}| u_\varepsilon|^2\right)^{1/2}
+r^2\left(\dashint_{T_{2r}}|F|^2\right)^{1/2}
+\|f\|_{C^{1+\alpha}(I_{2r})}\right\},
\end{equation}
where $\eta(t)$ is given by \eqref{def-eta-1}, $\gamma=\frac{1}{2}-\frac{1}{q}$ for some $q>2$ and $C$ depends only on $d,\mu,\alpha$ and $\Omega$.
\end{lemma}
\begin{proof}
By rescaling we may assume that $r=1$. Let $u_0$ be the weak solution of
\begin{equation*}
(\partial_t+\mathcal{L}_0)u_0=F\quad{\rm in}~~T_{1}\quad{\rm and}\quad u_0=u_\varepsilon\quad{\rm on}~~\partial_pT_{1}.
\end{equation*}
Define $w_\varepsilon$ as in the proof of Lemma \ref{lem-approx}, then by a similar argument as in the proof of Lemma \ref{lem-approx}, we obtain
\begin{equation}\label{Meyer-bound-1}
\left(\dashint_{T_1}|\nabla w_\varepsilon|^2\right)^{1/2}\leq C[\eta\left(\varepsilon\right)]^\gamma
\left\{\left(\dashint_{T_1}|\nabla u_0|^q\right)^{1/q}
+\left(\dashint_{T_1}|F|^2\right)^{1/2}\right\},
\end{equation}
where $\gamma=\frac{1}{2}-\frac{1}{q}>0$ and $q>2$.
By the Meyers-type estimates and Caccioppoli's inequality for $\partial_t+\mathcal{L}_\varepsilon$, we have
\begin{equation}\label{Meyer-bound}
\aligned
\left(\dashint_{T_1}|\nabla u_0|^q\right)^{1/q}&\leq C\left(\dashint_{T_1}|\nabla u_\varepsilon|^q\right)^{1/q}\\
&\leq C\left\{\left(\dashint_{T_2}|u_\varepsilon|^2\right)^{1/2}
+\left(\dashint_{T_2}|F|^2\right)^{1/2}
+\|f\|_{C^{1+\alpha}(I_2)}\right\},
\endaligned
\end{equation}
where $q>2$ and $C>0$ depend only on $d$, $\mu$, $\alpha$ and $M$.  Since $w_\varepsilon=0$ on $\partial_p T_1$, in view of Poincar\'e's inequality, \eqref{Meyer-bound-1} and \eqref{Meyer-bound}, we obtain
\begin{equation}\label{est-app-b-1}
\left(\dashint_{T_1}| w_\varepsilon|^2\right)^{1/2}\leq C[\eta\left(\varepsilon\right)]^\gamma
\left\{\left(\dashint_{T_2}| u_\varepsilon|^2\right)^{1/2}
+\left(\dashint_{T_2}|F|^2\right)^{1/2}+\|f\|_{C^{1+\alpha}(I_{2})}\right\}.
\end{equation}
Finally, by Lemma \ref{6.42}, Lemma \ref{7.47}, Lemma \ref{lemma-S-3} and a similar argument as in the proof of Lemma \ref{lem-approx}, we have
\begin{equation*}
\aligned
&\left(\dashint_{T_1}| \varepsilon \chi_{S}^\varepsilon
K_\varepsilon \left(\nabla u_0\right)
+\varepsilon^2 \phi_{S}^\varepsilon
\nabla^2K_\varepsilon \left(\nabla u_0\right)|^2\right)^{1/2}\\
&\leq C\Theta_\sigma(S)\left\{\left(\dashint_{T_1}|\nabla u_0|^2\right)^{1/2}
+ \left(\int_{T_{1-2\delta}}
|\nabla^2 u_0|^2\right)^{1/2}
+\delta^{-1}\left(\int_{T_1\setminus T_{1-3\delta}}
|\nabla u_0|^2\right)^{1/2}\right\}
\\
&\leq C\Theta_\sigma(S)\left\{\left(\dashint_{T_1}|\nabla u_0|^2\right)^{1/2}
+ \left(\int_{T_{1-2\delta}}|\nabla^2 u_0|^2\right)^{1/2}
+\delta^{-\frac{1}{2}-\frac{1}{q}}\left(\dashint_{T_1}
|\nabla u_0|^q\right)^{1/q}\right\}
\\
&\leq C[\eta\left(\varepsilon\right)]^\gamma
\left\{\left(\dashint_{T_2}| u_\varepsilon|^2\right)^{1/2}
+\left(\dashint_{T_2}|F|^2\right)^{1/2}+\|f\|_{C^{1+\alpha}(I_{2})}\right\},
\endaligned
\end{equation*}
where we have used \eqref{Meyer-bound} for the last step. This, together with \eqref{est-app-b-1}, gives \eqref{est-approx-boundary}.
\end{proof}

For a function $u$ in $T_r$, define
\begin{equation*}
\Psi(r;u)=\frac{1}{r}\inf_{\substack{E\in\mathbb{R}^{m\times d}\\ \beta\in\mathbb{R}^m}}\left\{\left(\dashint_{T_r}|u-E\cdot x-\beta|^2\right)^{1/2}+\|u-E\cdot x-\beta\|_{C^{1+\alpha}(I_r)}\right\}.
\end{equation*}

\begin{lemma}\label{lem-psi-u0}
Suppose that $u_0$ is a weak solution of $(\partial_t+\mathcal{L}_0)u_0=F$ in $T_r$, where $0<r\leq 1$ and $F\in L^p(T_r)$ for some $p>d+2$. Then there exists $\theta\in(0,1/4)$, depending only on $d$, $\mu$, $\alpha$, $p$ and $M$, such that
\begin{equation*}
\Psi(\theta r;u_0)+\theta r\left(\dashint_{T_{\theta r}}|F|^p\right)^{1/p}\leq\frac{1}{2}\left\{\Psi(r;u_0)
+r\left(\dashint_{T_r}|F|^p\right)^{1/p}\right\}.
\end{equation*}
\end{lemma}
\begin{proof}
See \cite[Lemma 6.3]{GS-2020-ARMA}.
\end{proof}

\begin{lemma}
Suppose that $u_\varepsilon$ is a weak solution of $(\partial_t+\mathcal{L}_\varepsilon)u_\varepsilon=F$ in $T_2$ and $u_\varepsilon=f$ on $I_2$, where $F\in L^p(T_2)$ for some $p>d+2$. Let $\theta\in(0,1/4)$ be given by Lemma \ref{lem-psi-u0}. Then for any $\varepsilon\leq r\leq1$,
\begin{equation}\label{est-ue}
\aligned
\Psi(\theta r;u_\varepsilon)&+\theta r\left(\dashint_{T_{\theta r}}|F|^p\right)^{1/p}\\
\leq&\frac{1}{2}\left\{\Psi(r;u_\varepsilon)
+r\left(\dashint_{T_r}|F|^p\right)^{1/p}\right\}\\
&+C\left[\eta\left(\frac{\varepsilon}{r}\right)\right]^\gamma
\left\{\frac{1}{r}\left(\dashint_{T_{2r}}|u_\varepsilon|^2\right)^{1/2}
+r\left(\dashint_{T_{2r}}|F|^p\right)^{1/p}+r^{-1}\|f\|_{C^{1+\alpha}(I_{2r})}\right\},
\endaligned
\end{equation}
where $C$ depends only on $d$, $\mu$, $\alpha$, $p$ and $M$.
\end{lemma}
\begin{proof}
With Lemma \ref{lem-approx-bound} and Lemma \ref{lem-psi-u0} at our disposal, the proof follows the same line of argument used for Lemma \ref{lem-8.6}. We omit the details.
\end{proof}

We are now ready to give the proof of Theorem \ref{thm-bound-Lip}.

\begin{proof}[Proof of Theorem \ref{thm-bound-Lip}]
By translation and dilation we may assume that $R=2$ and $(x_0,t_0)=(0,0)$. Moreover, by a simple covering argument, it suffices to show that if $u_\varepsilon$ is a weak solution of $(\partial_t+\mathcal{L}_\varepsilon)u_\varepsilon=F$ in $T_2$ and $u_\varepsilon=f$ on $I_2$, then for $\varepsilon\leq r<2$,
\begin{equation}\label{local-est}
\aligned
\left(\dashint_{T_r}|\nabla u_\varepsilon|^2\right)^{1/2}\leq C\left\{\left(\dashint_{T_2}|\nabla u_\varepsilon|^2\right)^{1/2}+\left(\dashint_{T_2}|F|^p\right)^{1/p}
+\|f\|_{C^{1+\alpha}(I_2)}\right\}.
\endaligned
\end{equation}
To show \eqref{local-est}, we apply Lemma \ref{lem-standard} with $h(t)=|E_t|$, where $E_t\in \mathbb{R}^{m\times d}$ is a matrix such that
\begin{equation*}
\Psi(t;u_\varepsilon)=\frac{1}{t}\inf_{\beta\in\mathbb{R}^m}\left\{
\left(\dashint_{T_t}|u_\varepsilon-E_t\cdot x-\beta|^2\right)^{1/2}
+\|f-E_t\cdot x-\beta\|_{C^{1+\alpha}(I_t)}\right\},
\end{equation*}
and
\begin{equation*}
H(t)=\Psi(t;u_\varepsilon)+t\left(\dashint_{T_t}|F|^p\right)^{1/p}.
\end{equation*}
In view of \eqref{est-ue}, we have
\begin{equation*}
H(\theta r)\leq \frac{1}{2}H(r)+C_0\left[\eta
\left(\frac{\varepsilon}{r}\right)\right]^\gamma\left\{H(2r)+h(2r)\right\}
\end{equation*}
for $\varepsilon\leq r\leq 1$. This gives \eqref{cond-2} with $\omega(t)=[\eta(t)]^\gamma$. Since $A$ satisfies the decay assumption \eqref{decay-cond}, by Lemma \ref{Lemma-decay}, we know that $\omega(t)$ satisfies \eqref{cond-3}.

It is easy to see that $H(r)$ satisfies the first inequality in \eqref{cond-1}. By the same argument as in \eqref{sec-cond}, we see that $H(r)$ satisfies the second inequality in \eqref{cond-1}, with $C_0$ depends only on $d$, $\alpha$ and $M$. As a result, by Lemma \ref{lem-standard}, we obtain
\begin{align*}
\frac{1}{r}\inf_{\beta\in\mathbb{R}^m}
\left(\dashint_{T_r}|u_\varepsilon-\beta|^2\right)^{1/2}
&\leq H(r)+h(r)\\
&\leq C\{H(1)+h(1)\}\\
&\leq C\left\{\left(\dashint_{T_1}|u_\varepsilon|^2\right)^{1/2}
+\left(\dashint_{T_1}|F|^p\right)^{1/p}
+\|f\|_{C^{1+\alpha}(I_1)}\right\}.
\end{align*}
This, together with Caccioppoli's inequality and the fact that $(\partial_t+\mathcal{L}_\varepsilon)\beta=0$ for any $\beta\in\mathbb{R}^m$, implies
\begin{align*}
\left(\dashint_{T_{r/2}}|\nabla u_\varepsilon|^2\right)^{1/2}
&\leq C\left\{\inf_{\beta\in\mathbb{R}^m}\left(\dashint_{T_1}|u_\varepsilon-\beta|^2\right)^{1/2}
+\left(\dashint_{T_1}|F|^p\right)^{1/p}
+\|f\|_{C^{1+\alpha}(I_1)}\right\}.
\end{align*}
The desired estimate \eqref{local-est} follows readily from this and Poincar\'e's inequality.
\end{proof}

\section*{Statements and Declarations}
{\bf Conflict of interest} The authors declare that there is no conflict of interest.

\end{document}